\documentclass[a4,11pt]{article}
\usepackage[dviout]{pict2e}
\usepackage{graphicx}
\usepackage{lscape}
\usepackage{enumerate}
\usepackage{latexsym}
\usepackage{amsthm}
\usepackage{subfig}

\usepackage{color}

\usepackage{framed}

\newtheorem{theorem}{Theorem}

\newtheorem{lemma}{Lemma}

\def\vec#1{\mbox{\boldmath $#1$}}
\usepackage{amsmath,amssymb}

\DeclareMathOperator{\argmin}{argmin}

\begin{document}

\title{Nonparametric smoothing for extremal quantile regression with heavy tailed distributions}

\author{
{\sc Takuma Yoshida}$^{1}$\\
$^{1}${\it Kagoshima University, Kagoshima 890-8580, Japan}\\
{\it E-mail: yoshida@sci.kagoshima-u.ac.jp}
}

\date{\empty}
\maketitle
%===========================================================================================================-

\begin{abstract}
In several different fields, there is interest in analyzing the upper or lower tail quantile of the underlying distribution rather than mean or center quantile. 
However, the investigation of the tail quantile is difficult because of data sparsity. 
In this paper, we attempt to develop nonparametric quantile regression for the extremal quantile level. 
In extremal quantile regression, there are two types of technical conditions of the order of convergence of the quantile level: intermediate order or extreme order. 
For the intermediate order quantile, the ordinary nonparametric estimator is used. 
On the other hand, for the extreme order quantile, we provide a new estimator by extrapolating the intermediate order quantile estimator.
The performance of the estimator is guaranteed by asymptotic theory and extreme value theory. 
As a result, we show the asymptotic normality and the rate of convergence of the nonparametric quantile regression estimator for both intermediate and extreme order quantiles. 
A simulation is presented to confirm the behavior of the proposed estimator. 
The data application is also assessed. 
\end{abstract}


%===========================================================================================================-
\section{Introduction}

In a wide variety of areas, such as in the study of heavy rainfall, low birth weight, and high-risk finance, the tail behavior of the distribution of the target variable is of interest rather than the average or median.
In these cases, we often investigate the upper or lower quantile of the data.
However, the estimation of the tail quantile is difficult because of data sparsity. 
Therefore, the development of the mathematical properties of the tail quantile would be welcome. 
The theoretical performance of the tail behavior of the distribution function is provided by extreme value theory (EVT). 
The fundamental properties of EVT were surveyed by Coles (2001), Beirlant et al. (2004a), and de Haan and Ferreira (2006). 
On the other hand, the performance of the estimator is often guaranteed by a large sample or asymptotic theory in statistics. 
Thus, the mathematical properties of the estimator of the tail quantile are analyzed using a hybrid of EVT and asymptotic theory. 
In many cases, it is important to research the target variable along with the information of other variable. 
Then we should analyze the data in the literature of regression.
In this paper, we consider the estimation of the extremal conditional quantile of the response $Y$ given $X=x$.

Many researchers have developed the extremal conditional quantile estimation. 
Beirlant and Goegebeur (2004) developed a Pareto distribution approach.
Gardes et al. (2010) and Gardes and Girard (2010) studied the nearest-neighbor estimation. 
Daouia et al. (2011, 2013), El Methni et al. (2014), and Girard and Louhichi (2015) investigated the extremal quantile of the nonparametric estimator of the conditional distribution function of $Y$ given $X=x$. 
The local-moment-type methods were studied by Goegebeur et al. (2017). 
Durrieu et al. (2015) have developed the weighted quasi-log-likelihood method.
On the other hand, quantile regression, which was pioneered by Koenker and Bassett (1978), is a typical approach used to investigate the conditional quantile.  
For the center quantile, several authors have developed quantile regression methods. 
These fundamental developments have been summarized by Koenker (2005). 
However, much less work has been done on quantile regression for the extremal quantile. 
Chernozhukov (2005), Chernozhukov and Fern\'andez-val (2011), Wang et al. (2012), and He et al. (2016) studied extremal quantile regression, but they focused only on linear models. 
For the tail quantile, the linear structure assumption is strong and its assumption is violated in several cases. 
Therefore, a nonparametric approach should be used in extremal quantile regression in such situations. 
Beirlant et al. (2004b) studied extremal nonparametric quantile regression, but the theoretical property was not investigated.
In this paper, we develop nonparametric quantile regression for the extremal quantile and mathematical properties. 

Before we describe our study, we review extremal quantile regression with linear models in more detail.
For extremal quantile regression, the quantile level $\tau$ approaches 0 or 1 as the sample size $n$ increases. 
This paper treats only the upper quantile and, hence, $\tau\rightarrow 1$ as $n\rightarrow \infty$.
Thus, there are two important types of the order of $\tau$: the intermediate order quantile and the extreme order quantile. 
The former means that $\tau \rightarrow 1$ and $n(1-\tau)\rightarrow \infty$ as $n\rightarrow \infty$, whereas in the latter $\tau \rightarrow 1$ and $n(1-\tau)\rightarrow c\in [0,\infty)$ as $n\rightarrow \infty$. 
If $\tau$ is fixed, it is the so-called center quantile.
According to Chernozhukov (2005) and Chernozhukov and Fern\'andez-val (2011), the quantile regression estimator with linear models has asymptotic normality for the intermediate order quantile but it converges to a non-degenerated distribution (not normal) for the extreme order quantile. 
Thus, the extreme order quantile is difficult to handle.
Wang et al. (2012) provided a nice approach to obtain the extreme order quantile estimator by extrapolation from the intermediate order quantile estimator. 
As a result, this extrapolated estimator has asymptotic normality. 
It seems that above results should be extended to the nonparametric quantile regression for many applications.

In this paper, we first construct the ordinary nonparametric estimator for the intermediate order quantile. 
We then use the $B$-spline method with $\ell_2$ penalty. 
This approach was originally considered by O'Sullivan (1986) and Eilers and Marx (1996) in mean regression. 
Pratesi et al. (2009), Reiss and Huang (2012), and Yoshida (2013) used the quantile regression for only the center quantile. 
We show the asymptotic bias and variance as well as the asymptotic normality of the penalized spline estimator.
Next, the extrapolated estimator is obtained for the extreme order quantile. 
Similar to the approach of Wang et al. (2012), we use the Weissman-type extrapolation method (see Weissman 1978). 
The asymptotic normality and the optimal rate of convergence of the extreme order quantile estimator are shown. 
To the best of our knowledge, this is the first time that the rate of convergence of the nonparametric estimator is revealed in the extremal quantile regression.

This paper is organized as follows. 
In Section 2, we coordinate the conditions of the true conditional quantile by EVT in nonparametric extremal quantile regression. 
Section 3 presents the nonparametric estimator and its asymptotic property for both intermediate and extreme order quantiles.
In Section 4 the Monte Carlo simulation is conducted to confirm the performance of the proposed estimator. 
Section 5 addresses the application to Beijing's ${\rm PM}_{2.5}$ pollution data. 
The conclusions and future research are described in Section 6. 
In the Appendix, the computational aspects of the penalized spline estimator and the proofs of the mathematical results that appear in this paper are presented. 
Throughout the paper and without loss of generality, we focus on the conditional high quantile because a low quantile of the response can be viewed as a high quantile of the inverted sign of the response.

\section{Conditional extremal quantiles}

\subsection{Extreme value theory}

Let $\{(X_i,Y_i);\ i = 1,\ldots,n\}$ be the independent copies of a random pair $(X,Y)\in\mathbb{R}\times\mathbb{R}$. 
We assume that the support of $X$ is bounded on $[a,b]$, where $-\infty<a<b<\infty$.
The conditional distribution function of $Y$ given $X=x$ is denoted by $F_Y(y|x)=P(Y\leq y|X=x)$. 
Then the conditional $100\tau\%$ quantile of $Y$ given $X=x$ is 
$$
q_Y(\tau|x)\equiv F_Y^{-1}(\tau|x)=\inf\{t, F_Y(t|x)\geq  \tau\}.
$$ 
The main purpose of this study is to estimate $q_Y(\tau|x)$ for a high quantile level $\tau\simeq  1$. 
The tail behavior of the distribution or quantile function can be characterized by EVT.
To analyze the conditional high quantile of $Y$ given $X=x$, we introduce the EVT conditions of $F_Y(\cdot|x)$ and $q_Y(\cdot|x)$. 
We first provide the error $Z$ to incorporate the stochastic structure of $Y$ given $X=x$. 
Here, we assume that $Z$ is independent to predictor $X$. 
Define $F(z)$ and $q(\tau)$ as the marginal distribution function and $100\tau\%$ quantile of $Z$.
Throughout the paper, we assume that $F$ and $F(\cdot|x)$ belong to the maximum domain of attraction of an extreme value distribution $G_\gamma$, denoted by $F,\ F_Y(\cdot|x)\in D(G_\gamma)$. 
The distribution function $Q$ belongs to the maximum domain of attraction, which means that for the random sample $Z_1,\ldots,Z_n$ from $Q$, there exists a constant $\alpha_n>0$ and $\beta_n\in\mathbb{R}$ such that for $1+\gamma z\geq 0$, 
$$
P\left(\frac{\max_{1\leq i\leq n}Z_i-\beta_n}{\alpha_n}\leq z\right)\rightarrow G_\gamma(z)=\exp[-(1+\gamma z)^{-1/\gamma}]
$$
as $n\rightarrow \infty$.
Here, $\gamma\in\mathbb{R}$ is the extreme value index (EVI). 
The EVI is very important since this generally controls the tail behavior of the distribution function. 
For $Q\in D(G_\gamma)$, if $\gamma=0$ or $\gamma<0$, $Q$ has a light tail or short tail. 
When $\gamma>0$, $Q$ has a heavy tail.
This paper only discusses the heavy-tail case and, hence, we assume that $\gamma>0$ from now on. 
The maximum domain of attraction is a very weak condition. 
For example, uniform, beta, Gaussian, $t$, Pareto, Cauchy, and many other distributions belong to the maximum domain of attraction with appropriately specified $\gamma\in\mathbb{R}$. 
The details of the maximum domain of attraction and EVI are given in fundamental books such as that by de Haan and Ferreira (2006). 

We now state the conditions to connect the tail behavior of $F$ and $F(\cdot|x)$. 
For this, we need an additional definition. 
Let $RV(a)=\{A\in\mathbb{R}^+\rightarrow \mathbb{R}^+ |A(mt)/A(t)\rightarrow m^a {\rm as}\ t\rightarrow \infty, m>0\}$ be the set of regularly varying functions, where $\mathbb{R}^+=(0,\infty)$. 
When $A\in RV(0)$, $A$ is the so-called slowly varying function.

\medskip

\noindent{\bf Conditions A}
\begin{enumerate}
\item[A1.] For the error $Z\sim F$, there exists $L\in RV(0)$ such that the distribution function $F$ satisfies $1-F(z)=z^{-1/\gamma}L(z)\{1+o(1)\}$ as $z\rightarrow \infty$.
\item[A2.] We have $\tilde{q}(\tau)=\partial q(t)/\partial t|_{t=1-\tau}$ regularly varying at 0 with index $-\gamma-1$. 
That is, for $x>0$, 
$$
\lim_{\tau\rightarrow 0} \frac{\tilde{q}(x\tau)}{\tilde{q}(\tau)}=x^{-\gamma-1}.
$$
\item[A3.] For $Z\sim F$ and $Y|x \sim F_Y(\cdot|x)$, there exists an auxiliary function $f(x)$  such that $V\equiv Y-f(x)$ has the distribution function $F_V(y|x)$ satisfying, as $y\rightarrow \infty$,
\begin{eqnarray*}
1-F_V(y|x)=H(x)\{1- F(y)\}(1+o(1)),
\end{eqnarray*}
where $H(x)>0$ is a positive, continuous, and bounded function on $[a,b]$ and has $E[H(X)]=1$. 
\item[A4.] For $q_V(\cdot|x)=F^{-1}_V(\cdot|x)$, $\partial q_V(\tau|x)/\partial \tau= \partial q(1-(1-\tau)/H(x))/\partial \tau\{1+o(1)\}$ uniformly in $x\in[a,b]$ as $\tau\rightarrow 1$.
\end{enumerate}

\medskip

Conditions A may not hold if either $F$ or $F_Y(\cdot|x)$ is not included in $D(G_\gamma)$. 
In other words, if $F,\ F_Y(\cdot|x)\in D(G_\gamma)$, Conditions A are natural.
Condition A1 is the formal notation of a Pareto-type tail (see Chernozhukov and Fern\'andez-val 2011). 
The equivalent to condition A1 is 
\begin{eqnarray}
q(\tau)=(1-\tau)^{-\gamma}\bar{L}(1/(1-\tau))\{1+o(1)\} \quad {\rm as}\ \tau\rightarrow 1\label{equicon1}
\end{eqnarray}
with $\bar{L}\in RV(0)$.
Therefore, if the distribution $F$ is continuous and $\partial \bar{L}(1-\tau)/\partial\tau\rightarrow 0$ as $\tau\rightarrow 1$, condition A2 holds from condition A1. 
Actually, $\partial q(\tau)/\partial \tau\sim \partial\{(1-\tau)^{-\gamma}\bar{L}(1/(1-\tau))\}/\partial \tau=\gamma(1-\tau)^{-\gamma-1}\bar{L}(1/(1-\tau))+(1-\tau)^{-\gamma} \partial \bar{L}(1/(1-\tau))/\partial \tau=\gamma(1-\tau)^{-\gamma-1}\bar{L}(1/(1-\tau))\{1+o(1)\}$ as $\tau\rightarrow 1$. 
Therefore, we have $\tilde{q}(\tau)=\partial q(t)/\partial t|_{t=1-\tau}\sim \gamma\tau^{-\gamma-1}\bar{L}(1/\tau)\{1+o(1)\}$.
Thus, condition A2 is weak. 
Condition A3 provides the extremely location-scale shifted model having the Pareto-type conditional quantile tail of $Y$ given $X=x$ along with an auxiliary function $f(x)$. 
Actually, it it easy to show from A3 that $q_V(\tau|x)=H(x)^{\gamma}q(\tau)(1+o(1))$ as $\tau\rightarrow 1$. 
Furthermore, since $\tau=F_Y(q_Y(\tau|x)|x)=P(Y<q_Y(\tau|x)|x)=P(Y-f(x)<q_Y(\tau|x)-f(x)|x)=F_V(q_Y(\tau|x)-f(x)|x)$, we obtain $q_V(\tau|x)=q_Y(\tau|x)-f(x)$. 
Consequently, we have 
\begin{eqnarray}
q_Y(\tau|x)=f(x)+h(x)q(\tau)\{1+o(1)\}\quad {\rm as}\ \tau\rightarrow 1, \label{truequantile}
\end{eqnarray}
where $h(x)=H(x)^{\gamma}$.
Chernozhukov (2005), Chernozhukov and Fern\'andez-val (2011) and Wang et al. (2012) also provided this type of condition in multiple linear models.
That is, they further assumed that $f(x)=x^T\beta$ and $h(x)=x^T c$ for $x=(x_1,\cdots,x_p)$, where $\beta$ and $c$ are unknown $p$-dimensional parameter vectors. 
Thus, A3 is the nonparametric model version of the above previous studies.
Condition A4 guarantees the existence of a conditional quantile density function (the derivative of the quantile function). 
Furthermore, the conditional quantile density function also behaves like a Pareto-type function by condition A4. 
Assumption A3 is strengthened by condition A4.

\medskip

\noindent{\bf Remark 1.} Let $U(t)=q(1-1/t)=\inf\{z|F(z)\geq 1-1/t\}$ and let $U_V(t|x)=q_V(1-1/t|x)$. 
In several articles (see, for example, de Haan and Ferreira 2006), the conditions of EVT are applied to $U(t)$ and $U_V(t|x)$ as $t\rightarrow \infty$. 
Since $q(\tau)=U(1/(1-\tau))$, the condition (\ref{equicon1}) is similar to $U(t)=t^{\gamma}L(t)\{1+o(1)\}$ with $t=1/(1-\tau)$.
Condition A4 can also be expressed as $\partial U_V(t|x)/\partial t=\partial U(tH(x)) /\partial t$ with $t=1/(1-\tau)$. 
Thus, we can reconsider the EVT conditions for quantiles by using $U$ and $U_V$.
In particular, the use of $U$ is appropriate when using the second-order condition of EVT (see Section 3.2).

\subsection{$B$-spline model}

The conditional quantile $q_Y(\tau|x)$ can be written as 
\begin{eqnarray}
q_Y(\tau|x)= \argmin_{a} E[\rho_\tau(Y-a)|X=x], \label{quantmodel1}
\end{eqnarray}
where $\rho_\tau(u)=u(\tau-I(u<0))$ is Koenker's check function (Koenker 2005) and $I$ is the indicator function. 
The estimator of $q_Y$ is often obtained along with an empirical version of (\ref{quantmodel1}).
To estimate $q_Y(\tau|x)$, we use the $B$-spline regression method as the nonparametric technique in this paper.
Let $\{B_k(x):k=1,\ldots,K+p\}$ be the $p$th degree of the $B$-spline basis with knots $a=\kappa_0<\kappa_1<\ldots<\kappa_{K}=b$. 
In addition, other sets of $2p$ knots are defined as $\kappa_{-p}=\ldots=\kappa_{-1}=a$ and $\kappa_{K+1}=\ldots=\kappa_{K+p}=b$.
We then define the $B$-spline model as
$$
s(x)=\sum_{k=1}^{K+p} B_k(x)b_k=\vec{B}(x)^T\vec{b}, 
$$
where $\vec{B}(x)=(B_1(x),\ldots,B_{K+p}(x))^T$ and $\vec{b}=(b_1,\ldots,b_{K+p})^T$ is an unknown parameter vector. 
We now describe the relationship between the $B$-spline model and EVT discussed in the previous section. 
Let $W^m[a,b]=\{g|g^{(k)}\ {\rm is}\ {\rm continuous},\ k=0,1,\ldots,m-1\ {\rm and}\ \int_{a}^b \{g^{(m)}(x)\}^2dx<\infty\}$ be the $m$th-order Sobolev space.  
From Barrow and Smith (1978), for any function $g\in W^m([a,b])$, there exists $\vec{b}_g\in\mathbb{R}^{K+p}$ such that $g(x)-\vec{B}(x)^T\vec{b}_g=K^{-d}g^{(d)}(x)O(1)$ as $K\rightarrow \infty$, where $d=\min\{m,p+1\}$. 
For simplicity, we assume that $m\leq p+1$, that is $d=m$. 
Actually, $(p,m)=(3,2)$ is the standard condition of $B$-spline smoothing.

For $\tau\in(0,1)$, let
$$
\vec{b}_0(\tau)= \argmin_{\vec{b}\in\mathbb{R}^{K+p}} E[\rho_\tau(Y-s(x))|X=x]
$$
and let $s_0(\tau|x)=\vec{B}(x)^T\vec{b}_0(\tau)$. 
We then found that $q_Y(\tau|x)=s_0(\tau|x)+K^{-m} \{\partial^m q_Y(\tau|x)/\partial x^m\}O(1)$ for $q_Y(\cdot|x)\in W^m[a,b]$. 
If $h\in W^m[a,b]$, Conditions A and (\ref{truequantile}) yield that $\{\partial^m q_Y(\tau|x)/\partial x^m\}=h^{(m)}(x)q(\tau)(1+o(1))=O((1-\tau)^{-\gamma})$ and, hence, $s_0(\tau|x)-q_Y(\tau|x)=O(K^{-m}(1-\tau)^{-\gamma})$ as $\tau\rightarrow 1$ and $K\rightarrow \infty$, which indicates that the condition B4 below is required.
 
If $f$ and $h$ defined in (\ref{truequantile}) belong to $W^m[a,b]$, there exists $\vec{b}_f,\vec{b}_h\in\mathbb{R}^{K+p}$ such that $f(x)-\vec{B}(x)^T\vec{b}_f=O(K^{-m})$ and $h(x)-\vec{B}(x)^T\vec{b}_h=O(K^{-m})$. 
We then obtain $\vec{b}_0(\tau)= \vec{b}_f+q(\tau)\vec{b}_h+O(K^{-m}q(\tau))$ as $\tau\rightarrow 1$ and $K\rightarrow \infty$. 
Therefore, (\ref{equicon1}) and condition A4 indicate that $\partial s_0(\tau|x)/\partial \tau \sim \vec{B}(x)^T\vec{b}_h \partial q(\tau)/\partial \tau$ is satisfied since $\vec{b}_f$ and $\vec{b}_h$ are not dependent on $\tau$. 
Thus, the $B$-spline model also holds (\ref{truequantile}) and condition A4 and, hence, the tail behavior of the $B$-spline model can be studied by using Conditions A. 
The following conditions are the fundamental assumptions for $B$-spline regression. 

\medskip

\noindent{\bf Conditions B}
\begin{enumerate}
\item[B1.] For some constant $\nu\geq 0$, $E[|Y|^{2+\nu}|X=x]<\infty$.
\item[B2.] The functions $f$ and $h$ in (\ref{truequantile}) are included in $W^m[a,b]$. 
\item[B3.] We have $\max_{1\leq j\leq K}\{\kappa_{j+1}-\kappa_j\}=O(K^{-1})$.
\item[B4.] For some $\alpha\in(0,1)$, the number of knots $K=O(n^{\alpha})$. 
\item[B5.] As $\tau\rightarrow 1$ and $K\rightarrow \infty$, $K^m(1-\tau)^\gamma\rightarrow \infty$.
\end{enumerate}

\medskip

Condition B1 is needed to that the estimator satisfies the Lyapunov condition of central limit theorem.  
When condition B2 holds, the $B$-spline model can approximate to $q_Y(\tau|x)$.
Conditions B3 and B4 are standard conditions for $B$-spline models. 
Together with condition B2, the $B$-spline model and EVT are connected for high quantile level. 
Condition B5 guarantees that the model bias between the conditional quantile and $B$-spline model converges to 0.

\section{Penalized $B$-spline estimator for extremal quantiles} 

In this section, we define the nonparametric $B$-spline estimator and develop the asymptotic result. 
Then, we consider two scenario of extremal quantile rate: (i) intermediate order quantiles that $\tau\rightarrow 1$ and $(1-\tau)n\rightarrow \infty$ as $n\rightarrow \infty$ and (ii) extreme order quantiles that $\tau\rightarrow 1$ and $(1-\tau)n\rightarrow c \in [0,\infty)$ as $n\rightarrow \infty$. 
We denote the intermediate order quantile level by $\tau_I$ and the extreme order quantile level by $\tau_E$, respectively. 
That is, as $n\rightarrow \infty$, $\tau_I,\tau_E\rightarrow 1$, $n(1-\tau_I)\rightarrow \infty$, $n(1-\tau_E)\rightarrow c\in[0,\infty)$, and $(1-\tau_I)/(1-\tau_E)\rightarrow \infty$.

\subsection{Estimation of intermediate order quantiles}

The ordinary $B$-spline quantile estimator for $\tau\in(0,1)$ is defined based on minimizing
$\sum_{i=1}^n \rho_\tau(Y_i-s(x_i))$. 
However, it is known that the ordinary estimator tends to have a wiggly curve caused by data sparsity. 
To avoid this, we introduce the penalization method to control the behavior of the estimator.
Although various types of penalties have been developed to prevent overfitting, we will use O'Sullivan's (1986) penalty. 
For $\tau\in(0,1)$, the penalized spline estimator $\tilde{\vec{b}}(\tau)=(\tilde{b}_1(\tau),\ldots,\tilde{b}_{K+p}(\tau))^T$ of vector $\vec{b}(\tau)=(b_1(\tau),\ldots,b_{K+p}(\tau))^T$ is constructed by minimizing 
\begin{eqnarray}
\sum_{i=1}^n \rho_\tau(Y_i-s(x_i)) +\lambda \int_a^b \{s^{(m)}(x)\}^2dx, \label{nonparaquantile}
\end{eqnarray}
where $\lambda>0$ is the smoothing parameter. 
Using $\tilde{\vec{b}}(\tau)$, for the intermediate order quantile level $\tau_I$, we define 
$$
\tilde{q}_Y(\tau_I|x)=\vec{B}(x)^T\tilde{\vec{b}}(\tau_I). 
$$

We study the asymptotic theory for the penalized spline estimator $\tilde{q}_Y(\tau_I|x)$. 
Then, the conditions of the number of knots and the smoothing parameter included in $\tilde{q}_Y(\tau_I|x)$ are very important. 
The penalty $\int_a^b \{s^{(m)}(x)\}^2dx$ can be written as $\vec{b}^T D_m^T RD_m\vec{b}$, where the $(K+p)$-matrix $R$ has elements $R_{ij}=\int_a^b B_i(x)B_j(x)dx$ and the $(K+p-m)\times(K+p)$ matrix $D_m$ satisfies $\vec{b}^{(m)}=D_m\vec{b}$, where $\vec{b}^{(m)}=(b_1^{(m)},\ldots,b_{K+p-m}^{(b)})^T$, and for $m=1,2,\ldots,$
$$
b_j^{(1)}= p\frac{b_j-b_{j-1}}{\kappa_{j+p}-\kappa_j},\quad
b_j^{(m)}=(p+1-m)\frac{b_j^{(m-1)}-b_{j-1}^{(m-1)}}{\kappa_{j+p+1-m}-\kappa_j}.
$$
From now on, we use the symbols $D_m$ and $R$.
Let $G(h)$ be the $(K+p)$-matrix with elements $G_{ij}=\int_a^b h(x)B_i(x)B_j(x)dx$ and 
$$
\Lambda(h) =\Lambda(h,n,\tau_I)= \gamma^{-1}G(h^{-\gamma})+\frac{\lambda q(\tau_I)}{(1-\tau_I) n}D_m^TRD_m.
$$
Let $G=G(1)$ be $G(h)$ with $h(x)\equiv 1$. 
Define 
$$
K(m,\tau_I)= K\left(\frac{\lambda q(\tau_I)}{n(1-\tau_I)}\right)^{1/2m},
$$ 
which controls the asymptotic scenario branch discussed in Remark 1 below.

\noindent{\bf Conditions C}
\begin{enumerate}
\item[C1.] We have $K(m,\tau_I)\geq 1$.
\item[C2.] We have $K\{\lambda q(\tau_I)/n(1-\tau_I)\}^{1/2m}\rightarrow\infty$ as $n\rightarrow \infty$.
\item[C3.] We have $\lambda =o(q(\tau_I)^{-1} \{n(1-\tau_I)/q(\tau_I)\}^2)$ as $n\rightarrow \infty$.
\end{enumerate}

\medskip

Condition C concerns with the asymptotic property of the penalized spline estimator. 
C1 is detailed in Remark 2. 
C2 allows us to use the large $K$. 
If C3 fails, the asymptotic bias of the penalized spline estimator cannot be vanished.
We now show the asymptotic distribution of $\tilde{q}(\tau_I|x)$. 
First, we derive the two types of bias, model bias and shrinkage bias. 
Roughly speaking, the model bias is the bias between the $B$-spline model and the true function, and the shrinkage bias is the difference between the expectation of the penalized estimator and the unpenalized estimator. 
According to Section 2,2, the model bias is $b_a(\tau_I|x)=s_0(\tau_I|x)-q_Y(\tau_I|x)=O(K^{-m}q(\tau))$.  
This model bias becomes the negligible order from condition C2. 
That is, the bias is dominated by the shrinkage bias.
Define
\begin{eqnarray*}
b_\lambda(\tau|x)&=&\frac{\lambda q(\tau)}{(1-\tau) n}\vec{B}(x)^T\Lambda(H^{-\gamma})^{-1}D_m^TRD_m\vec{b}_0(\tau),\\
v^2(\tau|x)&=&\frac{q(\tau)^2}{(1-\tau) n}\vec{B}(x)^T\Lambda(H^{-\gamma})^{-1}G \Lambda(H^{-\gamma})^{-1}\vec{B}(x).
\end{eqnarray*}
As a result, $b_\lambda(\tau_I|x)$ is the asymptotic shrinkage bias and $v(\tau_I|x)$ is the asymptotic variance of $\tilde{q}_Y(\tau_I|x)$. 
The following theorem shows the asymptotic order of the asymptotic bias and variance of the intermediate order quantile estimator. 

\begin{theorem}\label{biasvariance1}
Under Conditions A--C, as $n\rightarrow \infty$, 
$$
b_\lambda(\tau_I|x)=O\left(q(\tau)\left(\frac{\lambda q(\tau_I)}{(1-\tau_I) n}\right)^{1/2}\right),\quad v^2(\tau_I|x)=O\left(\frac{q(\tau_I)^2}{(1-\tau_I)n}\left(\frac{\lambda q(\tau_I)}{(1-\tau_I) n}\right)^{-1/2m}\right).
$$
\end{theorem}

From condition C3 and Theorem \ref{biasvariance1}, we see that the shrinkage bias and variance converge to 0 as $n\rightarrow \infty$.
Using the central limit theorem, Lyapunov's condition, and a Cram\'er--Wold device, the asymptotic normality of the estimator $\tilde{q}_Y(\tau_I|x)$ can be shown. 

\begin{theorem}\label{as.dist}
Suppose that Conditions A--C hold. As $n\rightarrow \infty$, $b_\lambda(\tau|x)$ and $v^2(\tau|x)$ are the asymptotic bias and variance of $\tilde{q}_Y(\tau|x)$ and 
\begin{eqnarray*}
 \left(\frac{v(\tau_I|x)}{q_Y(\tau_Y|x)}\right)^{-1}\left\{\frac{\tilde{q}_Y(\tau_I|x)}{q_Y(\tau_Y|x)}-1-\frac{b_\lambda(\tau_I|x)}{q_Y(\tau_Y|x)}\right\} \stackrel{D}{\longrightarrow } N(0,1).
\end{eqnarray*}
Furthermore, under $\lambda=O(q(\tau_I)^2\{(1-\tau_I)n\}^{1/(2m+1)})$, the optimal rate of convergence of the mean integrated squared error (MISE) of $\tilde{q}_Y(\tau_I|x)/q_Y(\tau_I|x)$ is 
$$
E\left[\left\{\frac{\tilde{q}_Y(\tau_I|x)}{{q}_Y(\tau_I|x)}-1\right\}^2\right]=O(\{(1-\tau_I)n\}^{-2m/(2m+1)}).
$$
\end{theorem}

Theorems \ref{biasvariance1} and \ref{as.dist} yield that the trade-off between bias and variance is controlled by $\lambda$. 
Thus, this indicates that the careful choice of $K$ is not important in the penalized spline methods.
According to Yoshida (2013), for the center quantile level $\tau$, the MISE of the penalized spline quantile estimator has the order $O(n^{-2m/(2m+1)})$. 
Thus, the rate of convergence of the MISE of the penalized spline estimator for the intermediate quantile level is slower than that for the center quantile level. 
This result is not surprising in the context of the difficulties of the estimation for the tail quantile. 

When $\vec{B}(x)=\vec{x}$ and $\lambda=0$, the estimator is reduced to the ordinary quantile regression with the linear model. 
In the linear regression, the model bias is 0 and, hence, the bias term vanishes. 
On the other hand, since $G=E[XX^T]$ and $\Lambda(H^{-\gamma})^{-1}=\gamma^{-1} E[H(X)^{-\gamma}XX^T]^{-1}$, the 
asymptotic variance becomes 
$$
v^2(\tau_I|\vec{x})=\frac{q(\tau_I)^2}{(1-\tau_I)n}\gamma^2\vec{x}^TE[H(X)^{-\gamma}XX^T]^{-1}E[XX^T]E[H(X)^{-\gamma}XX^T]^{-1}\vec{x},
$$
which is similar in form to the asymptotic variance of the linear estimator of Lemma 3 of Wang et al.\ (2012).
Then, the rate of convergence of the MISE of the linear estimator is $E[\{\tilde{q}_Y(\tau_I|x)/q_Y(\tau|x)-1\}^2]=O(\{(1-\tau_I)n\}^{-1})$. 
Thus, it can be considered that Theorem 2 is the generalization of the asymptotic result of the linear-type parametric estimator.

\medskip

\noindent{\bf Remark 2.} 
Claeskens et al.\ (2009) have studied the asymptotic properties of the penalized spline mean estimator in two scenarios: 
roughly speaking, case (a) small $K$ scenario and case (b) large $K$ scenario. 
In case (a), the asymptotic behavior of penalized splines is similar to that of regression splines, which have the unpenalized estimator ($\lambda\equiv 0$). 
Case (b) results in the penalized splines nearing the smoothing splines. 
We briefly describe the asymptotic scenario branch along with the result of Claeskens et al.\ (2009). 
The penalized spline mean estimator is obtained as $\hat{f}(x)=\vec{B}(x)^T(Z^TZ+(\lambda/n)D_m^TRD_m)^{-1}Z^T \vec{y}$, where $Z=(\vec{B}(x_1),\ldots,\vec{B}(x_n))^T$ is the design matrix and $\vec{y}=(y_1,\ldots,y_n)$. 
Then, the two asymptotic scenarios are divided by the asymptotic order of the maximum eigenvalue of $(Z^TZ+(\lambda/n)D_m^TRD_m)^{-1}$, which is obtained as $K(m)^{2m}$
$$
K(m)=K\left(\frac{\lambda}{n}\right)^{1/2m}(1+o(1)).
$$
If $K(m)<1$ for a sufficiently large $n$, $K$, and $\lambda$, we achieve case (a). 
When $K(m)>1$ for a sufficiently large $n$, $K$, and $\lambda$, we achieve case (b). 
Although Claeskens et al.\ (2009) focused only on mean regression, these two scenarios can also be discussed with respect to quantile regression. 
The asymptotic scenario branch discussed in this section is dependent on the asymptotic order of $\Lambda(h)^{-1}$. 
Similar to Claeskens et al.\ (2009), the order of the maximum eigenvalue of $\Lambda(h)^{-1}$ can be obtained as $K(m,\tau_I)^{2m}$, which corresponds to $K(m)$ in mean regression.  
Consequently, condition C1 indicates that the large $K$ scenario should be studied. 
We finally note why we focus on the large $K$ scenario.
Ruppert (2002) recommended that one should first set the knots with a large $K$ to obtain the overfitted estimator and control $\lambda$ to achieve smoothness and fitness. 
Therefore, the large $K$ scenario matches the concept of Ruppert (2002) and this motivates us to consider the large $K$ scenario.

\vspace{3mm}

\subsection{Estimation of extreme order quantiles}

For the extreme order quantile, the estimator $\tilde{q}_Y(\tau_E|x)$ discussed in the previous section would not have asymptotic normality (Chernozhukov 2011). 
In this paper, we try to approximate the extreme conditional quantile from intermediate quantile. 
According to Weissman (1978), the following holds: 
$$
q_Y(\tau_E|x)\approx \left(\frac{1-\tau_I}{1-\tau_E}\right)^\gamma q_Y(\tau_I|x).
$$
From this, using the estimator of the intermediate order quantile $\tilde{q}_Y(\tau|x)$, we define the extrapolated estimator of the extreme order quantile. 
To achieve this, we need to estimate the EVI $\gamma$.

Let $\tau_1>\ldots>\tau_k$ be the sequence of quantile levels, where $\tau_j=1-([n^\eta]+j)/(n+1)$, $\eta\in(0,1)$ and $[a]$ is the integer part of $a$. 
Then, since $(1-\tau_j)n=n([n^\eta]+j)/(n+1)\rightarrow \infty$ as $n\rightarrow \infty$, all $\tau_j$ are intermediate order quantiles. 
Using this sequence, we define the Hill-types estimator of $\gamma$ as 
\begin{eqnarray*}
\hat{\gamma}(x)=\frac{1}{k-1}\sum_{j=1}^{k-1} \log\left(\frac{\tilde{q}_Y(\tau_j|x)}{\tilde{q}_Y(\tau_k|x)}\right).
\end{eqnarray*}
In this paper, we assume that the tail behavior of $F_Y(\cdot|x)$ and $F(\cdot)$ are equivalent (see, Condition A1). 
Therefore, it is somewhat unnatural that the estimator of $\gamma$ varies with $x$. 
Nevertheless, we define the extrapolated estimator with $\hat{\gamma}(x)$ and investigate the mathematical property. 
For $x\in[a,b]$, using $\hat{\gamma}(x)$, we define the estimator of the extreme order quantile as 
\begin{eqnarray*}
\hat{q}_Y(\tau_E|x)= \left(\frac{1-\tau_I}{1-\tau_E}\right)^{\hat{\gamma}(x)} \tilde{q}_Y(\tau_I|x).
\end{eqnarray*}

We next consider the EVI estimator along with condition A1. 
Define the common index (pooled) estimator 
$$
\hat{\gamma}^{C}=\frac{1}{n}\sum_{i=1}^n \hat{\gamma}(x_i)
$$
and the extrapolated estimator with common index estimator $\hat{\gamma}^{C}$ as
\begin{eqnarray*}
\hat{q}^{C}_Y(\tau_E|x)= \left(\frac{1-\tau_I}{1-\tau_E}\right)^{\hat{\gamma}^{C}} \tilde{q}_Y(\tau_I|x).
\end{eqnarray*}
  
To investigate the asymptotic distribution of $\hat{q}_Y(\tau_E|x)$ and $\hat{q}^{C}_Y(\tau_E|x)$, we impose the second-order condition in Conditions A.
\begin{itemize}
\item[A5] The function $U(t)=F^{-1}(1-1/t)$ satisfies the second-order condition with $(\gamma,\rho,A)$. That is, there exist $\rho<0$ and $A(t)\in RV(\rho)$ such that as $t\rightarrow \infty$,
\begin{eqnarray*}
A(t)^{-1}\left\{\frac{U(tz)}{U(t)}-z^\gamma\right\}\rightarrow \frac{(z^\gamma)(z^\rho-1)}{\rho}.
\end{eqnarray*}
Furthermore, $A(t)=\gamma d t^\rho$ with $d\not=0$.
\item[${\rm A3}^\prime$] There exist $\delta>0$ and positive, continuous and bounded function $H_1(x)$ such that as $y\rightarrow \infty$,
$$
1-F_V(y|x)=H(x)\{1-F(y)\}+H_1(x)(1-F(y))^{1+\delta}(1+o(1))
$$
\end{itemize} 

Condition A5 is the standard second-order condition of EVT and is detailed in de Haan and Ferreira (2006). 
Condition ${\rm A3}^\prime$ provides the second order of tail behavior of $F(y)$. 
From conditions A5 and ${\rm A3}^\prime$, we see that $U_Y(\cdot|x)$ also satisfy the second-order condition with $(\gamma,\rho^*=\min\{\rho,-\delta\},A^*(\cdot|x))$ and $A^*(t|x)=\gamma d^*(x) t^{\rho^*}$ with $d^*(x)\not=0$, which were proven in Lemma 2 of Wang et al. (2012). 
Using this, we show the asymptotic property of the Hill-type estimator of the EVI in the following. 

\begin{theorem}
Suppose that the smoothing parameter included in $\tilde{q}(\tau_I|x)$ satisfies $\lambda=O(q(\tau_I)^2\{(1-\tau_I)n\}^{1/(2m+1)})$. 
Furthermore, suppose that $k\rightarrow \infty$, $k/n\rightarrow 0, n^\eta\log(k)/k^{m/(2m+1)}\rightarrow 0$ and $k^{m/(2m+1)}(n/k)^{\max\{-\gamma,-\delta,\rho\}}\rightarrow 0$ as $n\rightarrow \infty$. Under Conditions A--C, as $n\rightarrow \infty$, 
\begin{eqnarray*}
\frac{\hat{\gamma}(x)-\gamma-b(k|x)}{v(k|x)}\stackrel{D}{\longrightarrow } N(0,1),\ \ \ 
\end{eqnarray*}
and 
\begin{eqnarray*}
\frac{\hat{\gamma}^{C}-\gamma-E[b(k|X)]}{E[v(k|X)]}\stackrel{D}{\longrightarrow } N(0,1),
\end{eqnarray*}
where $b(k|x)$ and $v(k|x)$ are defined in (\ref{asbiasvarianceevi}) of Appendix and have an asymptotic order $O(k^{-m/(2m+1)})$. 
Furthermore, 
$$
E[\left\{\hat{\gamma}(x)-\gamma\right\}^2]=O\left(k^{-2m/(2m+1)}\right)
$$
and 
$$
E[\left\{\hat{\gamma}^{C}-\gamma\right\}^2]=O\left(k^{-2m/(2m+1)}\right).
$$
\end{theorem}

Using Theorem 3, we obtain the asymptotic normality of the ratio of $\hat{q}_Y(\tau_E|x)$ and $\hat{q}^{C}_Y(\tau_E|x)$.
\begin{theorem}
Suppose that the same conditions as Theorem 3. 
Furthermore, assume that $k^{-m/(2m+1)}\log\{(1-\tau_I)/(1-\tau_E)\}\rightarrow 0$ as $k, n\rightarrow \infty$, $\tau_I,\tau_E\rightarrow 1$.
As $n\rightarrow \infty$, $\tau_E\rightarrow 1$ and $n(1-\tau_E)\rightarrow c \in[0,\infty),$
\begin{eqnarray*}
\frac{\frac{\hat{q}_Y(\tau_E|x)}{q_Y(\tau_E|x)}-1-bias(\tau_E|x)}{s(\tau_E|x)}\stackrel{D}{\longrightarrow } N(0,1)
\end{eqnarray*}
and 
\begin{eqnarray*}
\frac{\frac{\hat{q}^{C}_Y(\tau_E|x)}{q_Y(\tau_E|x)}-1-bias^C(\tau_E|x)}{s^C(\tau_E|x)}\stackrel{D}{\longrightarrow } N(0,1)
\end{eqnarray*}
where $bias(\tau_E|x)$, $s(\tau_E|x)$, $bias^C(\tau_E|x)$ and $s^C(\tau_E|x)$ are defined in (\ref{biasExtrapolation}), (\ref{sdExtrapolation}), (\ref{biasExtrapolationC}) and (\ref{sdExtrapolationC}) of proof of Theorem 4 of Appendix. 
Furthermore, 
$$
E\left[\left\{\frac{\hat{q}_Y(\tau_E|x)}{q_Y(\tau_E|x)}-1\right\}^2\right]=O\left(\max\left\{k^{-\frac{2m}{2m+1}}\log^2\left(\frac{1-\tau_I}{1-\tau_E}\right),\{(1-\tau_I)n\}^{-\frac{2m}{2m+1}}\right\}\right).
$$
and 
$$
E\left[\left\{\frac{\hat{q}^{C}_Y(\tau_E|x)}{q_Y(\tau_E|x)}-1\right\}^2\right]=O\left(\max\left\{k^{-\frac{2m}{2m+1}}\log^2\left(\frac{1-\tau_I}{1-\tau_E}\right),\{(1-\tau_I)n\}^{-\frac{2m}{2m+1}}\right\}\right).
$$
\end{theorem}

For the asymptotic order in Theorem 4, the term $O(k^{-2m/(2m+1)}\log^2\{(1-\tau_I)/(1-\tau_E)\})$ is derived from $\hat{\gamma}(x)$ and the another term is derived from $\tilde{q}_Y(\tau_I|x)$. 
If we use $\tau_I=\tau_1=(n-[n^{\eta}])/(n+1)$ or $\tau_I=O(\tau_1)$, we have $(1-\tau_I)n=O([n^\eta])=o(k)$ since $[n^\eta]/k\rightarrow 0$. 
That is, the asymptotic inference of $\hat{q}_Y(\tau_E|x)$ is dominated by that of $\tilde{q}_Y(\tau_I|x)$ and hence, 
the rate of convergence of the estimator is 
$$
E\left[\left\{\frac{\hat{q}_Y(\tau_E|x)}{q_Y(\tau_E|x)}-1\right\}^2\right]=O\left(\{(1-\tau_I)n\}^{-\frac{2m}{2m+1}}\right).
$$
One may have sense of discomfort with this result since the extreme order quantile estimator and the intermediate order quantile estimator has same rate of convergence. 
Indeed, the leading terms of $\tilde{q}_Y(\tau_I|x)$ and $\hat{q}_Y(\tau_E|x)$ are similar. 
However, the convergence speed of the subsequent term of $\hat{q}_Y(\tau_E|x)$ is obviously slower than that of $\tilde{q}_Y(\tau_I|x)$ because of the influence of $\hat{\gamma}(x)$. 
Therefore, for the application with a finite sample, the behavior of $\tilde{q}_Y(\tau_I|x)$ would be more stable than $\hat{q}_Y(\tau_E|x)$ .
On the other hand, when $\tau_I=\tau_k$ or $\tau_I=O(\tau_k)$, which leads to $n(1-\tau_I)=O(k)$, is adopted, $\hat{q}_Y(\tau_E|x)$ is heavily affected by $\hat{\gamma}(x)$ but not by $\tilde{q}_Y(\tau_I|x)$. 
Actually, since $n(1-\tau_E)\rightarrow c\in [0,\infty)$ and $\log(\{1-\tau_I\}/\{1-\tau_E\})=\log(\{n(1-\tau_I)\}/\{n(1-\tau_E)\})=O(\log[k/\{n(1-\tau_E)\}])$, we have 
\begin{eqnarray*}
E\left[\left\{\frac{\hat{q}_Y(\tau_E|x)}{q_Y(\tau_E|x)}-1\right\}^2\right]&=&O\left(\max\left\{k^{-\frac{2m}{2m+1}}\log^2\left(\frac{k}{n(1-\tau_E)}\right), k^{-\frac{2m}{2m+1}}\right\}\right)\\
&=&O\left(k^{-\frac{2m}{2m+1}}\log^2\left(\frac{k}{n(1-\tau_E)}\right)\right).
\end{eqnarray*}

For the common index quantile estimator $\hat{q}^{C}_Y(\tau_E|x)$ with $\tau_I=O(\tau_k)$, the asymptotic order of $\hat{q}_Y^{C}(\tau_E|x)/q_Y(\tau_E|x)$ are dominated by the term of $\hat{\gamma}^C$, and this do not vary with $x$. 
This result is quite unnatural in the quantile regression although $O(\log^2(k/\{n(1-\tau_E)\}))$, which is the difference between the asymptotic inference of $\hat{\gamma}^C$ and $\tilde{q}_Y(\tau_I|x)$, is quite small. 
Therefore, if the common index quantile estimator is mainly used, we may have to choose the baseline quantile $\tau_I$ so that $\tau_I>\tau_k$. 
Thus, the balance of $\tau_I$ and $k$ controls the asymptotic behavior of $\hat{q}_Y(\tau_E|x)$.
The same is true of $\hat{q}^{C}_Y(\tau_E|x)$. 

Wang et al. (2012) obtained the extrapolated estimator in the linear model with $\tau_I=\tau_k$. 
From their result, the rate of convergence of the MISE of the linear estimator is $E[\{\hat{q}_Y(\tau_E|x)/q_Y(\tau_E|x)-1\}^2]=O(k^{-1}\log^2(\{1-\tau_I\}/\{1-\tau_E\}))$. 
That is, the difference in the rate of convergence between the parametric estimator and the nonparametric estimator is $k^{-1}$ and $k^{-2m/(2m+1)}$, which could be intuitively derived from the classical works on parametric and nonparametric regression.

\vspace{3mm}

\noindent{\bf Remark 3}

The intermediate order quantile and the extreme order quantile are separated mathematically by the rate of the quantile level. 
However, in data analysis, the distinction between these two rates should be drawn for fixed $n$. 
Define $\xi=\xi(\tau,n)=(1-\tau)n$.
Using $\xi$, Chernozhukov and Fern\'andez-val (2011) suggested the following rule of thumb.
For the quantile level $\tau$, if $\xi<30$, it is the extreme order inference, that is, $\tau=\tau_E$ and we should use $\hat{q}_Y(\tau|x)$. 
When $\xi\geq 30$, it is sufficient to use the intermediate order quantile estimator $\tilde{q}_Y(\tau|x)$. 
If the predictor is the continuous, this threshold is $\xi=15$--20. 
However, they noted that the above rule is conservative.  
In this paper, we treat $\tau_1=\frac{n-[n^{\eta}]}{n+1}$ as the intermediate order quantile. 
For example, when $n=200$, $\tau=0.925$ leads to $\xi=15$. 
Then, we have $\eta\approx 0.5$. 
For $n=1000$, $\tau=0.985$ and $\eta=0.4$ correspond to $(1-\tau_1)n\approx (1-\tau)n=15$. 
On the other hand, Wang et al. (2012) suggested to use $\eta=0.1$.
In their rule, we have $\xi(\tau_1,n)=3$ for $n<1025$.
Thus, it seems that the rule of Chernozhukov and Fern\'andez-val (2011) is more conservative rather than that of Wang et al. (2012). 
In our experience, the rule of Wang et al. (2012) worked well for $n\leq 1000$. 
Therefore, in the simulation study of the next section, we also use $\eta=0.1$. 
However, the determination of the split of the intermediate order and the extreme order is still a difficult problem and further study would be welcomed.

\section{Simulation}

The practical performance of the proposed estimator is confirmed by Monte Carlo simulation. 
Define the true regression model as 
\begin{eqnarray}
Y_i=f(X_i)+\sigma(X_i)\varepsilon_i(X_i),\ \ i=1,\ldots,n, \label{simumodel}
\end{eqnarray}
where 
$$
f(x)=\sqrt{x(1-x)}\sin\left(\frac{2\pi(1+2^{-7/5})}{x+2^{-7/5}}\right)
$$
and $\sigma(x)=10^{-1}(1+x)$. 
The predictor $X_i$ is independently generated from the standard uniform distribution. 
This setting was introduced by Daouia et al. (2013). 
We consider two types of error distribution: (a) $\varepsilon_i(X_i)=\varepsilon_i\sim t_{5}$ and (b) $\varepsilon_i(X_i)\sim t_{s(x)}$, where 
$$
s(x)=[\nu(x)]+1,\ \ \ \nu(x)=[\{1.1-0.5\exp[-64(x-0.5)^2]\}\{0.1+\sin(\pi x)\}]^{-1}.
$$
The error type (b) is also used by Daouia et al. (2013). 
For the $t_\nu$ distribution, the EVI is $\gamma=1/\nu$ and hence, $\gamma=0.2$ and $\gamma(x)=1/s(x)$ for (a) and (b), respectively 
For both cases, the EVI is larger than 0, which indicates that the distribution of $Y_i$ has a heavy tail. 
In (\ref{simumodel}), the conditional $\tau$th quantile of $Y$ given $X=x$ is $q_Y(\tau|x)=f(x)+\sigma(x)q_\varepsilon(\tau|x)$, where $q_\varepsilon(\tau|x)$ is the $\tau$th quantile of $\varepsilon_i(x)$. 
For the case (a), $q_\varepsilon(\tau|x)=q(\tau)$ is the $\tau$th quantile of $t_5$ and is not dependent on $x$. 
Thus, the model (\ref{simumodel}) with (a) is the location-scale shifted model and is of the form of (\ref{truequantile}).  
In the case of (b), $\gamma(x)=1/2$ for $x\in[0.12, 0.88]$ and $\gamma(x)\in(0,1/3)$ otherwise. 
That is, the model (\ref{simumodel}) with (b) has high EVI at the center and low (but larger than 0) EVI at the boundary.
The conditional quantile with (b) fail due to Conditions A. 
However, it is important to confirm the performance of the estimator under (b).

We construct the intermediate order quantile estimator $\tilde{q}_Y(\tau|x)$, the Hill-type estimator $\hat{\gamma}(x), \hat{\gamma}^{C}$, and the extreme order quantile estimator $\hat{q}_Y(\tau|x)$ and $\hat{q}_Y^{C}(\tau|x)$. 
For the intermediate order quantile estimator, we use the number of knots $K=40$ and the smoothing parameter selected via generalized approximated cross-validation (Yuan 2006). 
To obtain $\hat{\gamma}$, $\hat{\gamma}^{C}$, $\hat{q}_Y$ and $\hat{q}_Y^{C}$, we need to determine $\tau_j=1-\frac{[n^\eta]+j}{n+1}$ ($j=1,\ldots,k$). 
In this simulation, $\eta$ is chosen so that $(1-\tau_1)n=\xi=3$ and $k=[7.5n^{1/3}]$. 
Such $\eta$ and $k$ are selected from a pilot study.
Wang et al. (2012) used $\eta=0.1$ and $k=[4.5 n^{1/3}]$ in the linear regression. 
Thus, our $k$ is somewhat larger than that of Wang et al. (2012).

For the estimator $\hat{f}(x)$ of the true function $f(x)$, the Mean Integrated Squared Error (MISE): 
$$
MISE(\hat{f})=\int_0^1 E[\{\hat{f}(x)-f(x)\}^2] dx
$$
is used as the accuracy measure of the estimator. 
We calculate the estimated MISE of $\tilde{q}_Y(\tau|x)$, $\hat{q}_Y(\tau|x)$ and $\hat{q}^{C}_Y(\tau|x)$ over 400 replications. 
The estimators $\tilde{q}_Y(\tau|x)$, $\hat{q}_Y(\tau|x)$ and $\hat{q}^{C}_Y(\tau|x)$ are denoted by PSE-I, PSE-E and PSE-Ep. 
As the competitors, we consider the functional nonparametric estimator (Gardes et al. 2010) and the kernel smoothing estimator (Daouia et al. 2013). 
The estimators $\hat{q}_1(\tau,x)$ and $\hat{q}_2(\tau,x)$ defined in Gardes et al. (2010) are denoted by FNS-I and FNS-E, respectively. 
Furthermore, the estimators $\hat{q}_n(\tau|x)$ and $\tilde{q}^{{\rm RP}}_n(\tau|x)$ defined by Daouia et al. (2013) are labeled by KSE-I and KSE-E in this section. 
Thus, FNS-I, FNS-E, KSE-I and KSE-E are also demonstrated in simulation.

\begin{figure}[h]
\begin{center}
\subfloat[True conditional quantiles]{
\includegraphics[width=65mm,height=35mm]{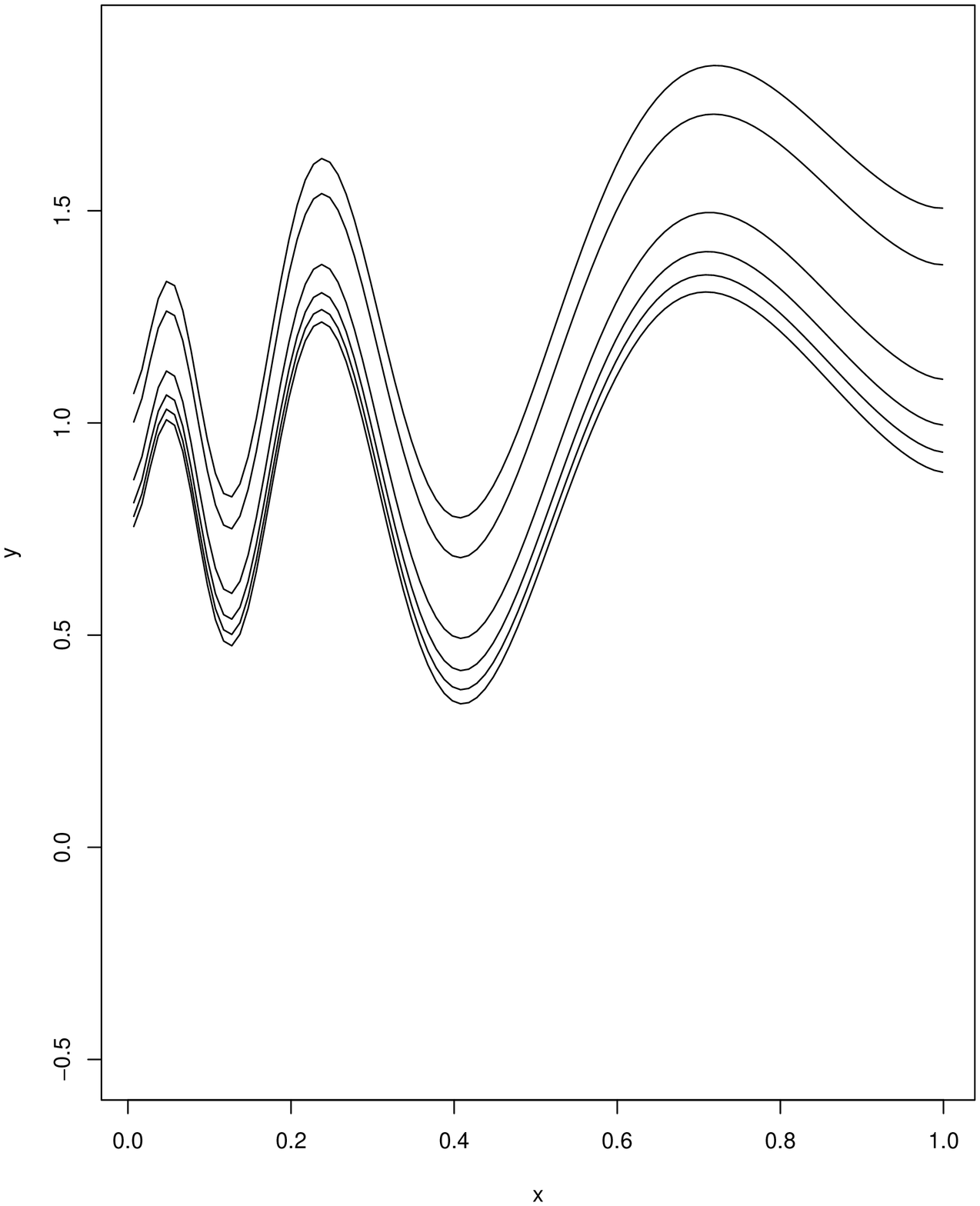}
}
\subfloat[$n=200$]{
\includegraphics[width=65mm,height=35mm]{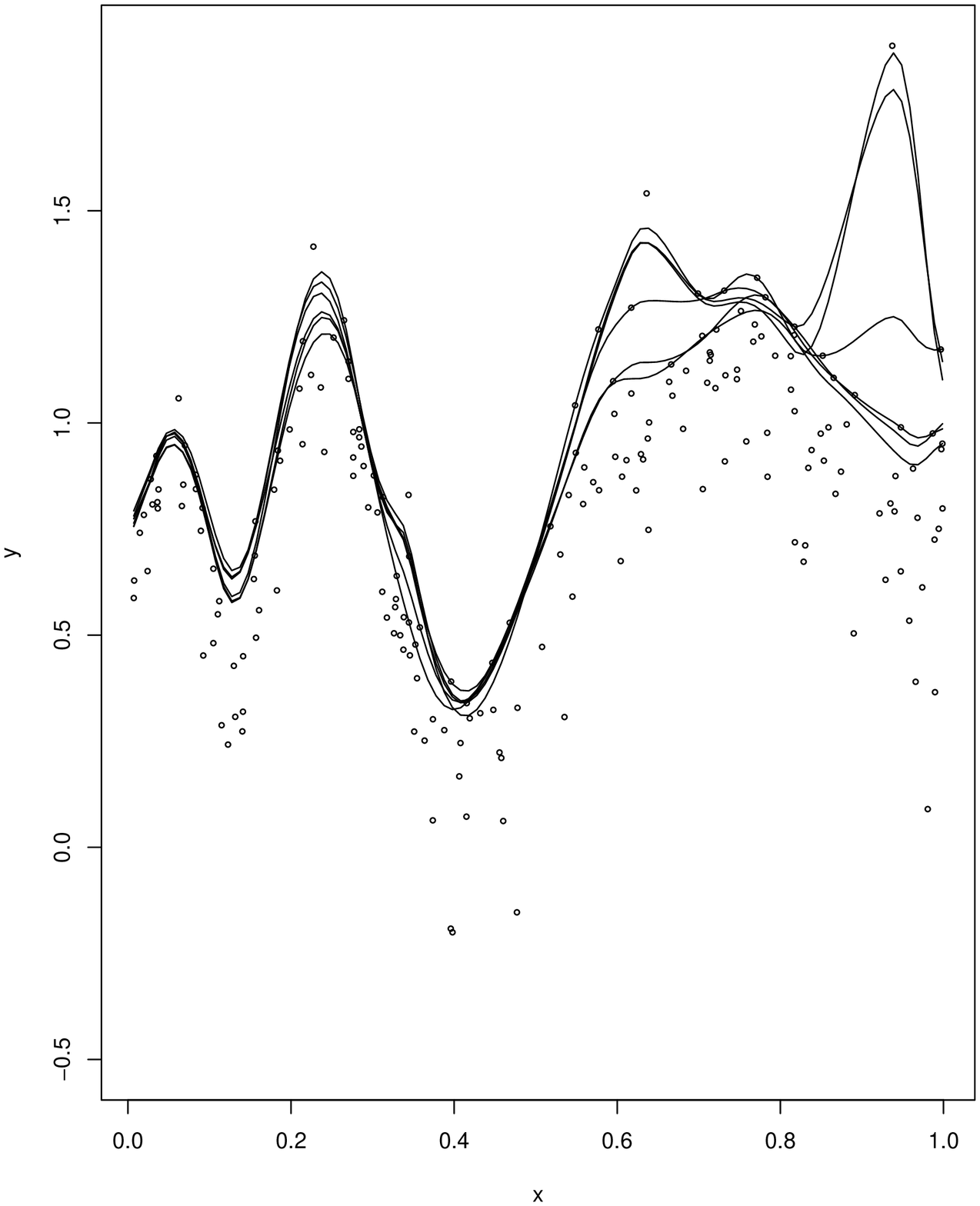}
}\\
\subfloat[$n=600$]{
\includegraphics[width=65mm,height=35mm]{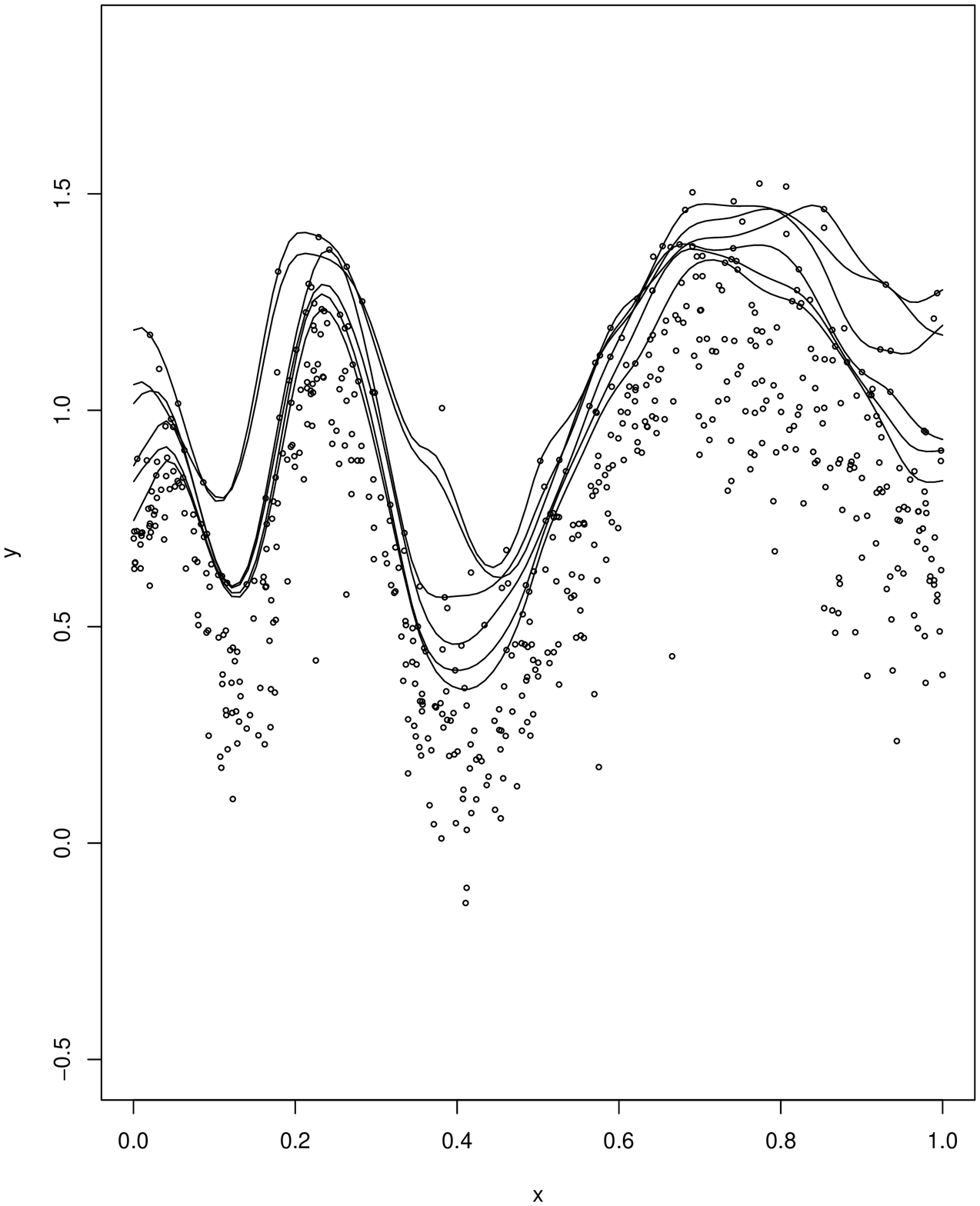}
}
\subfloat[$n=1000$]{
\includegraphics[width=65mm,height=35mm]{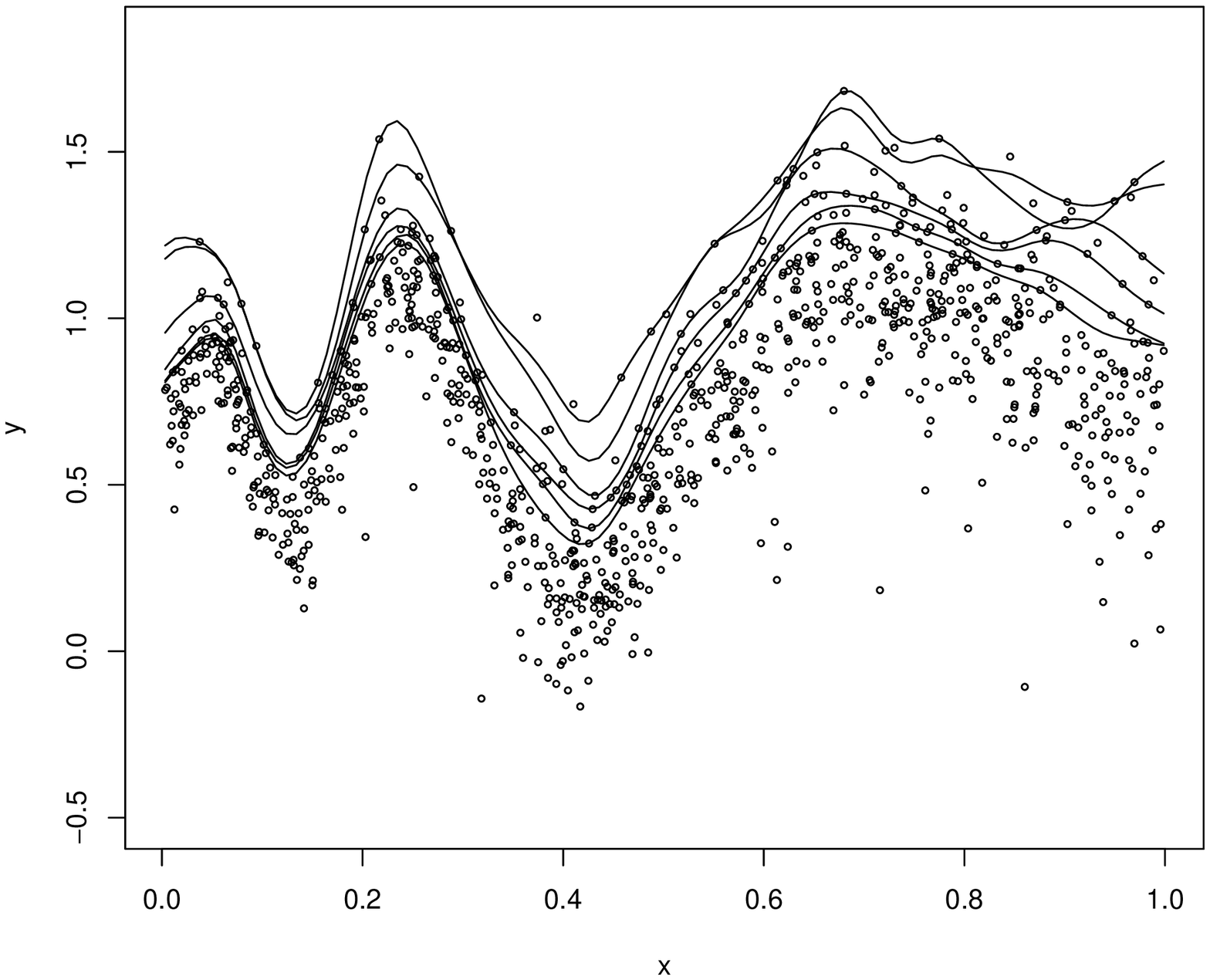}
}
\end{center}
\caption{True conditional quantiles for $\tau=0.8, 0.85, 0.9, 0.95, 0.99$ and 0.995, and these intermediate order quantile estimators for one dataset. \label{fig1} }
\end{figure} 

We report the simulation results for the case (a). 
Figure \ref{fig1} shows the true conditional quantiles and the intermediate order quantile estimators for one dataset with $n=200, 600$ and 1000.
For $\tau=0.8, 0.85$ and 0.9, the estimator behaved well, but for $\tau\geq 0.95$, there was a significant difference between the true function and the intermediate order quantile estimator. 
In Figure \ref{fig2}, the MISEs of the estimators for $\tau\in[0.5,0.995]$ are illustrated. 
We can observe that the proposed estimator behaves better than the competitors. 
From Figure \ref{fig2} (d), we can find that the estimator behaves well as $n$ increases.
This indicates that the estimator has a consistency property. 
However, as $\tau$ increases, the performance of the estimator becomes drastically decreases. 
Therefore, for $\tau\approx 1$, it is difficult to predict the conditional quantile using the intermediate order quantile estimator.

\begin{figure}
\begin{center}
\subfloat[$n=200$]{
\includegraphics[width=75mm,height=50mm]{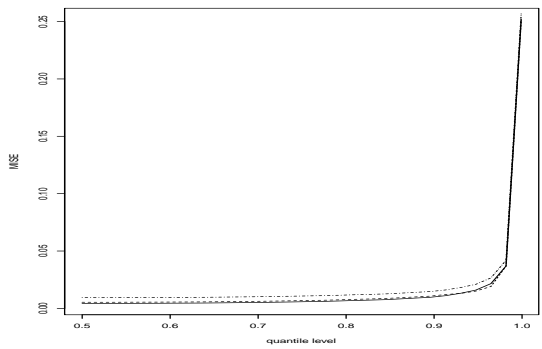}
}
\subfloat[$n=600$]{
\includegraphics[width=75mm,height=50mm]{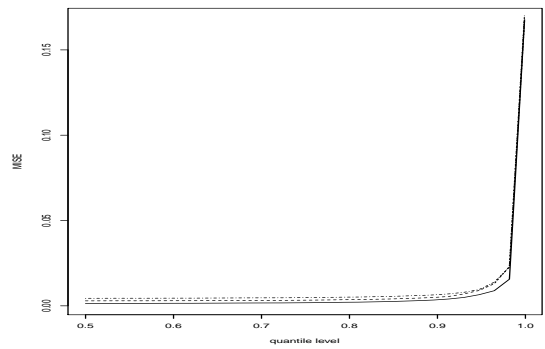}
}\\
\subfloat[$n=1000$]{
\includegraphics[width=75mm,height=50mm]{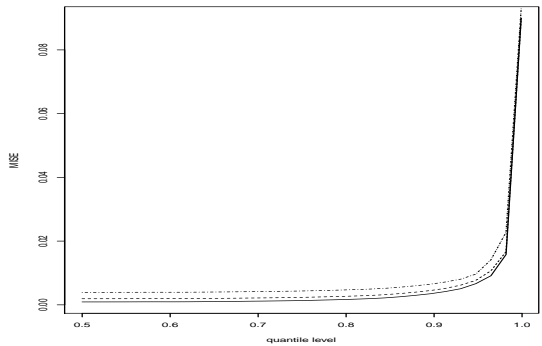}
}
\subfloat[PSE-I]{
\includegraphics[width=75mm,height=50mm]{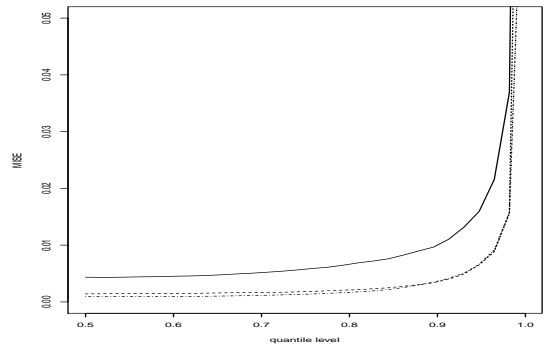}
}
\end{center}
\caption{MISE of the intermediate order quantile estimators for $\tau\in[0.5,0.995]$. In (a), (b) and (c), the solid line is PSE-I. The dashed line and dot-dashed lines are KSE-I and FNS-I, respectively. 
In (d), the solid, dashed and dot-dashed lines are PSE for $n=200$, $n=600$ and $n=1000$. \label{fig2} }
\end{figure} 

We next show the performance of the EVI estimator. 
Figure \ref{fig3} shows the behavior of $\hat{\gamma}^{C}$ over $k$ for one dataset and the distribution of $\hat{\gamma}^{C}$ using $k=[7.5n^{1/3}]$ by Monte Carlo simulation. 
From the results, we see that the suggested $k=[7.5n^{1/3}]$ is good choice. 
When $n=1000$, the behavior of $\hat{\gamma}^{C}$ is stable from (c) and (d). 

\begin{figure}
\begin{center}
\subfloat[$n=200$]{
\includegraphics[width=65mm,height=35mm]{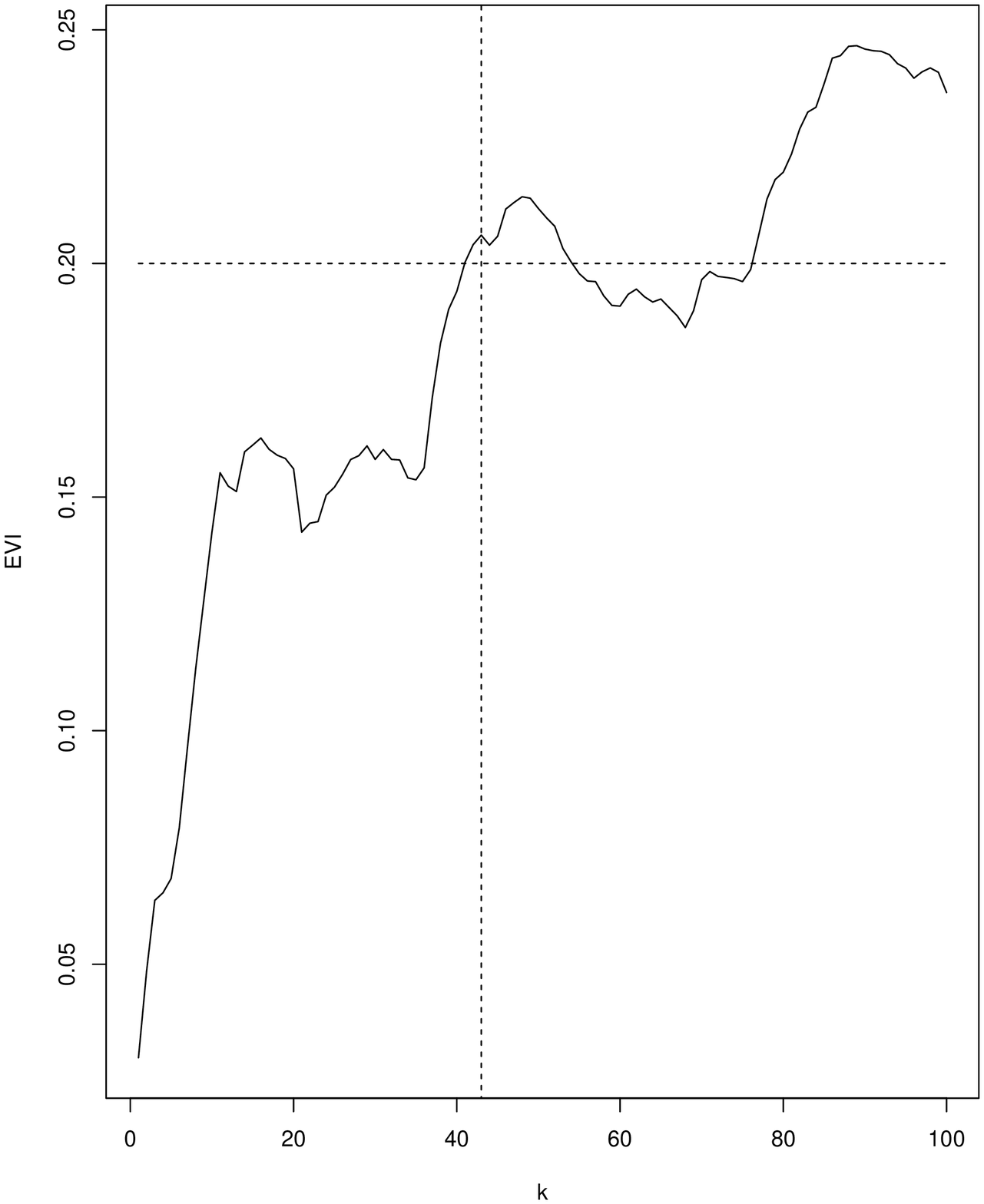}
}
\subfloat[$n=600$]{
\includegraphics[width=65mm,height=35mm]{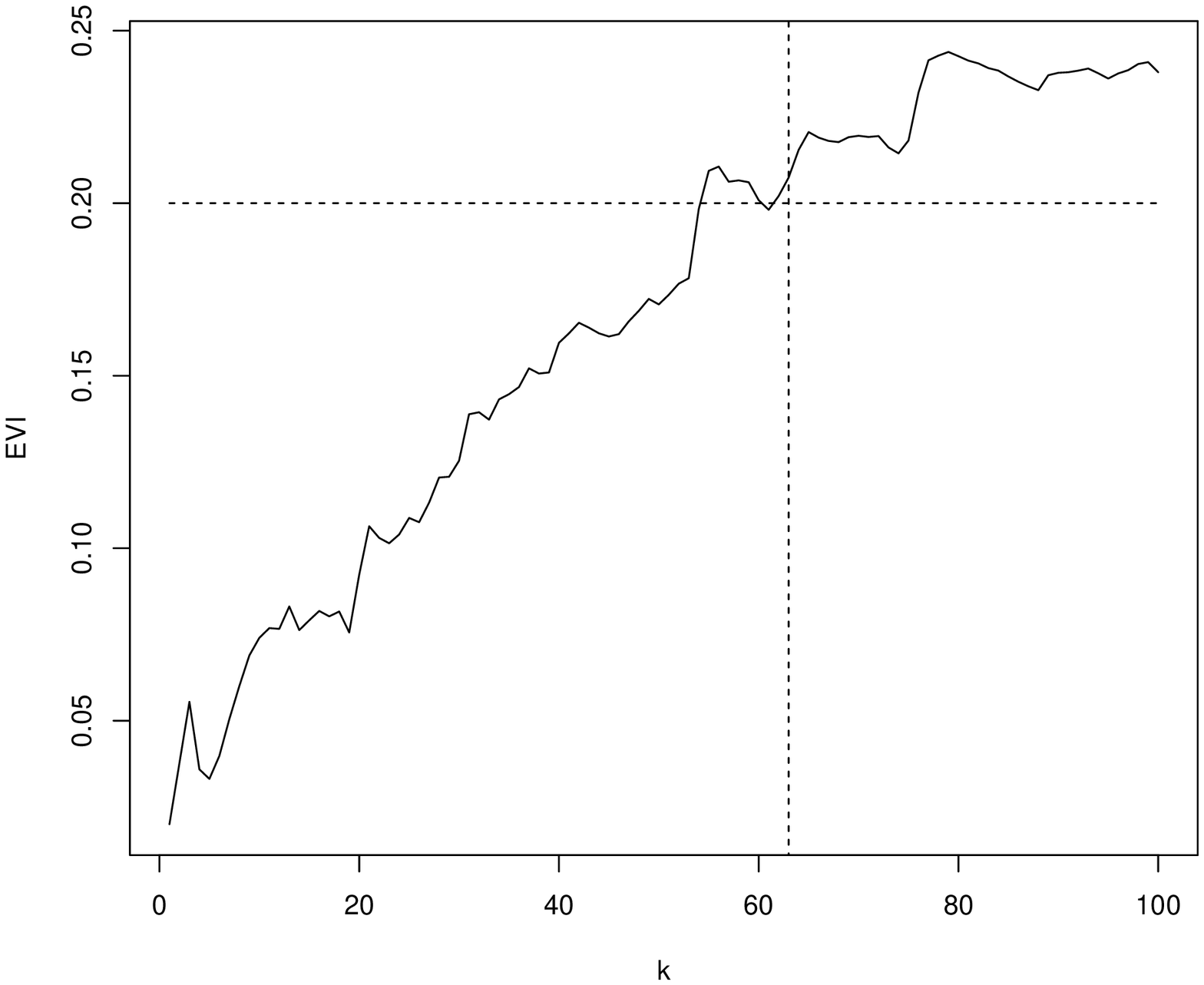}
}\\
\subfloat[$n=1000$]{
\includegraphics[width=65mm,height=35mm]{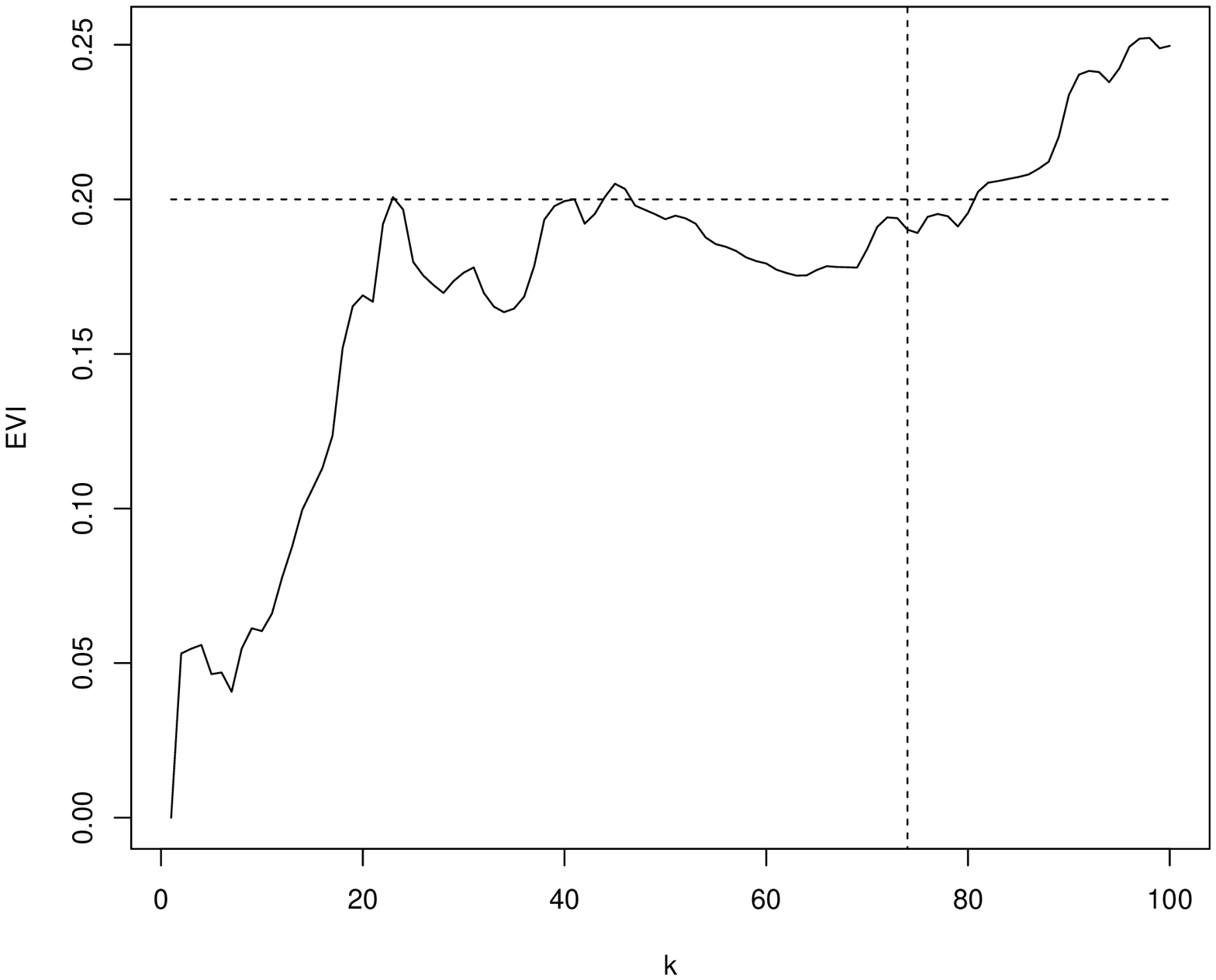}
}
\subfloat[The distribution of $\hat{\gamma}^{C}$]{
\includegraphics[width=65mm,height=35mm]{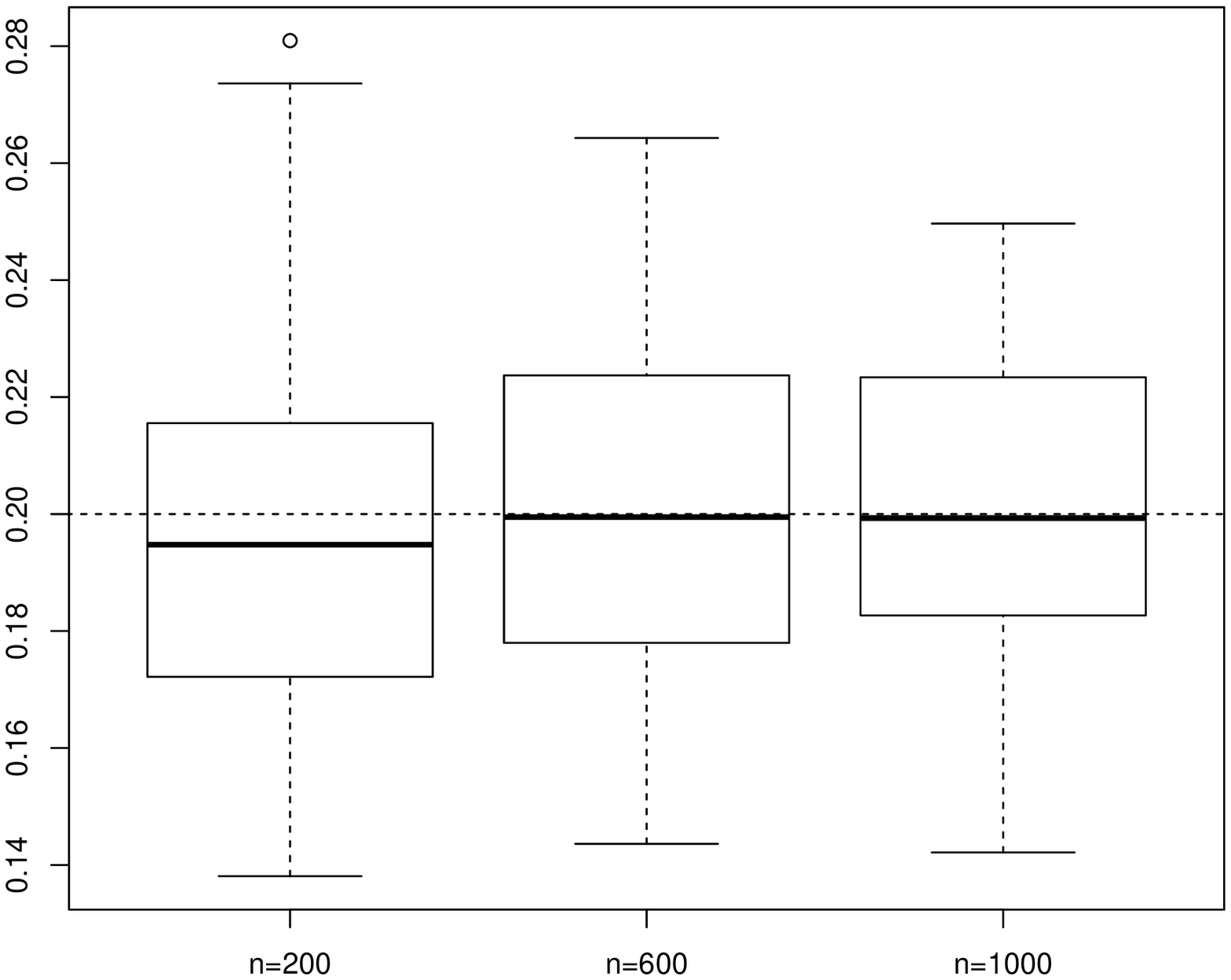}
}
\end{center}
\caption{(a--c): The sample path of the common index estimator of EVI with $k$ for one dataset. 
The dataset is similar to Figure \ref{fig1} for each $n$. 
The dashed line is $k=[7.5 n^{1/3}]$ and $\gamma=0.2$. 
(d) Box plot of $\hat{\gamma}^{C}$ with $k=[7.5 n^{1/3}]$ from 400 replications. \label{fig3} }
\end{figure}

\begin{figure}
\begin{center}
\subfloat[$n=200$]{
\includegraphics[width=50mm,height=45mm]{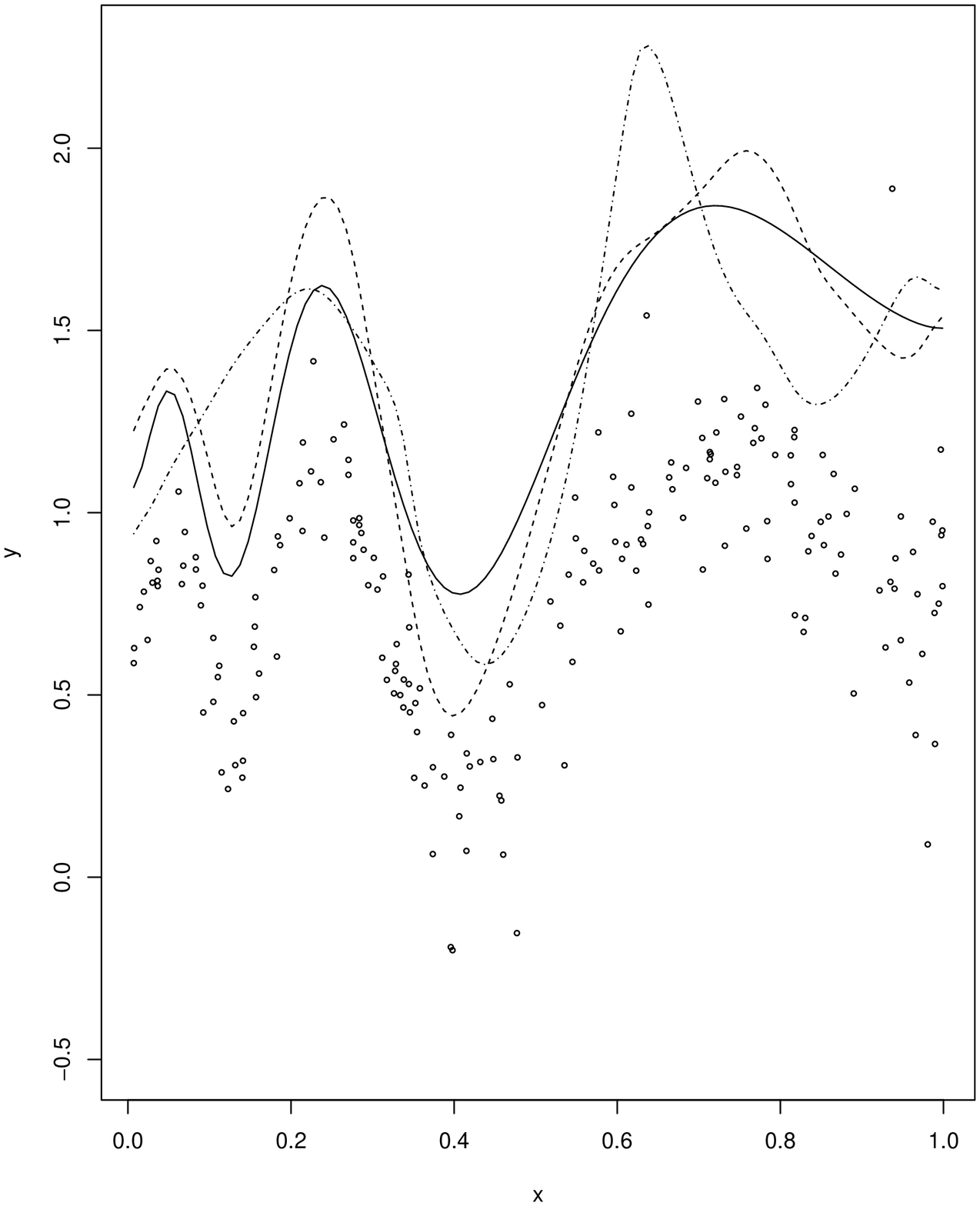}
}
\subfloat[$n=600$]{
\includegraphics[width=50mm,height=45mm]{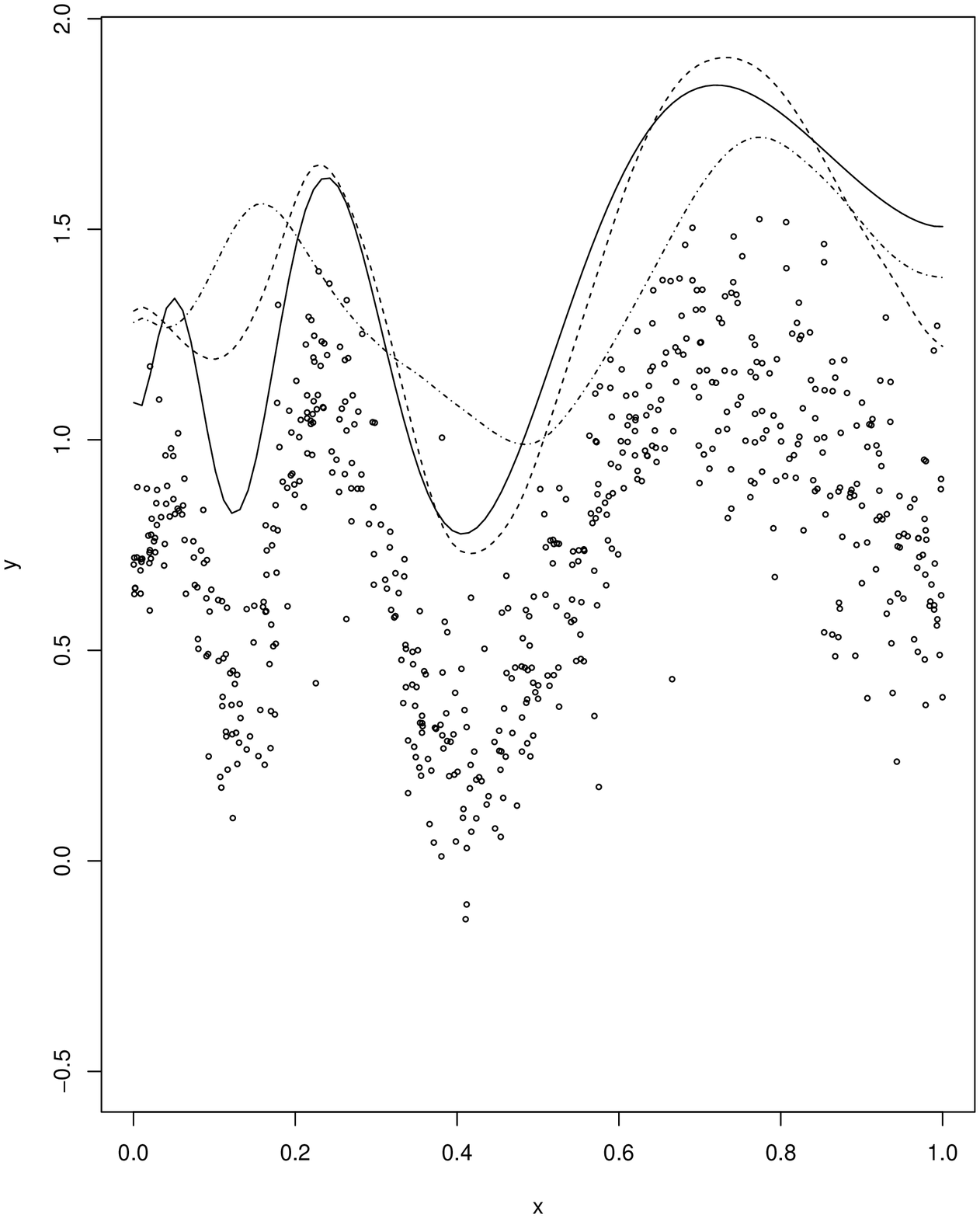}
}
\subfloat[$n=1000$]{
\includegraphics[width=50mm,height=48mm]{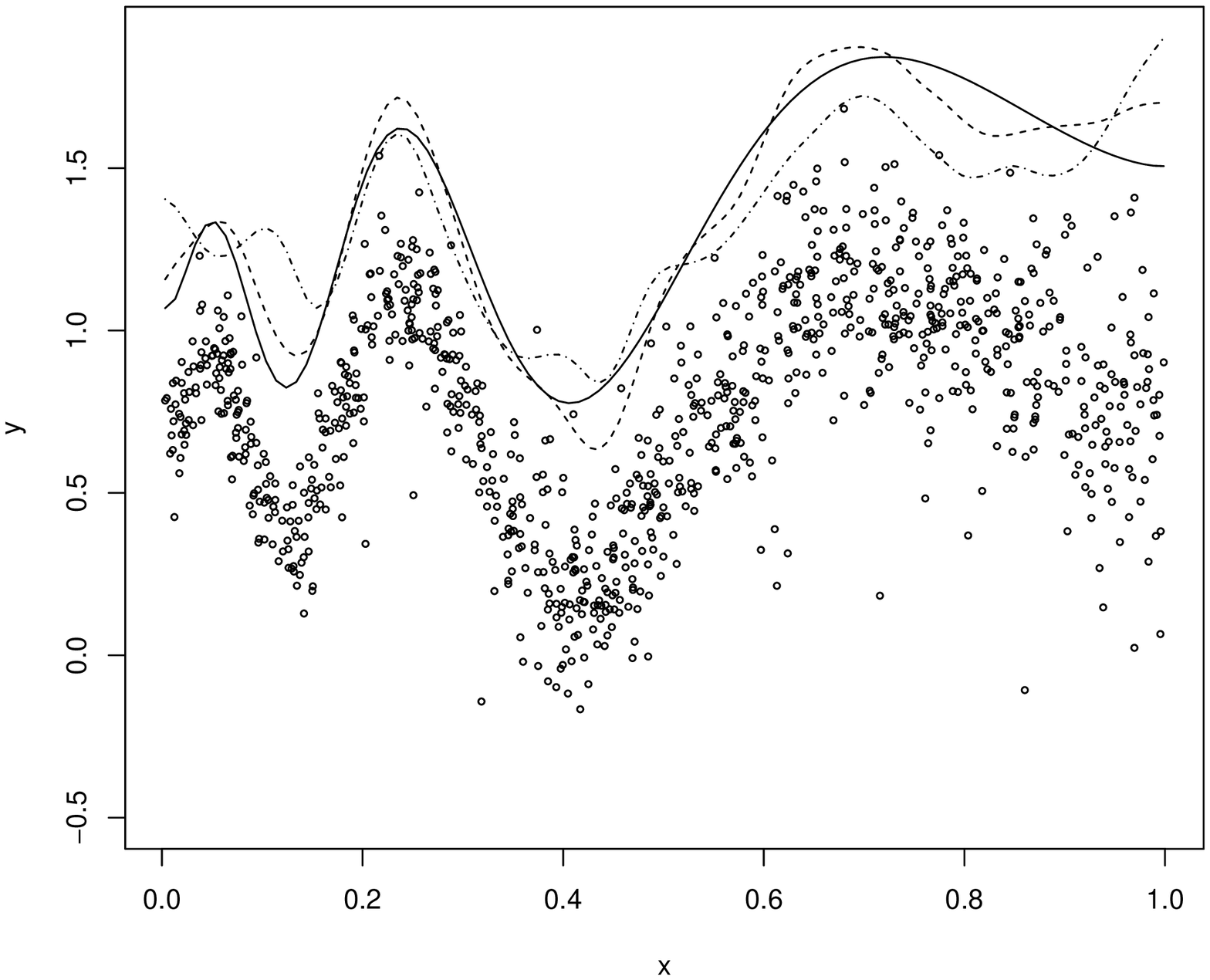}
}
\\
\subfloat[MISE, $n=200$]{
\includegraphics[width=50mm,height=45mm]{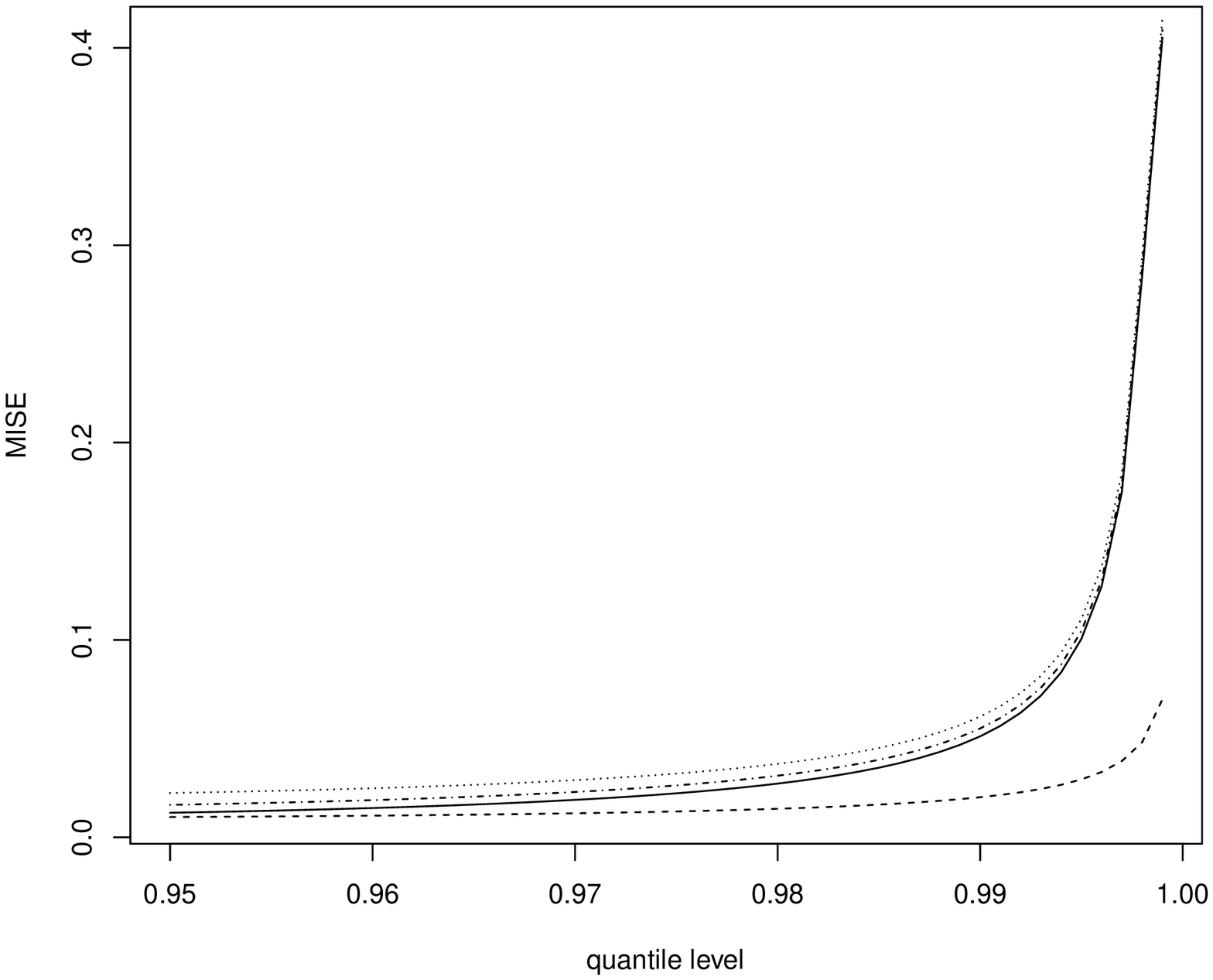}
}
\subfloat[MISE, $n=600$]{
\includegraphics[width=50mm,height=45mm]{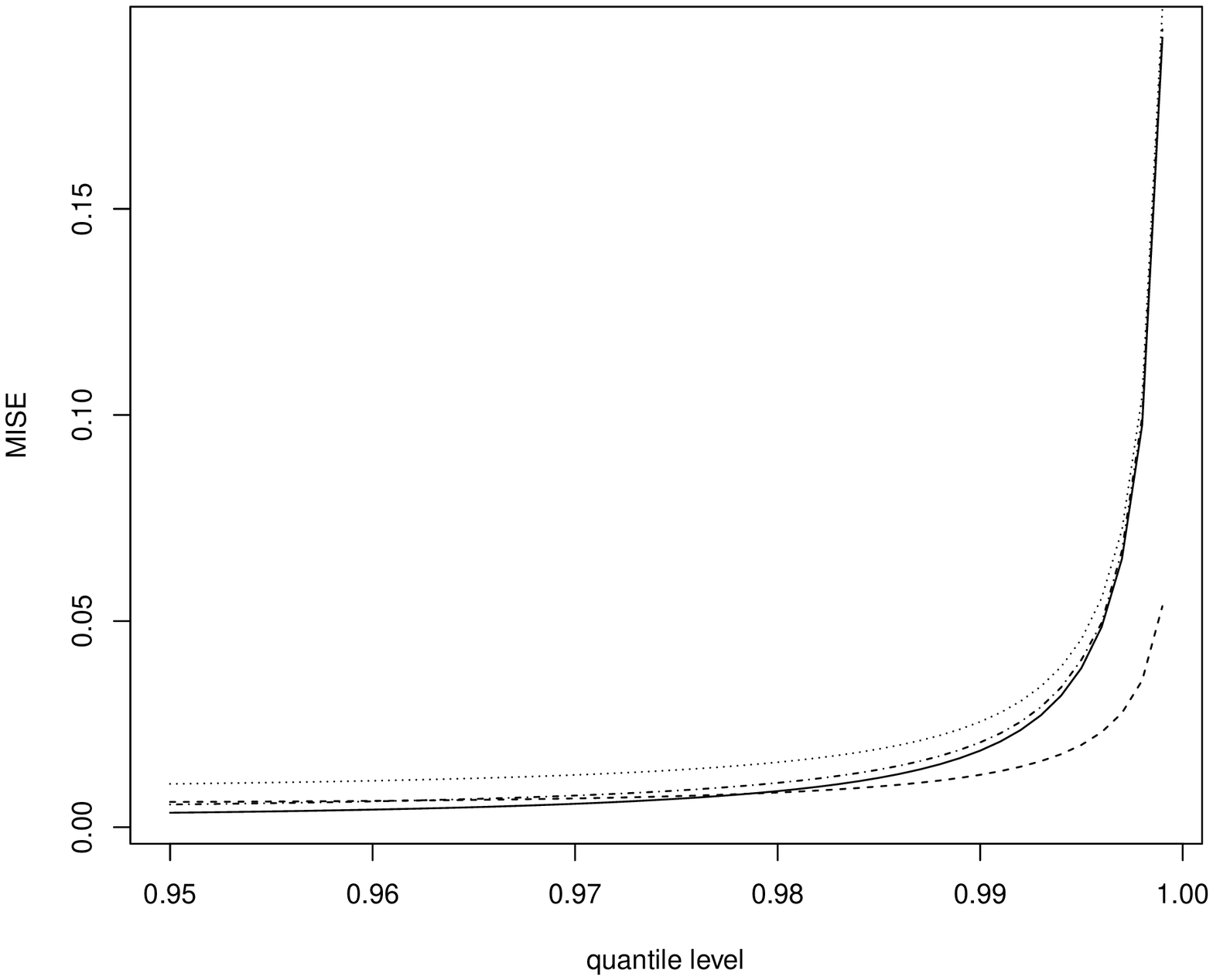}
}
\subfloat[MISE, $n=1000$]{
\includegraphics[width=50mm,height=45mm]{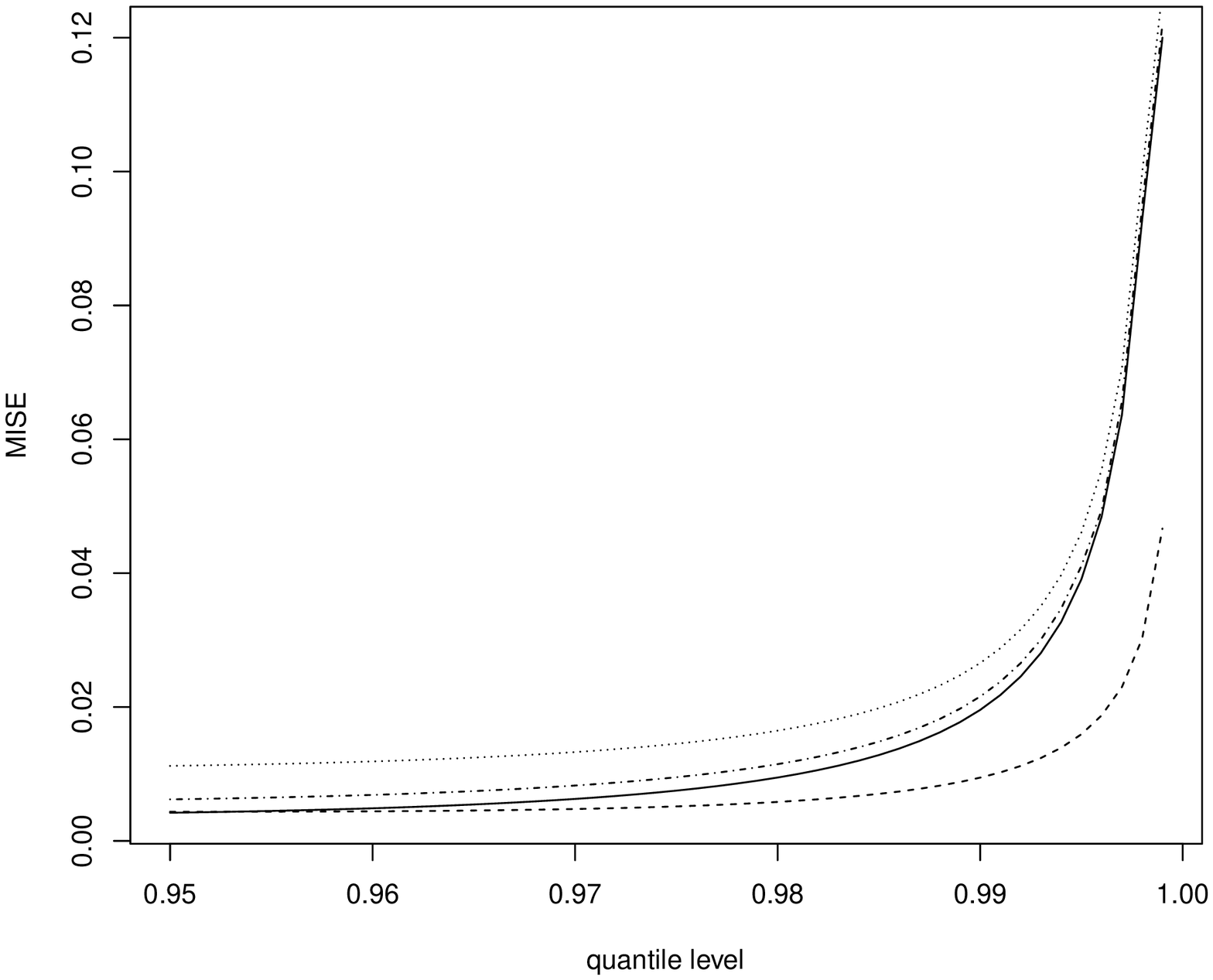}
}
\end{center}
\caption{(a--c): The true conditional quantiles (solid) and the extreme order quantile estimator PSE-E (dot-dashed) and PSE-Ep (dashed) for $\tau=0.995$ for one dataset. 
The dataset is similar to that given in Figure \ref{fig1} for each $n$. 
(d--f) MISE of the estimators for $\tau\in[0.95, 0.999]$. 
The solid, dashed, dotted and dot-dashed lines are PSE-E, PSE-Ep, FNS-E and KSE-E, respectively. \label{fig4} }
\end{figure} 

Figure \ref{fig4} shows the extreme order quantile estimators PSE-E and PSE-Ep for one dataset and the MISE of the extreme order quantile estimators for $\tau\in[0.95, 0.999]$.
From (a--c), we can observe that the estimator behaves well. 
We can see that the behavior of the PSE-Ep is stable than the PSE-E. 
This is not a surprising result since the estimator of EVI included in PSE-Ep is not dependent on $x$ unlike PSE-E.
It can be recognized from Figure \ref{fig4} (d--f) that the proposed estimator has better behavior than the competitors although the differences are not large. 
Furthermore, the performance of the PSE-Ep was superior to that of PSE-E. 
We think that this is a result of the stability of $\hat{\gamma}^{C}$. 
It can be recognized from Figure \ref{fig5} that the extrapolated estimator has consistency.

\begin{figure}
\begin{center}
\subfloat[MISE, PSE-E]{
\includegraphics[width=65mm,height=50mm]{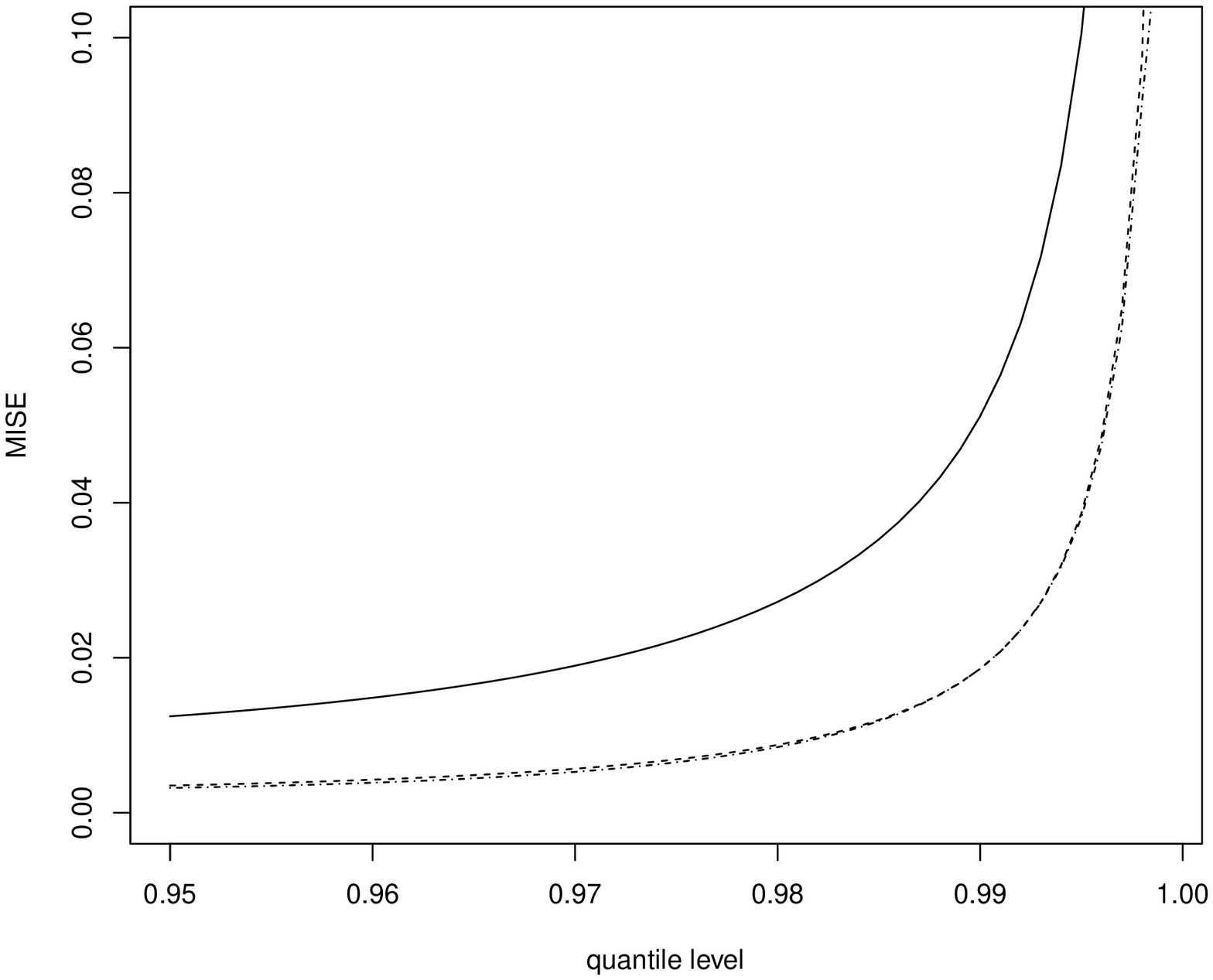}
}
\subfloat[MISE, PSE-Ep]{
\includegraphics[width=65mm,height=50mm]{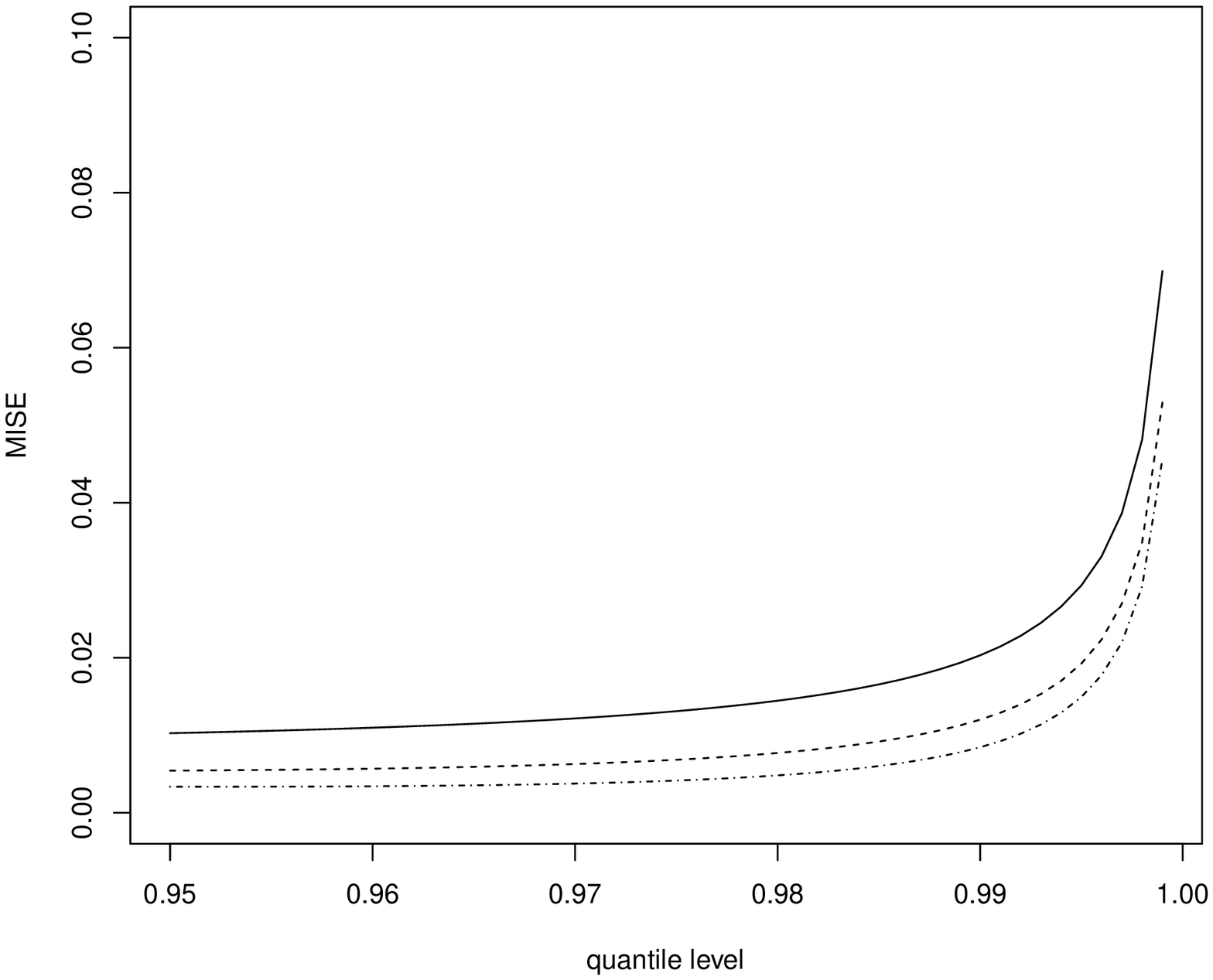}
}
\end{center}
\caption{(a) and (b) are the MISE of the PSE-E and the PSE-Ep, respectively. 
The results are similar to Figure \ref{fig4} (d--f).
For both panels, solid, dashed and dot-dashed lines are for $n=200$, 600 and 1000, respectively. 
\label{fig5}}
\end{figure}

From now on, we describe the simulation results for the model (b).
Figure \ref{fig6} shows the true conditional quantiles and the intermediate order quantile estimators for $\tau\in[0.8,0.995]$ for one dataset. 
It appears that for $\tau\in[0.8,0.9]$, the estimator can capture the true conditional quantile even for $n=200$. 
However, for $\tau\geq 0.95$, the estimator has a wiggly curve. 
In Figure \ref{fig7}, the results of MISE of the intermediate order quantile estimators for each $n$ are illustrated. 
We found that the proposed estimator performs well for $\tau\in[0.5,0.95)$. 
However, the MISE drastically grows as $\tau$ increases. 
The behaviors of PSE-E and PSE-Ep for one dataset are described in Figure \ref{fig8} (a--c).
It can be seen from Figures \ref{fig6} and \ref{fig8} (a--c) that the PSE-E and PSE-Ep performed better than PSE-I. 
Figure \ref{fig8} (d)--(f) shows the MISE of the extreme order quantile estimators. 
It can be confirmed that the performance of PSE-E is slightly better than that of PSE-Ep. 
We see that the proposed estimators have better behavior than the competitors. 
Figure \ref{fig9}, the consistency of the PSE-E and the PSE-Ep can be observed in numerically. 
Although the performance of the proposed estimator is drastically superior to that of Daouia et al. (2013), this simulation result indicates that our method is one of useful tools to the problem of extremal quantile regression.

\begin{figure}[h]
\begin{center}
\subfloat[True conditional quantiles]{
\includegraphics[width=65mm,height=35mm]{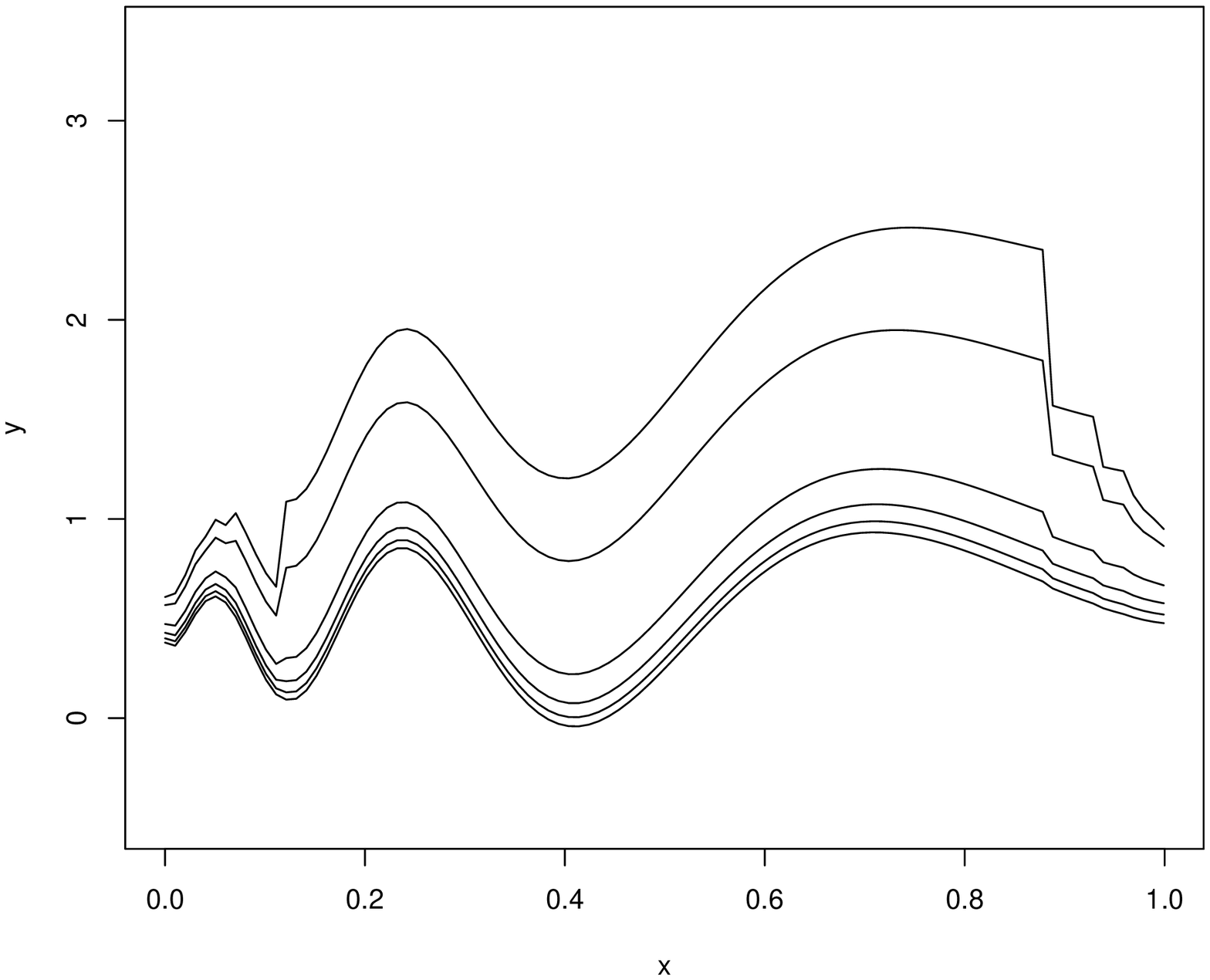}
}
\subfloat[$n=200$]{
\includegraphics[width=65mm,height=35mm]{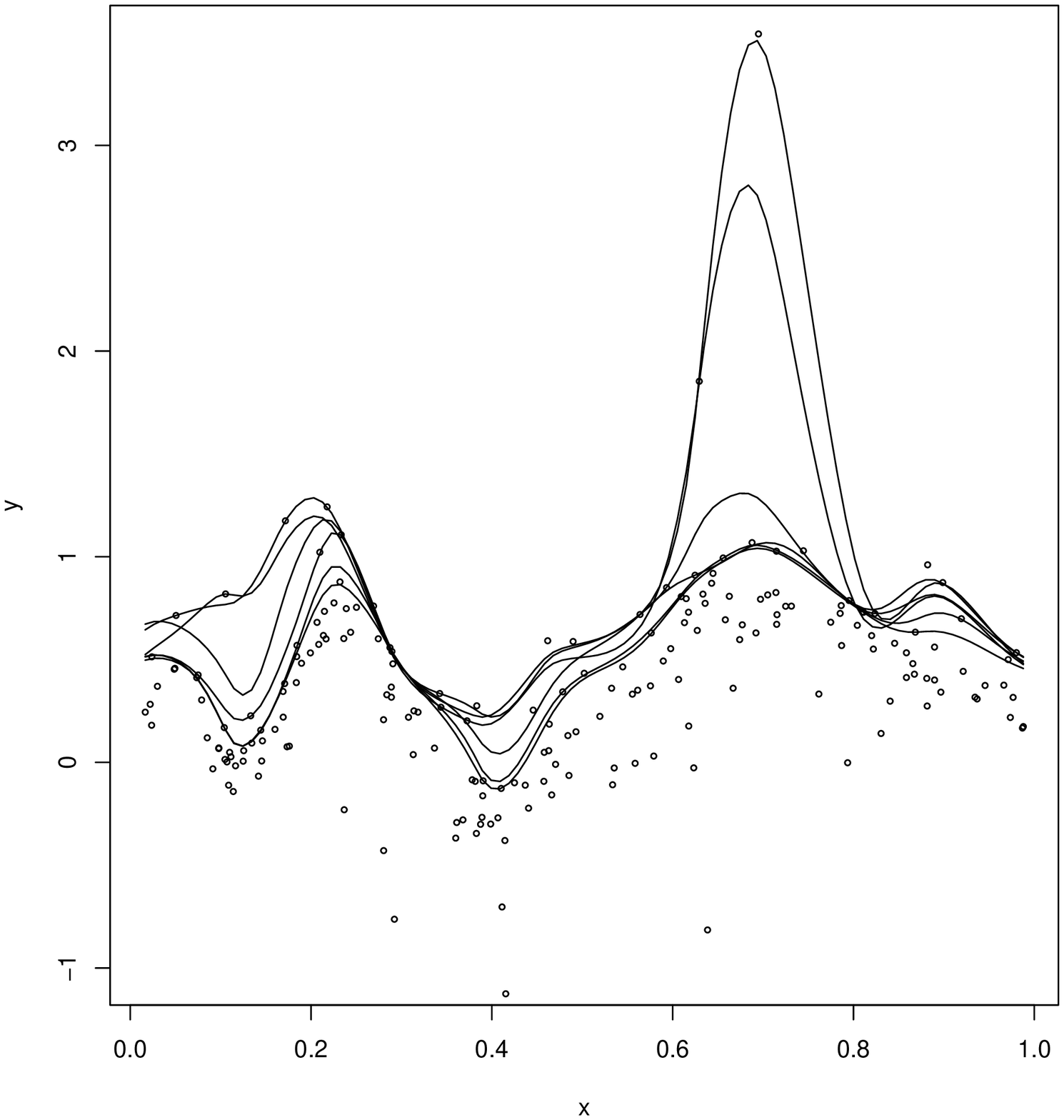}
}\\
\subfloat[$n=600$]{
\includegraphics[width=65mm,height=35mm]{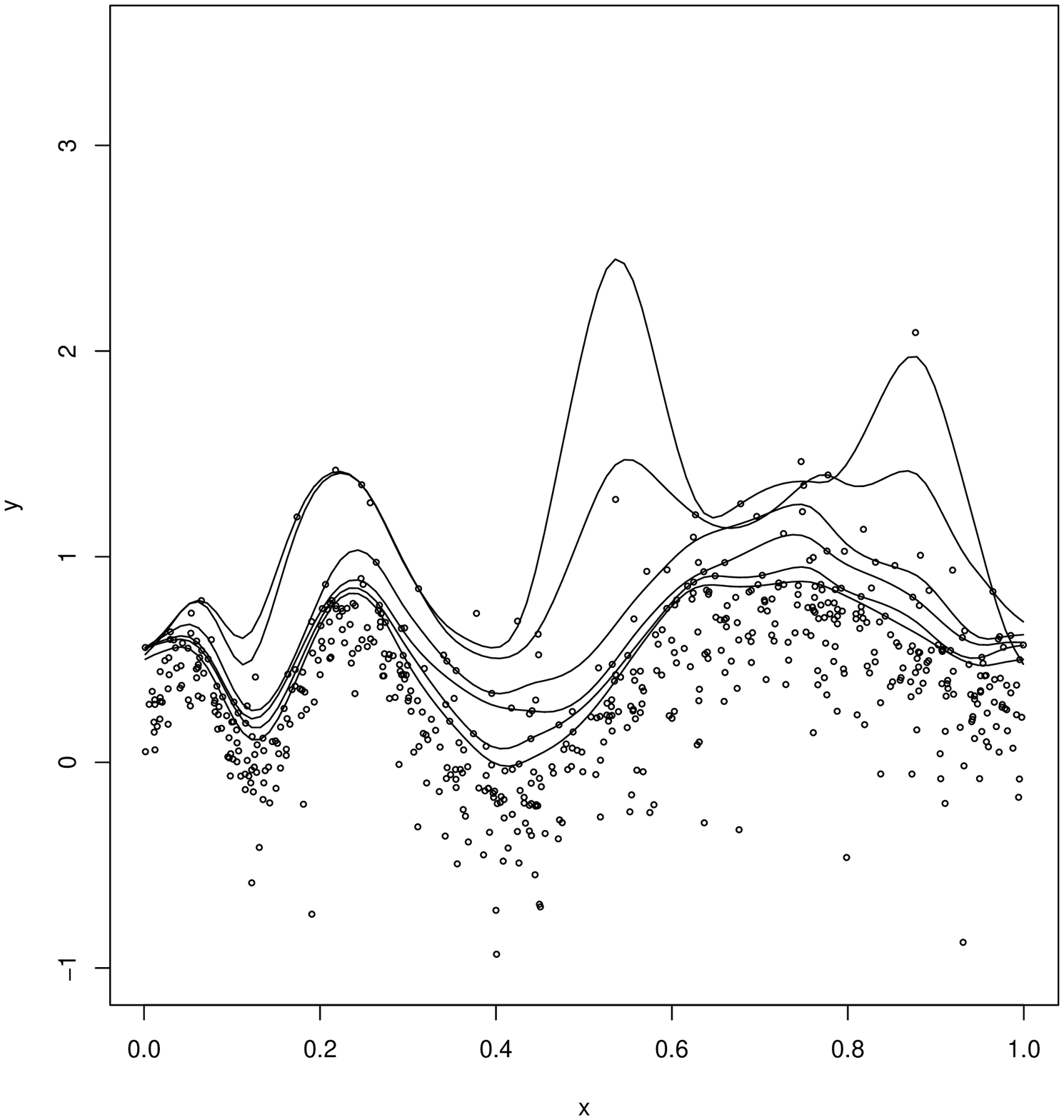}
}
\subfloat[$n=1000$]{
\includegraphics[width=65mm,height=35mm]{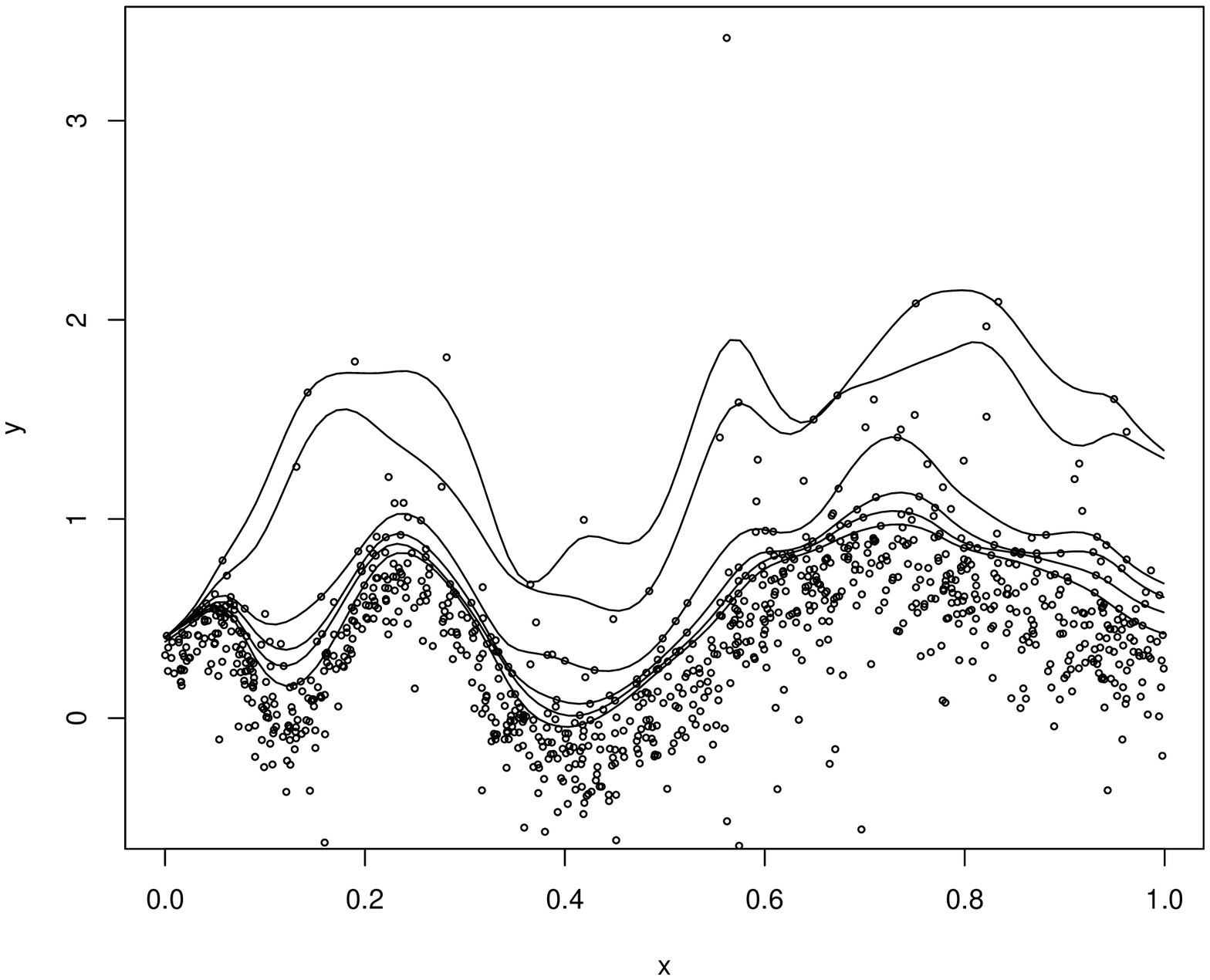}
}
\end{center}
\caption{True conditional quantiles for $\tau=0.8, 0.85, 0.9, 0.95, 0.99, 0.995$ and these intermediate order quantile estimators for one dataset with model (b). \label{fig6} }
\end{figure} 

\begin{figure}[h]
\begin{center}
\subfloat[$n=200$]{
\includegraphics[width=75mm,height=50mm]{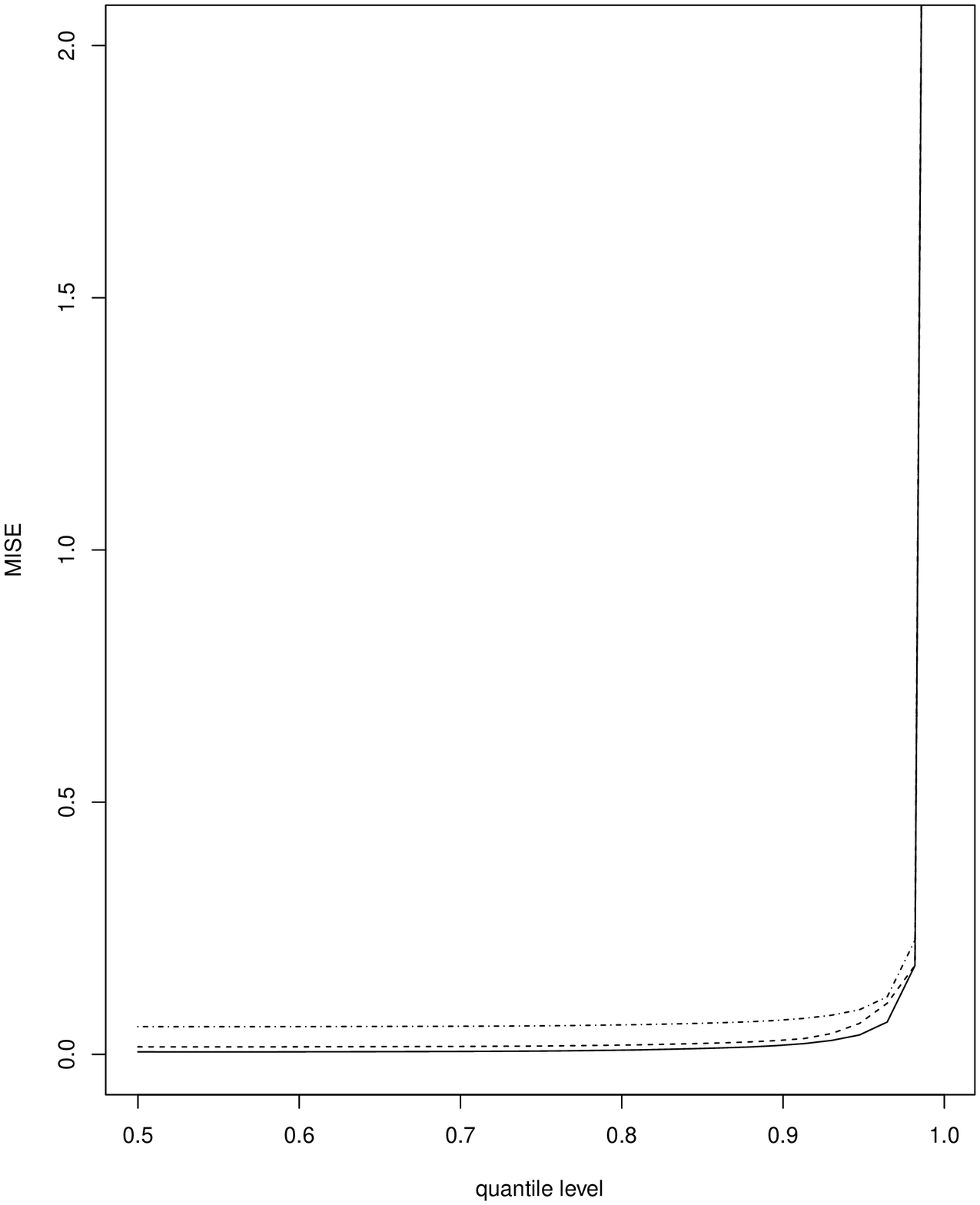}
}
\subfloat[$n=600$]{
\includegraphics[width=75mm,height=50mm]{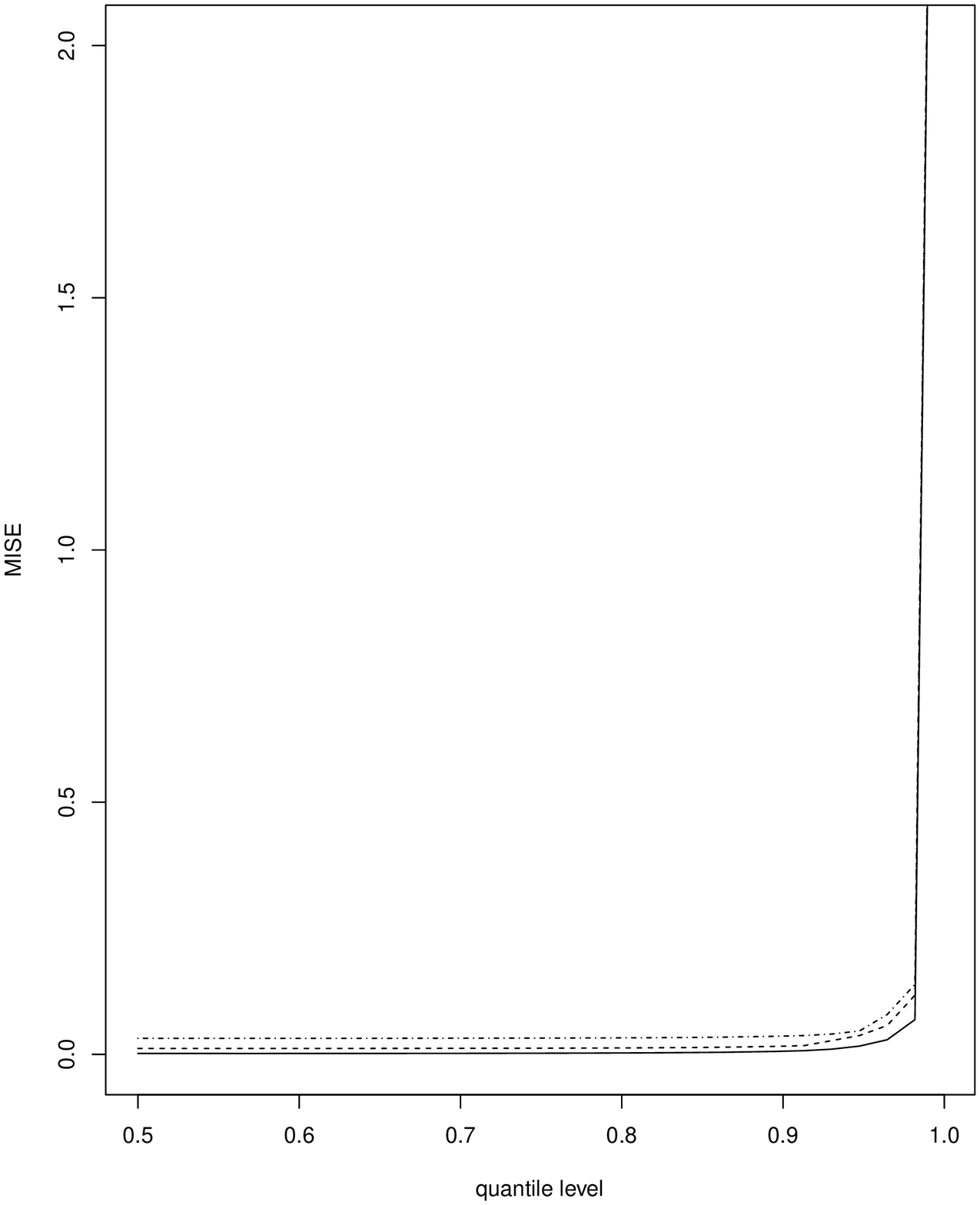}
}\\
\subfloat[$n=1000$]{
\includegraphics[width=75mm,height=50mm]{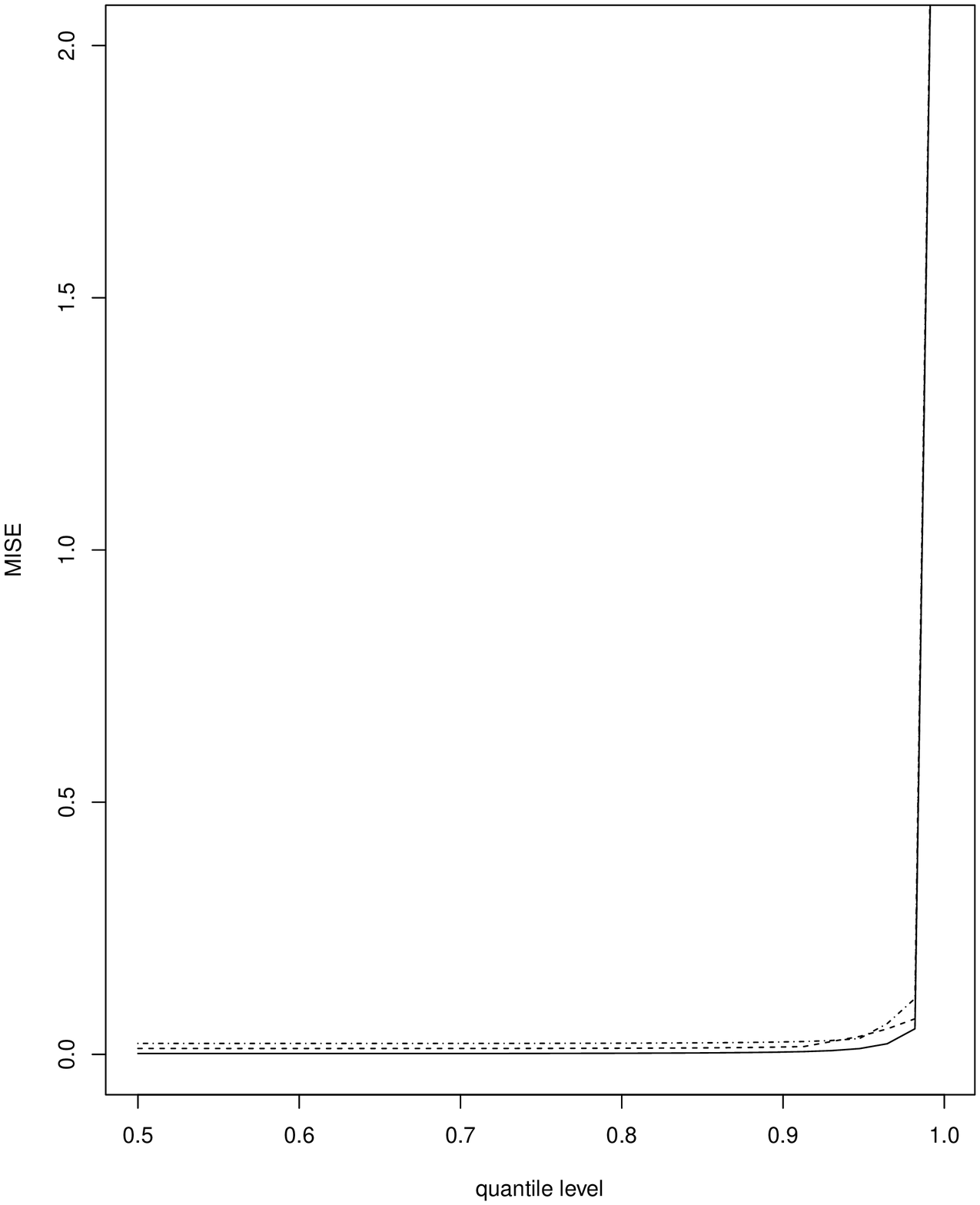}
}
\subfloat[PSE-I]{
\includegraphics[width=75mm,height=50mm]{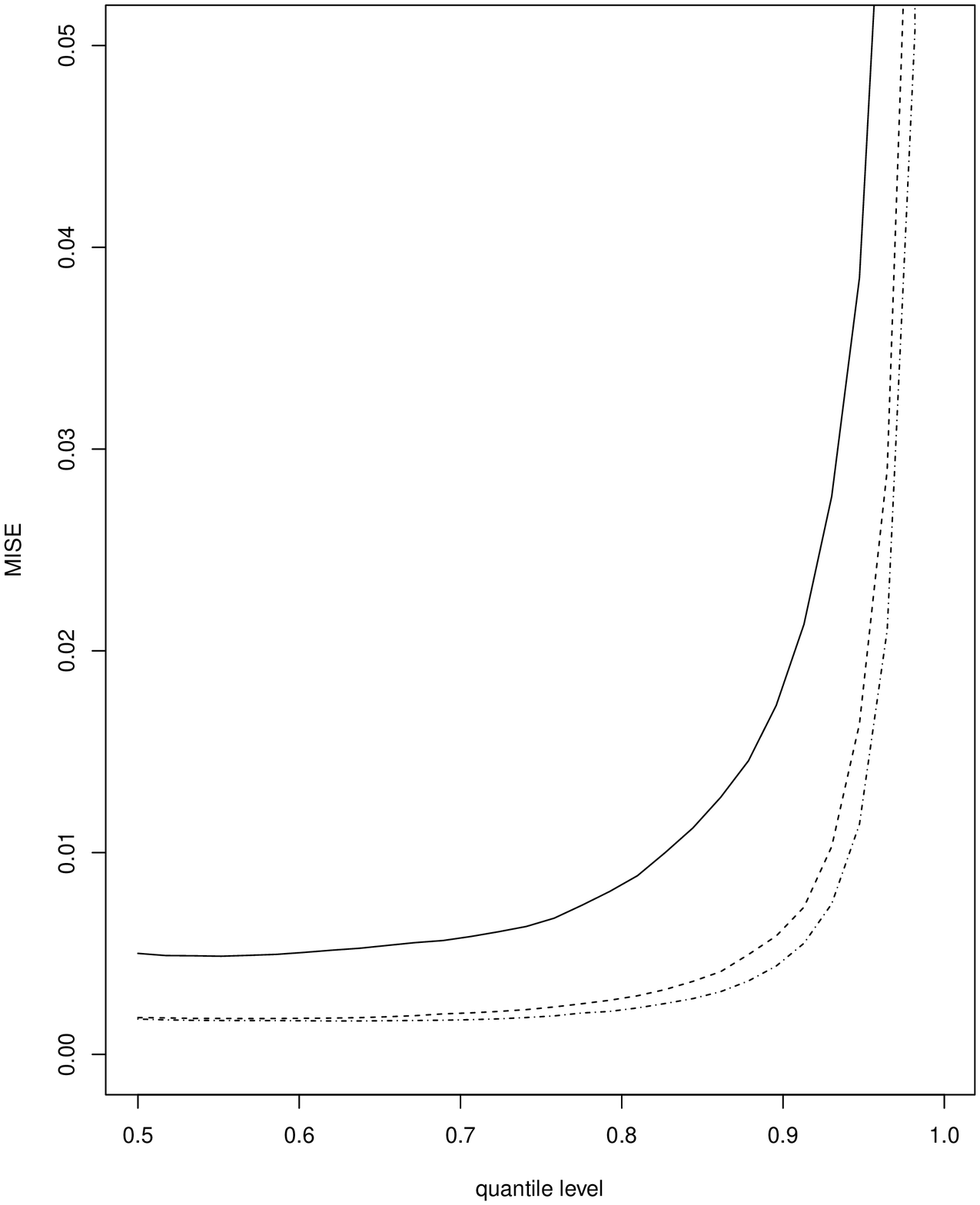}
}
\end{center}
\caption{MISE of the intermediate order quantile estimators for $\tau\in[0.5,0.995]$. The description is similar to Figure \ref{fig2}. \label{fig7} }
\end{figure}

\begin{figure}[h]
\begin{center}
\subfloat[$n=200$]{
\includegraphics[width=50mm,height=45mm]{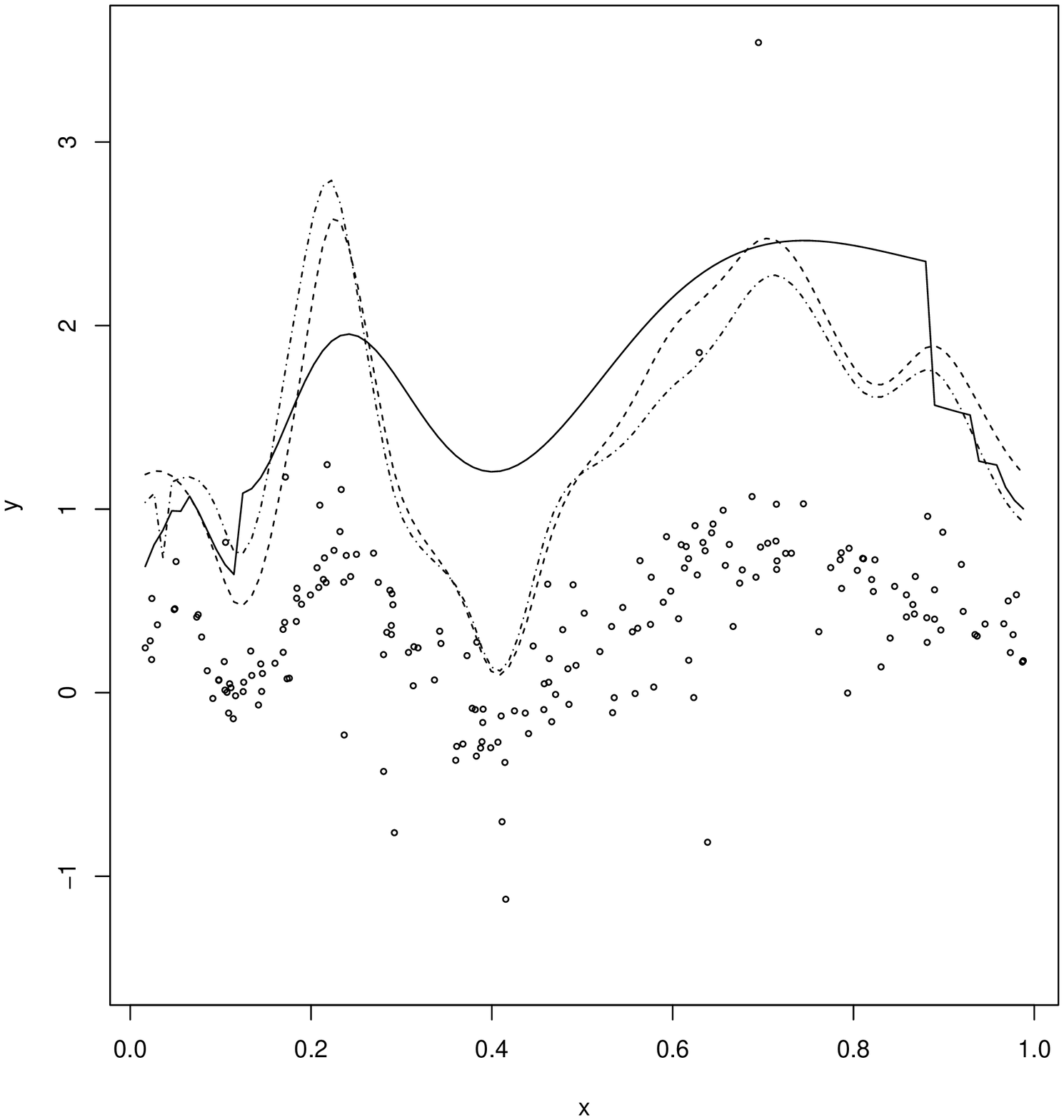}
}
\subfloat[$n=600$]{
\includegraphics[width=50mm,height=45mm]{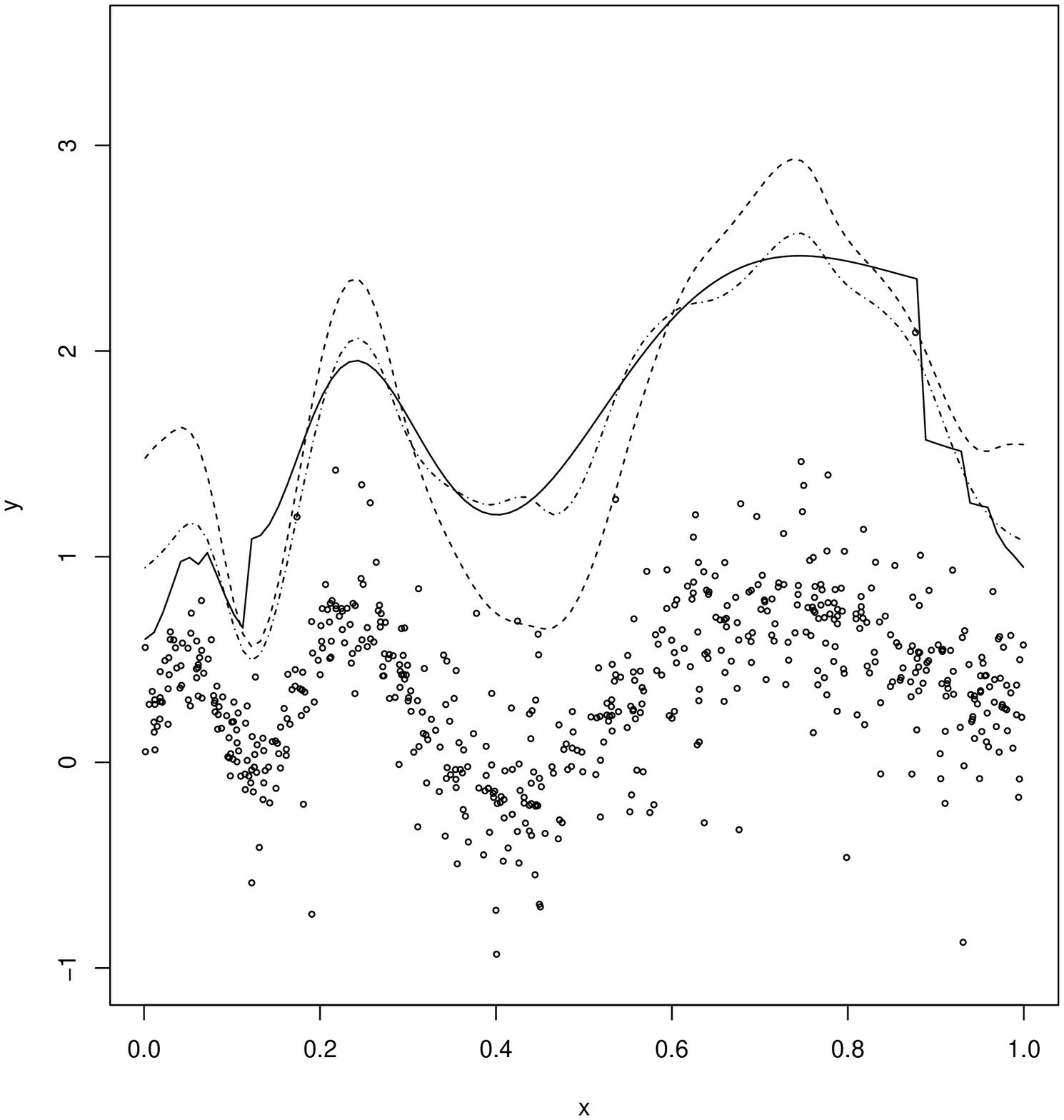}
}
\subfloat[$n=1000$]{
\includegraphics[width=50mm,height=45mm]{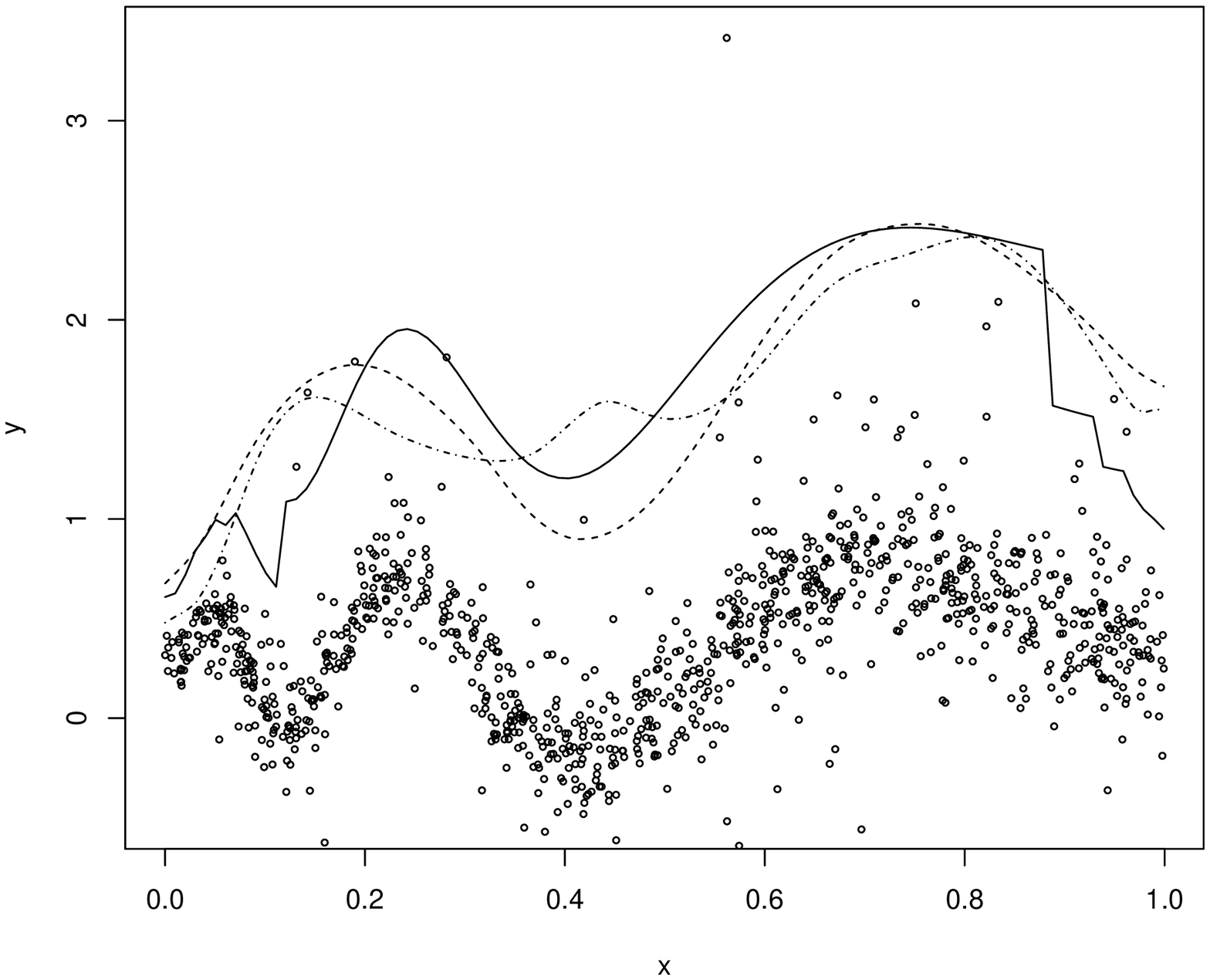}
}\\
\subfloat[MISE, $n=200$]{
\includegraphics[width=50mm,height=45mm]{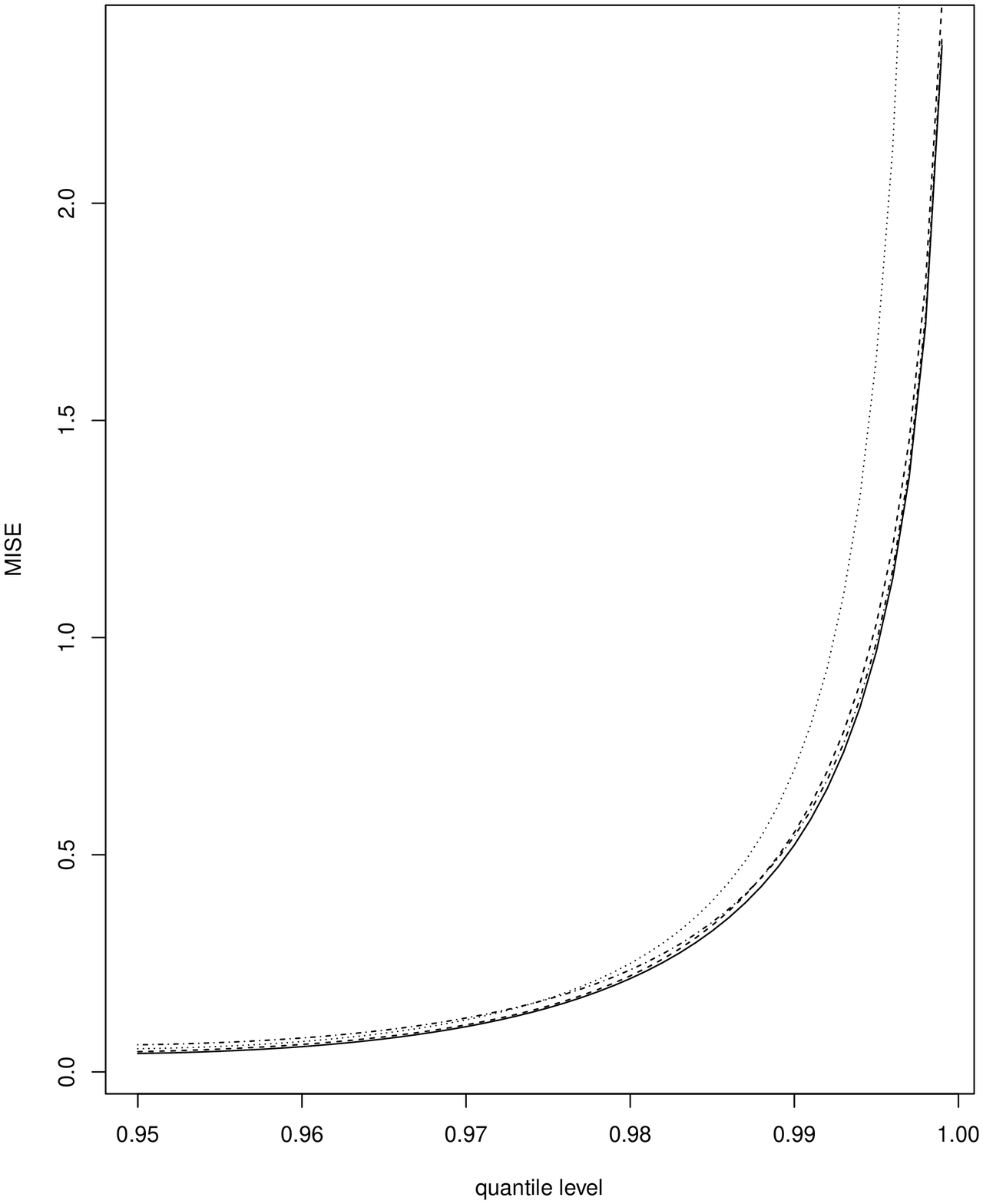}
}
\subfloat[MISE, $n=600$]{
\includegraphics[width=50mm,height=45mm]{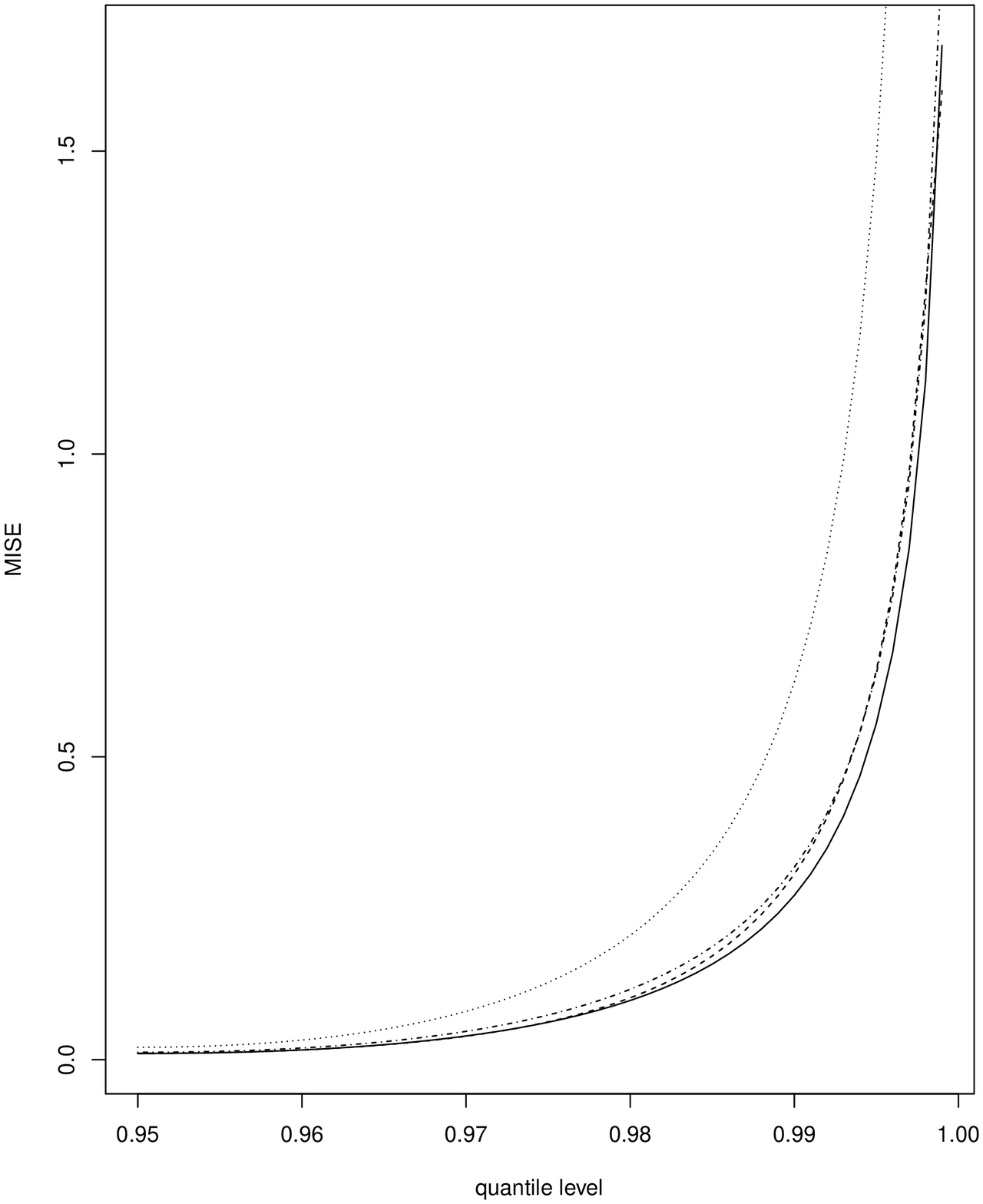}
}
\subfloat[MISE, $n=1000$]{
\includegraphics[width=50mm,height=45mm]{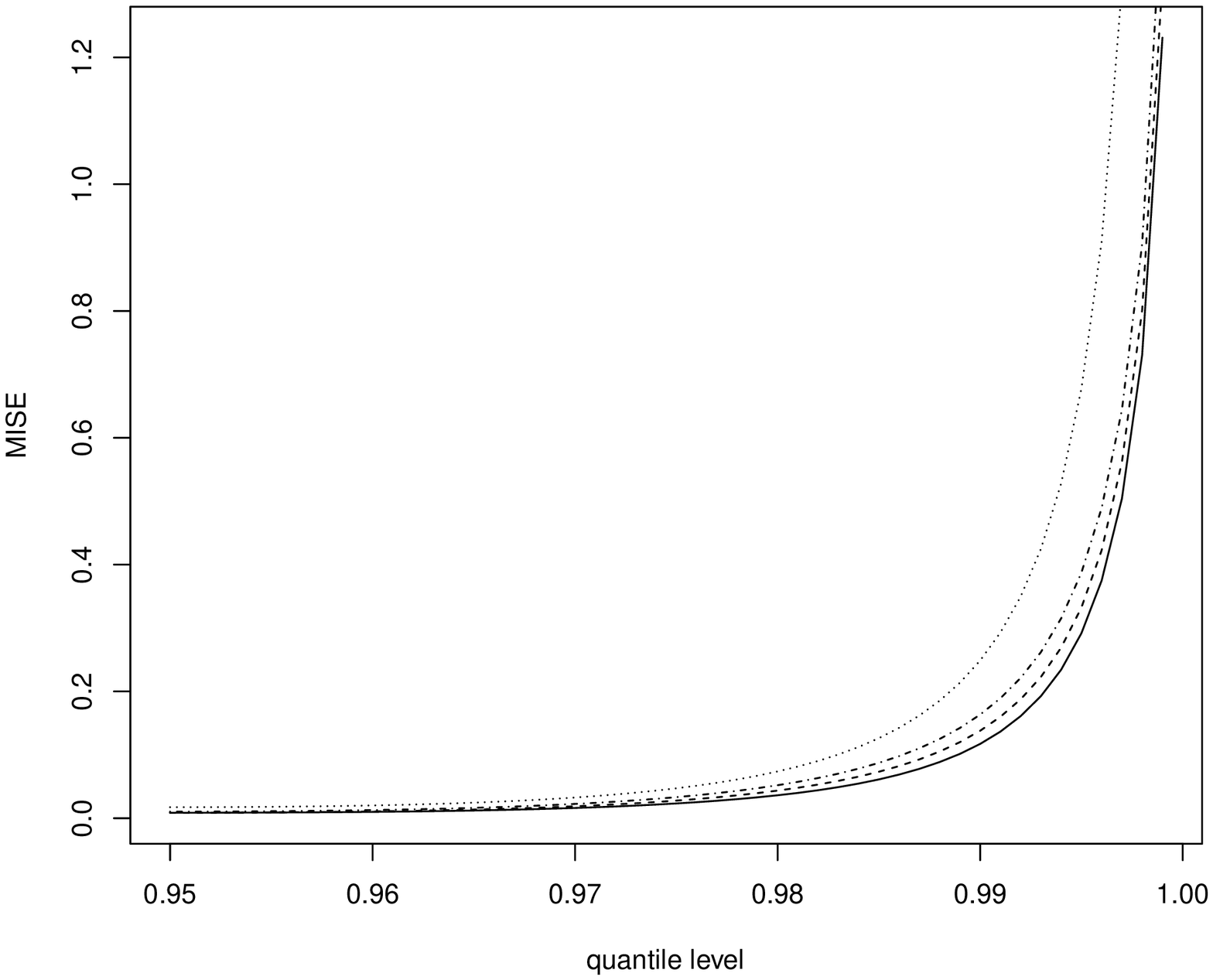}
}
\end{center}
\caption{(a--c): The true conditional quantiles the extreme order quantile estimators for $\tau=0.995$ for one dataset with model (b). 
The dataset is similar to that given in Figure \ref{fig6} for each $n$. 
(d--f) MISE of the estimators for $\tau\in[0.95, 0.999]$. 
The description is similar to Figure \ref{fig4} \label{fig8} }
\end{figure}

\begin{figure}[h]
\begin{center}
\subfloat[MISE, PSE-E]{
\includegraphics[width=70mm,height=50mm]{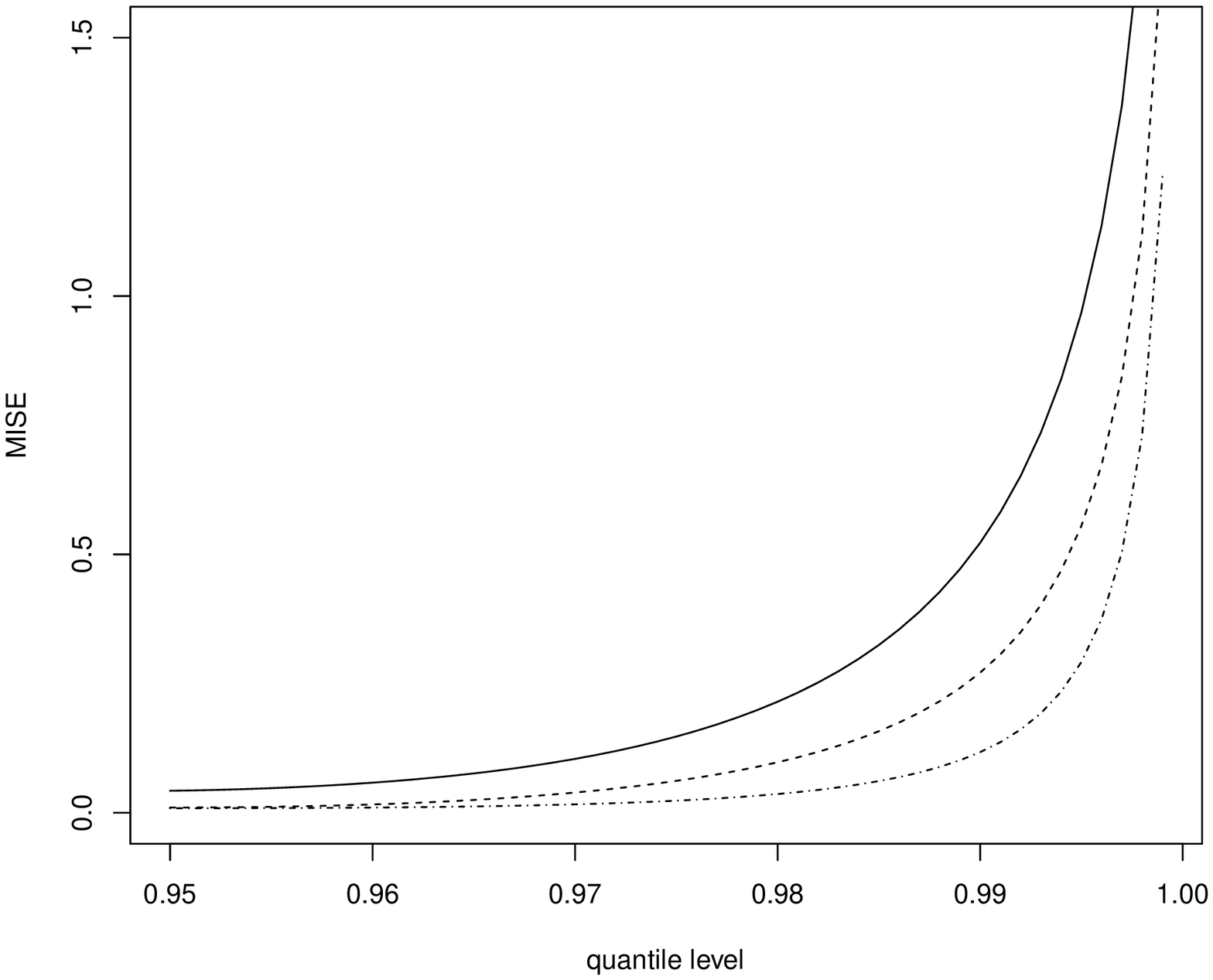}
}
\subfloat[MISE, PSE-Ep]{
\includegraphics[width=70mm,height=50mm]{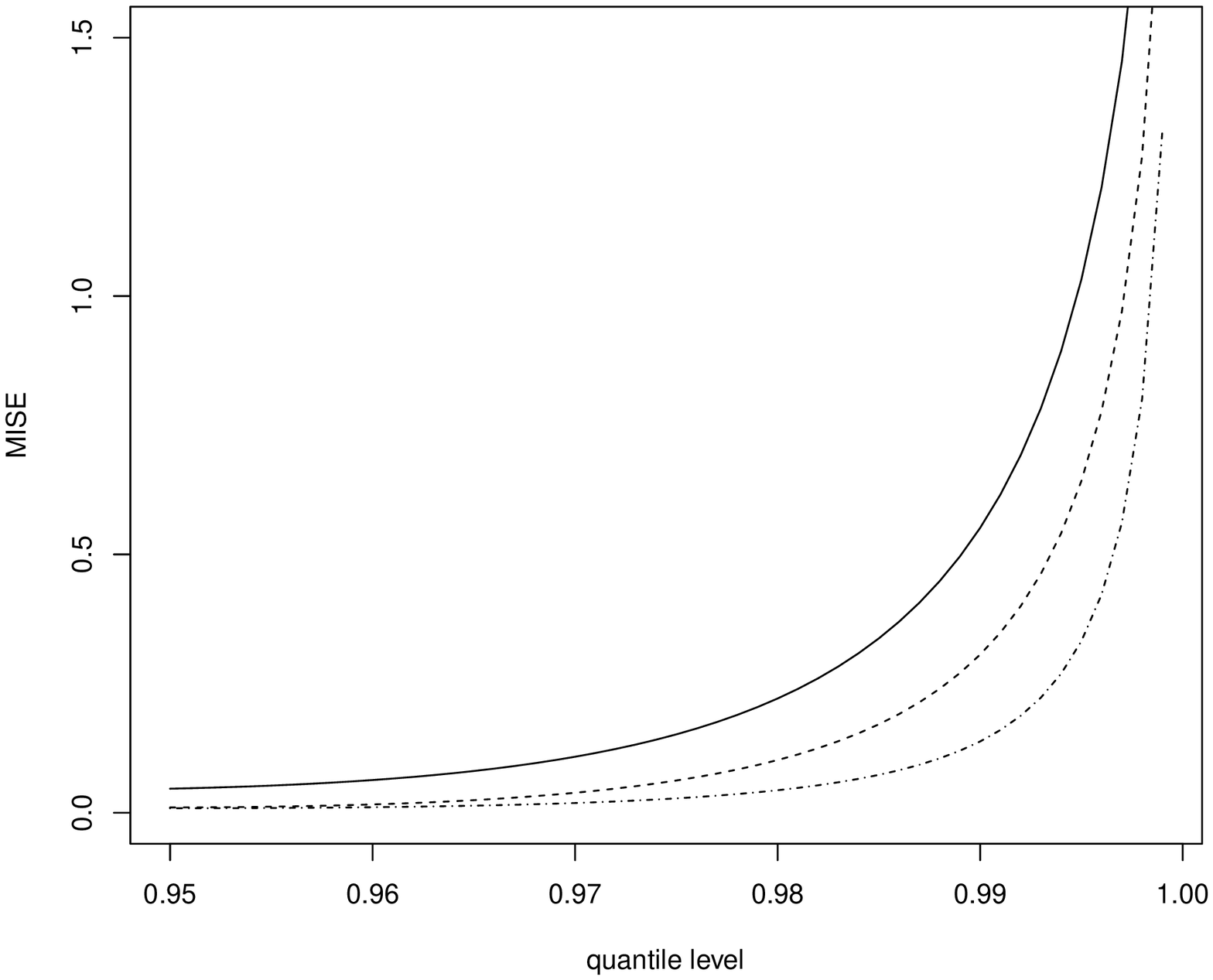}
}
\end{center}
\caption{(a) and (b) are the MISE of the PSE-E and the PSE-Ep, respectively. 
The description is similar to Figure \ref{fig5}
\label{fig9}}
\end{figure}

\section{Data example}

In this section, we apply the proposed methods to Beijing's ${\rm PM_{2.5}}$ Pollution data. 
The data is available from the website of the UCI Machine Learning Repository and Liang et al. (2016) provided several analyses for this data. 
One of fundamental purposes of this data is to analyze the relationship between ${\rm PM_{2.5}}$ concentration and other meteorological variables. 
Our particular interest here is the prediction of high conditional quantiles of $Y$, ${\rm PM_{2.5}}$ concentration ($\mu{\rm g/m^3}$), with the predictor $x$, temperature (degrees Celsius). 
We can observe from the scatter plot of $y$ and $x$ (see Figure \ref{fig10}) that this relationship is not linear for the upper quantile. 
Therefore, the nonparametric approach is suitable for this data.
We demonstrate the analysis for each year from 2011 to 2014. 
We then omit the missing data and hence the sample size is $n=8032$, 8295, $8678$, and 8661 in 2011, 2012, 2013, and 2014. 
We construct the extreme order quantile estimator for $\tau=0.999$. 
The quantile level $\tau=0.999$ indicates that about only eight events occur each year.

\begin{figure}[h]
\begin{center}
\subfloat[2011]{
\includegraphics[width=65mm,height=35mm]{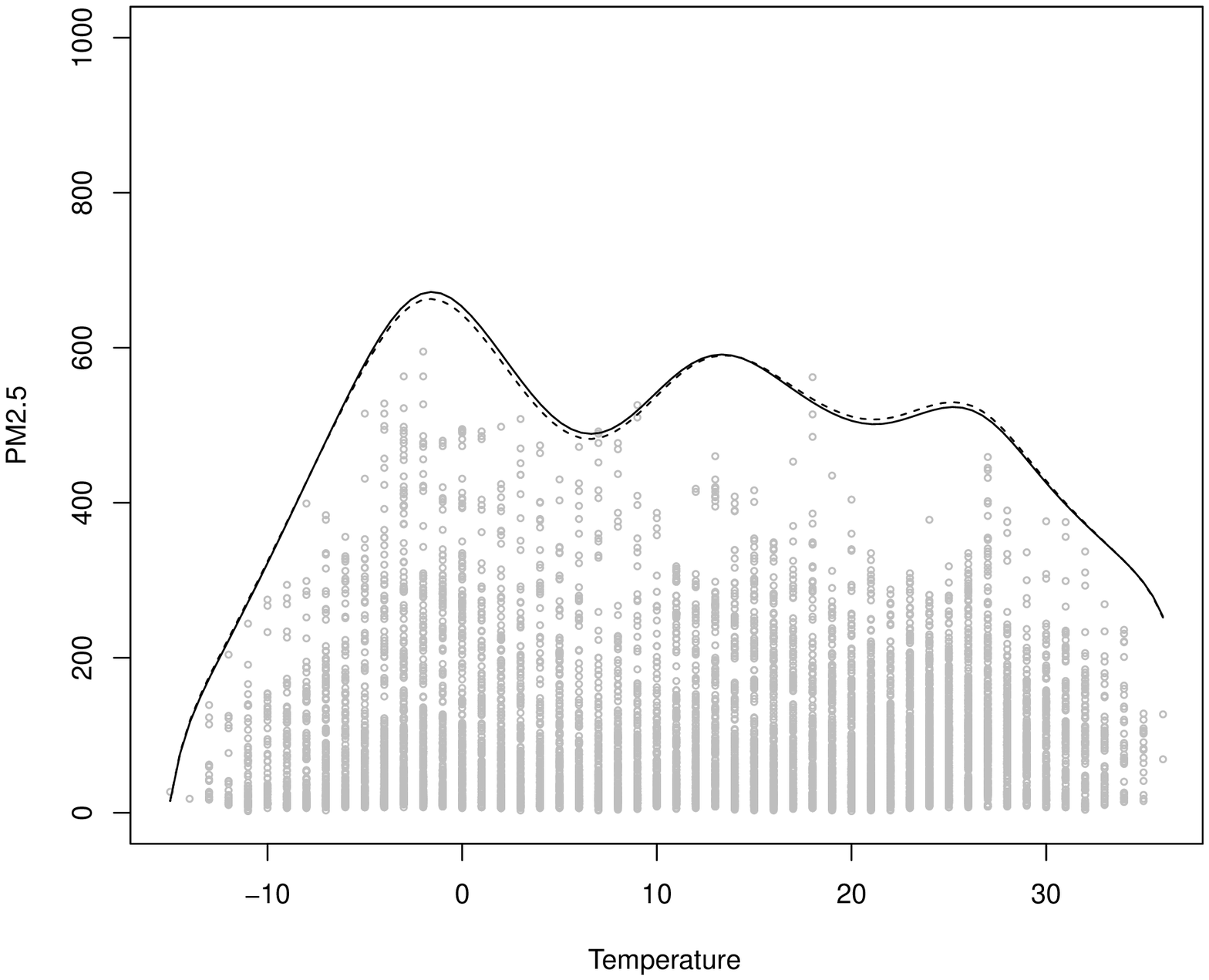}
}
\subfloat[2012]{
\includegraphics[width=65mm,height=35mm]{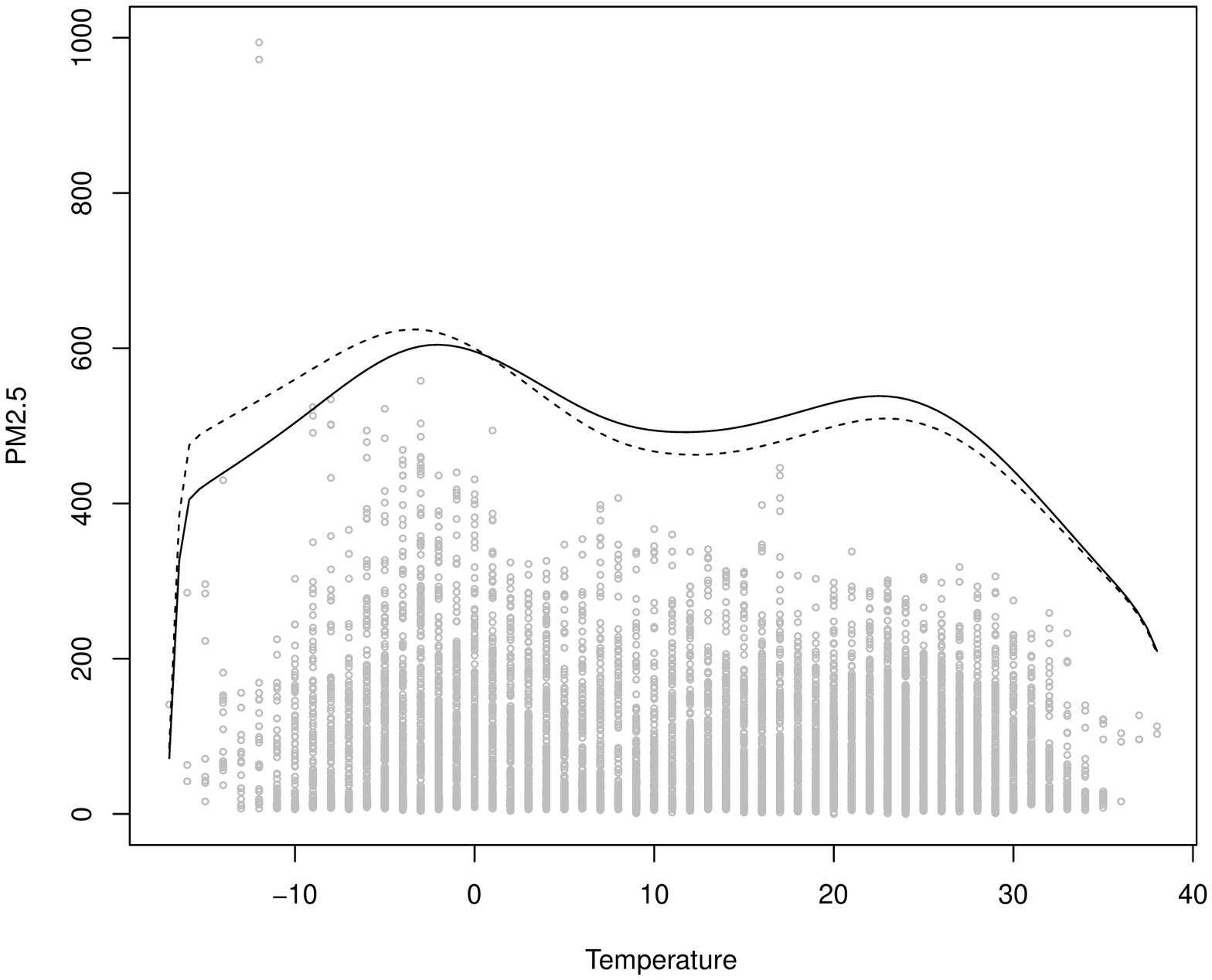}
}
\\
\subfloat[2013]{
\includegraphics[width=65mm,height=35mm]{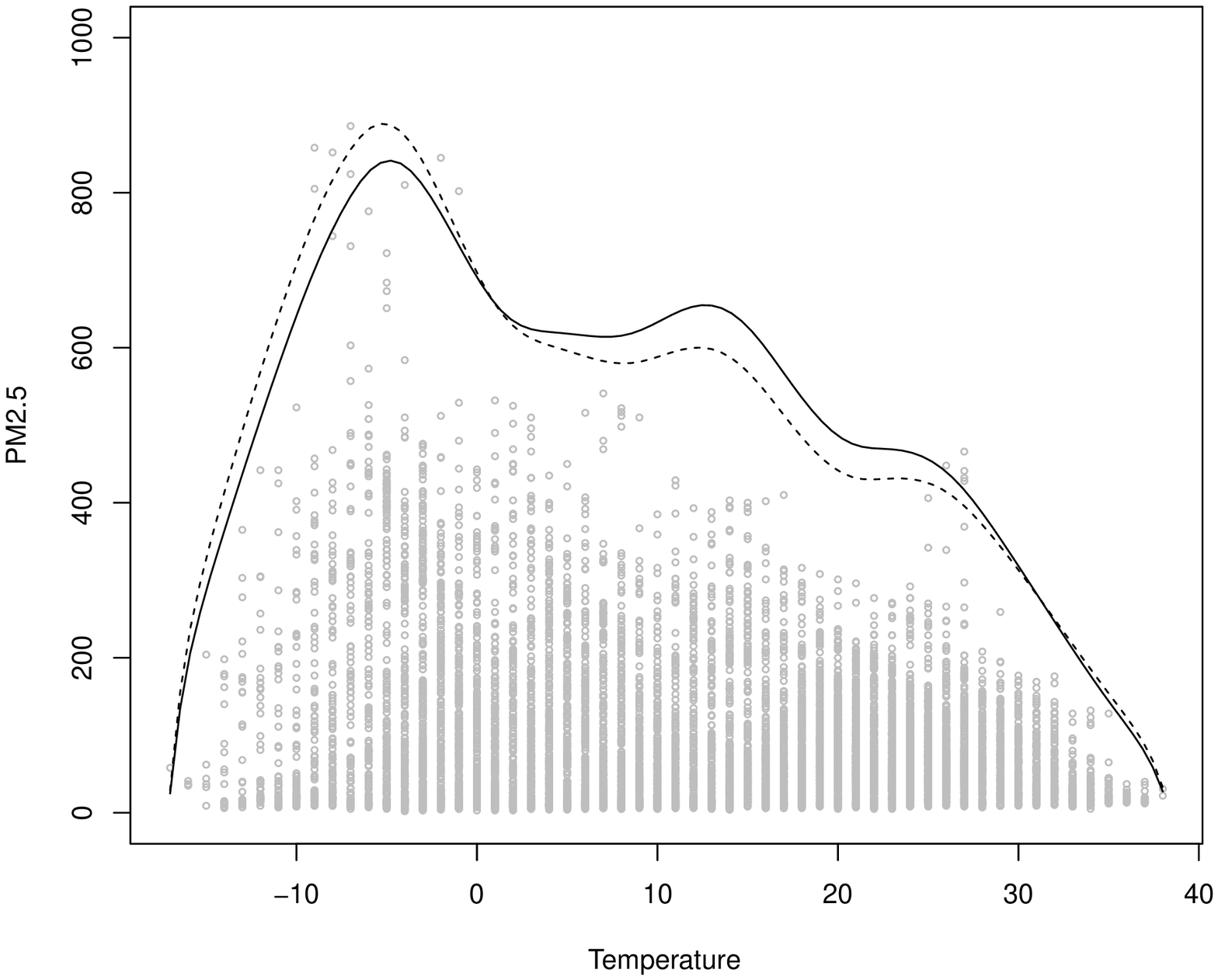}
}
\subfloat[2014]{
\includegraphics[width=65mm,height=35mm]{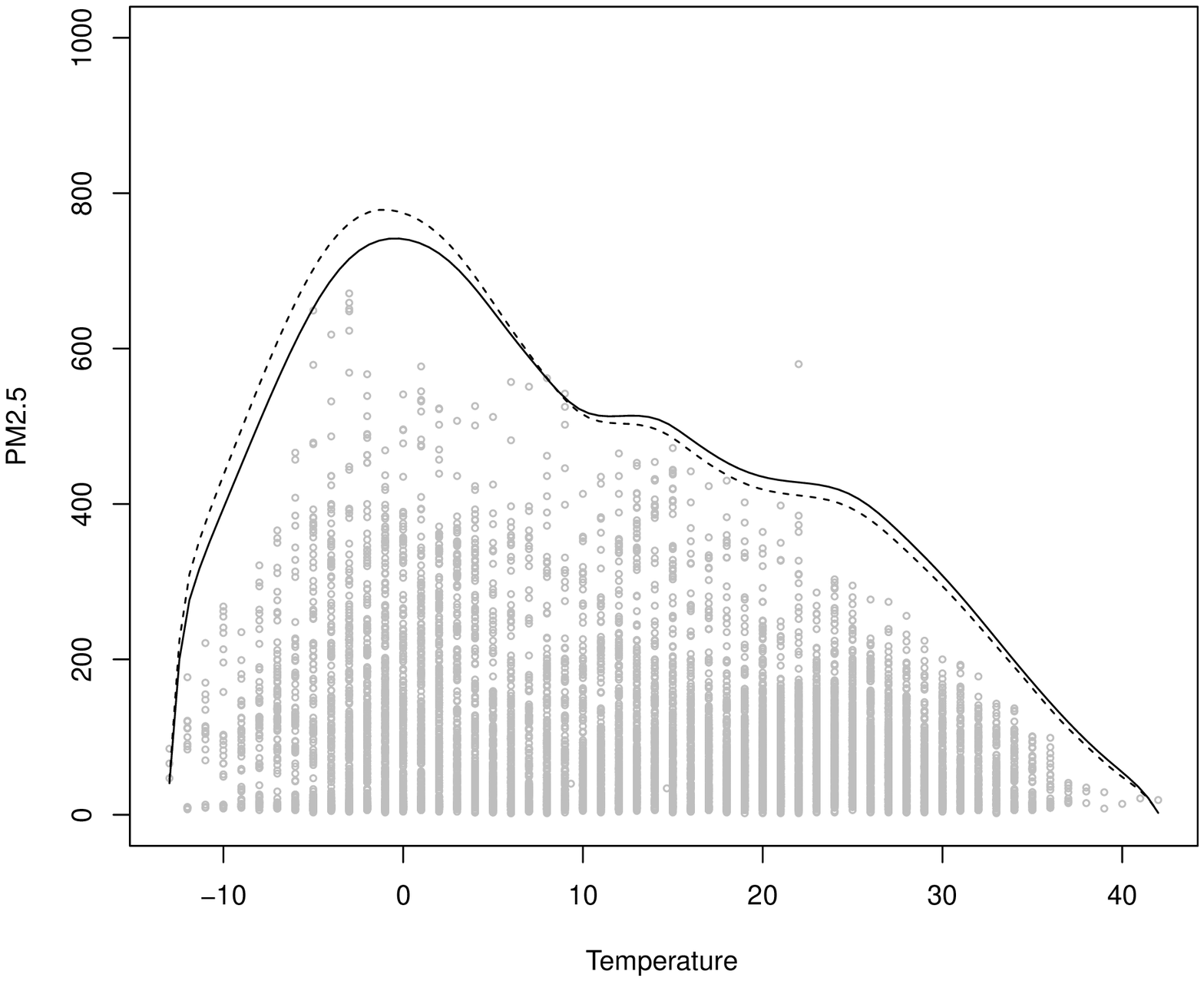}
}
\end{center}
\caption{The proposed extreme order quantile estimator for $\tau=0.999$ for years from 2011 to 2014. 
The solid line is $\hat{q}^{C}_Y(0.999|x)$ and the dashed line is $\hat{q}_Y(0.999|x)$ for each panel.
\label{fig10}}
\end{figure}

Figure \ref{fig10} shows the extreme order quantile estimators for $\tau=0.999$. 
It seems that the tail behavior is stable in 2011 and 2012.
In 2013 and 2014, the estimator of the conditional quantile at $x\in [0,10]$ has a large value compared with those in 2011 and 2012. 
Thus, in the cold season of 2013 and 2014, the risk of ${\rm PM}_{2.5}$ of pollution was increased. 
We can observe that $\hat{q}_Y(0.999|x)$ and $\hat{q}^{C}_Y(0.999|x)$ are quite similar. 
This indicates that the estimator of EVI $\hat{\gamma}(x)$ hardly changes with $x$. 
We report the EVI estimator used for constructing the extrapolated estimator $\hat{q}_Y$ and $\hat{q}_Y^{pool}$.
To obtain the estimator of EVI, we utilized $\eta=0.4$ as $\tau_1=n-[n^\eta]/(n+1)$ so that about $\tau_1=0.995$.
Then, EVI is estimated by using the intermediate order quantiles estimators for $\tau\in(0.978,0.995)$.
This choice leads to $\xi=n(1-\tau_1)\approx 37$ and, hence, this is a very conservative situation in the study of Chernozhukov and Fern\'andez-val (2011). 
Furthermore, we then adopted to use $k$ so that the sample path of the pooled EVI estimator is stable in each year.  
Figure \ref{fig11} shows the sample path of $\hat{\gamma}^{C}$ and selected $k$.
As a result, $\hat{\gamma}=0.220$, 0.226, $0.260$, and 0.231 in 2011, 2012, 2013, and 2014. 
Thus, the pooled EVI estimators are the same in 2011, 2012, and 2014, and it is only slightly larger in 2013.

\begin{figure}[h]
\begin{center}
\subfloat[2011]{
\includegraphics[width=65mm,height=35mm]{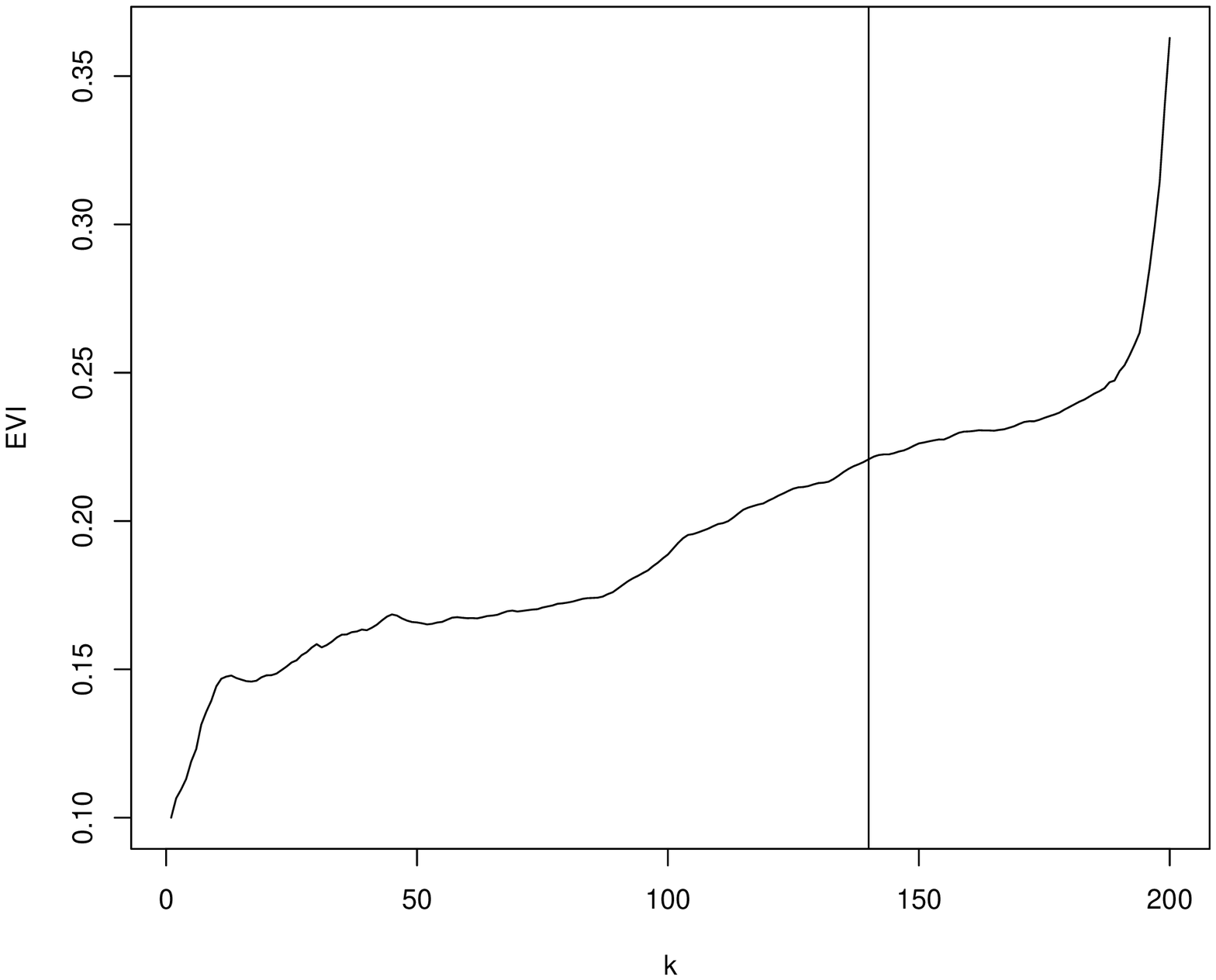}
}
\subfloat[2012]{
\includegraphics[width=65mm,height=35mm]{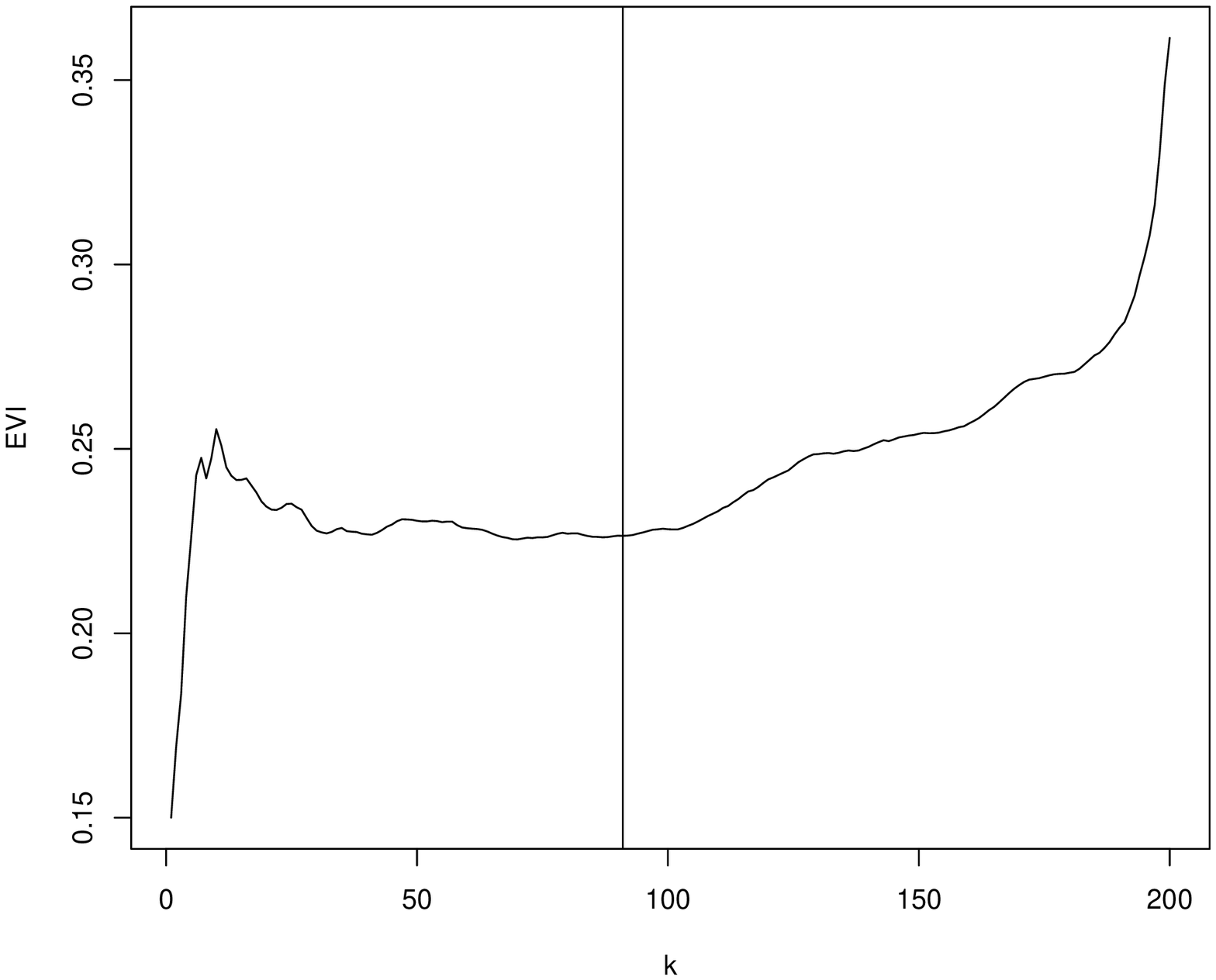}
}
\\
\subfloat[2013]{
\includegraphics[width=65mm,height=35mm]{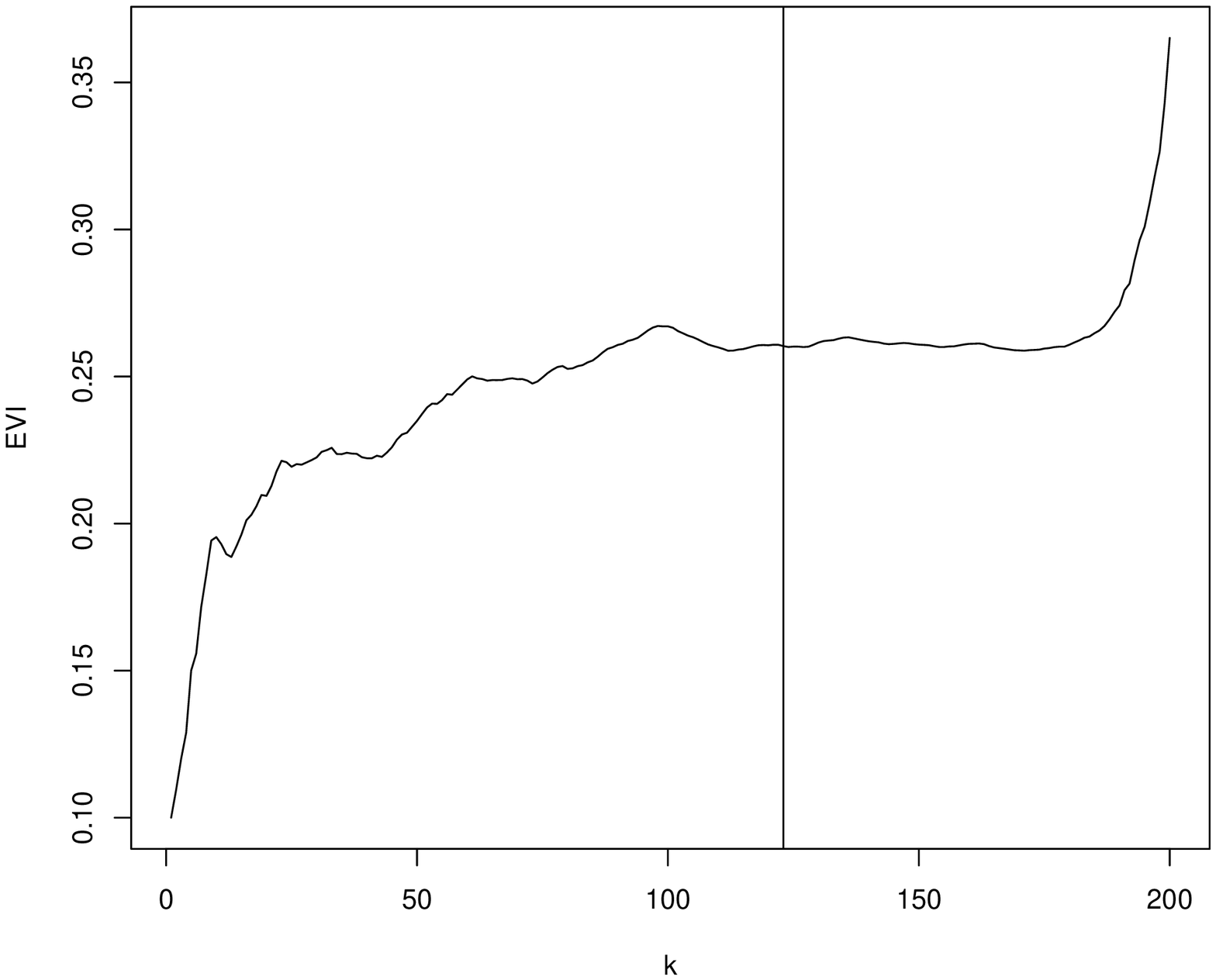}
}
\subfloat[2014]{
\includegraphics[width=65mm,height=35mm]{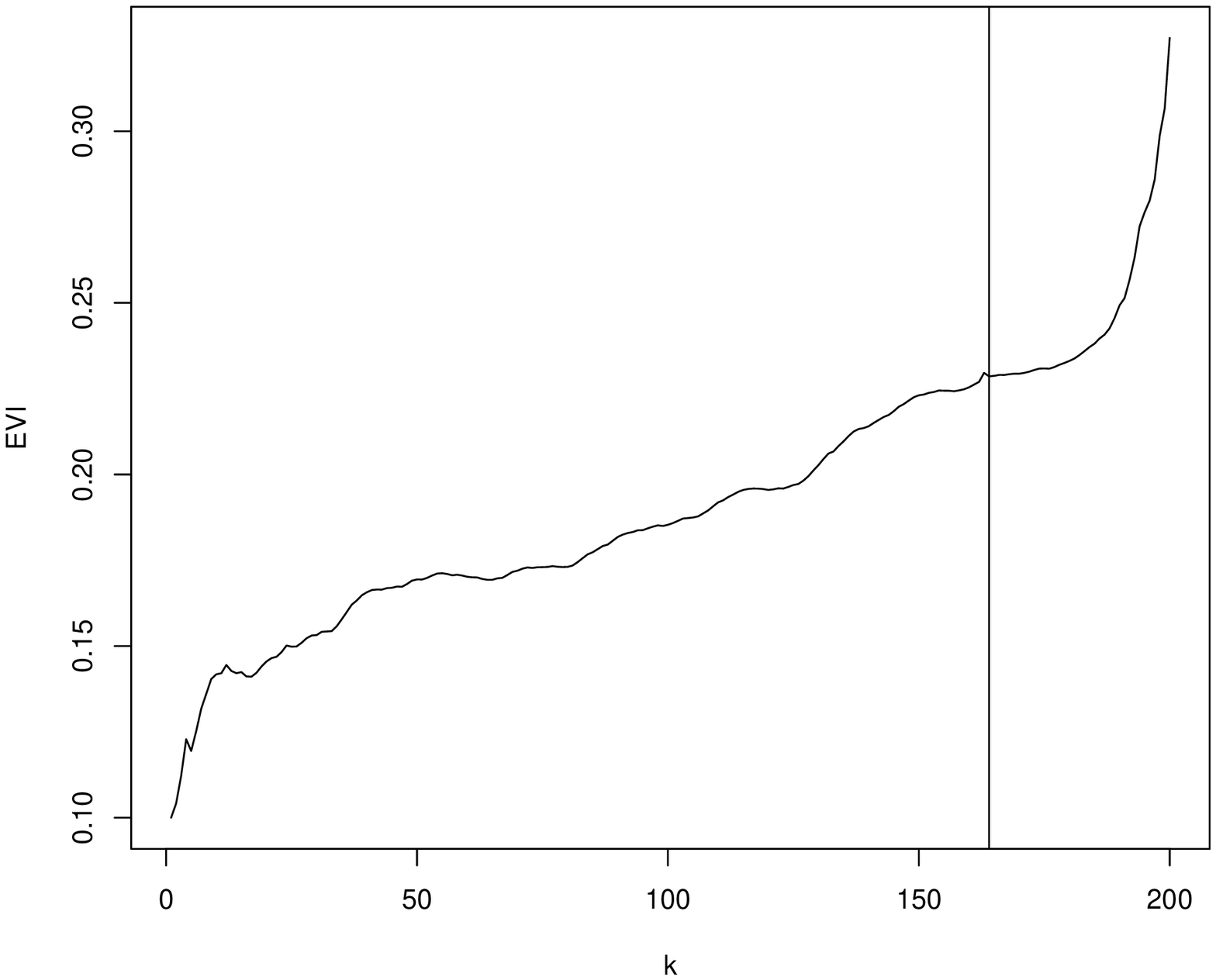}
}
\end{center}
\caption{The sample path of EVI estimates $\hat{\gamma}^{C}$ with $k$ for each year.
\label{fig11}}
\end{figure}

Figure \ref{fig12} illustrates $\hat{\gamma}(x)$ with selected $k$ for each year. 
We can observe that $\hat{\gamma}(x)$ has a narrow curve with $x$ in 2011, 2012, and 2014. 
On the other hand, in 2013, $\hat{\gamma}(x)$ at the boundary is rather larger than at the center.
Indeed, it can be seen from Figure \ref{fig10} that the extreme point can be observed at $x < 0$ in 2013.

\begin{figure}[h]
\begin{center}
\subfloat[2011]{
\includegraphics[width=65mm,height=35mm]{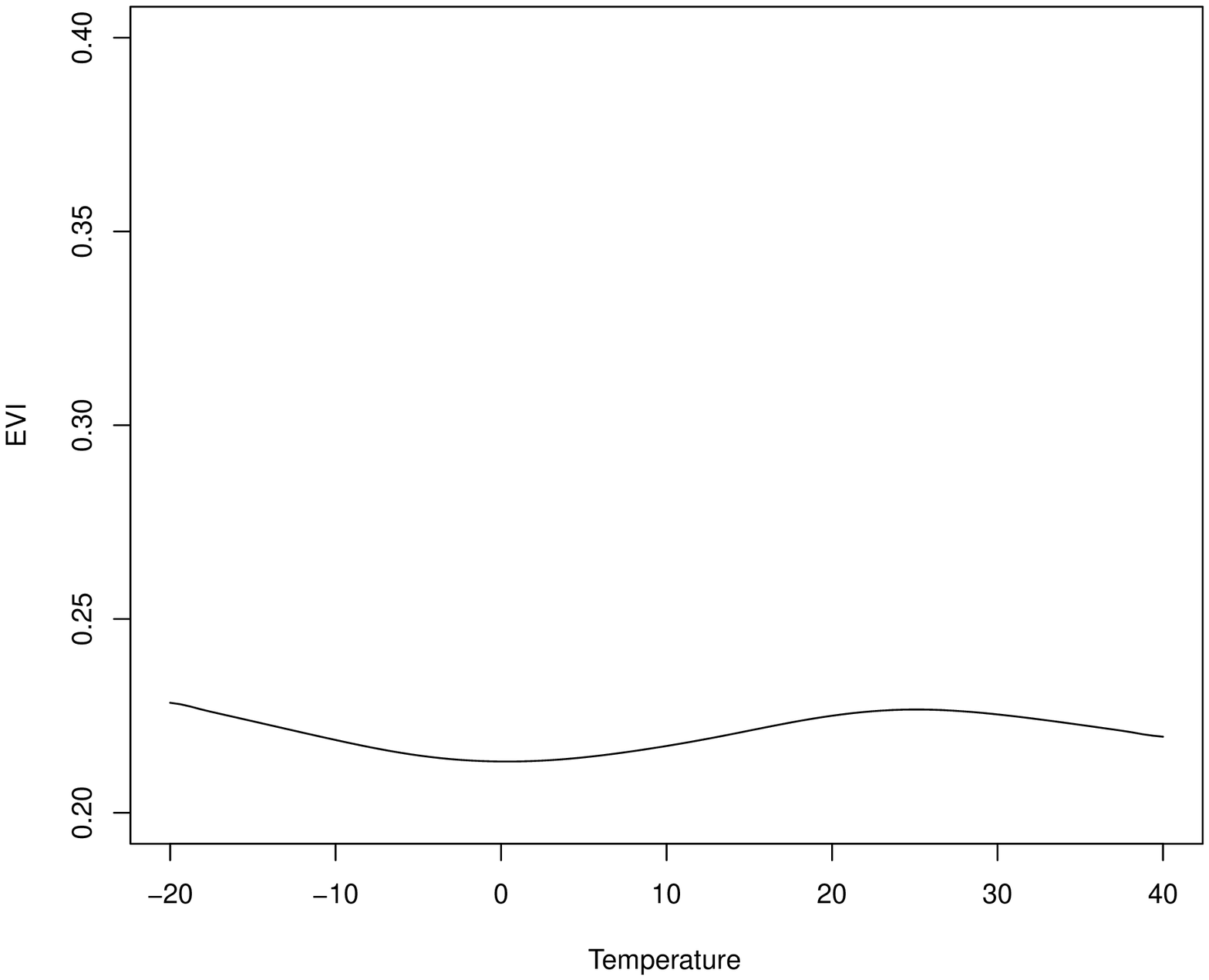}
}
\subfloat[2012]{
\includegraphics[width=65mm,height=35mm]{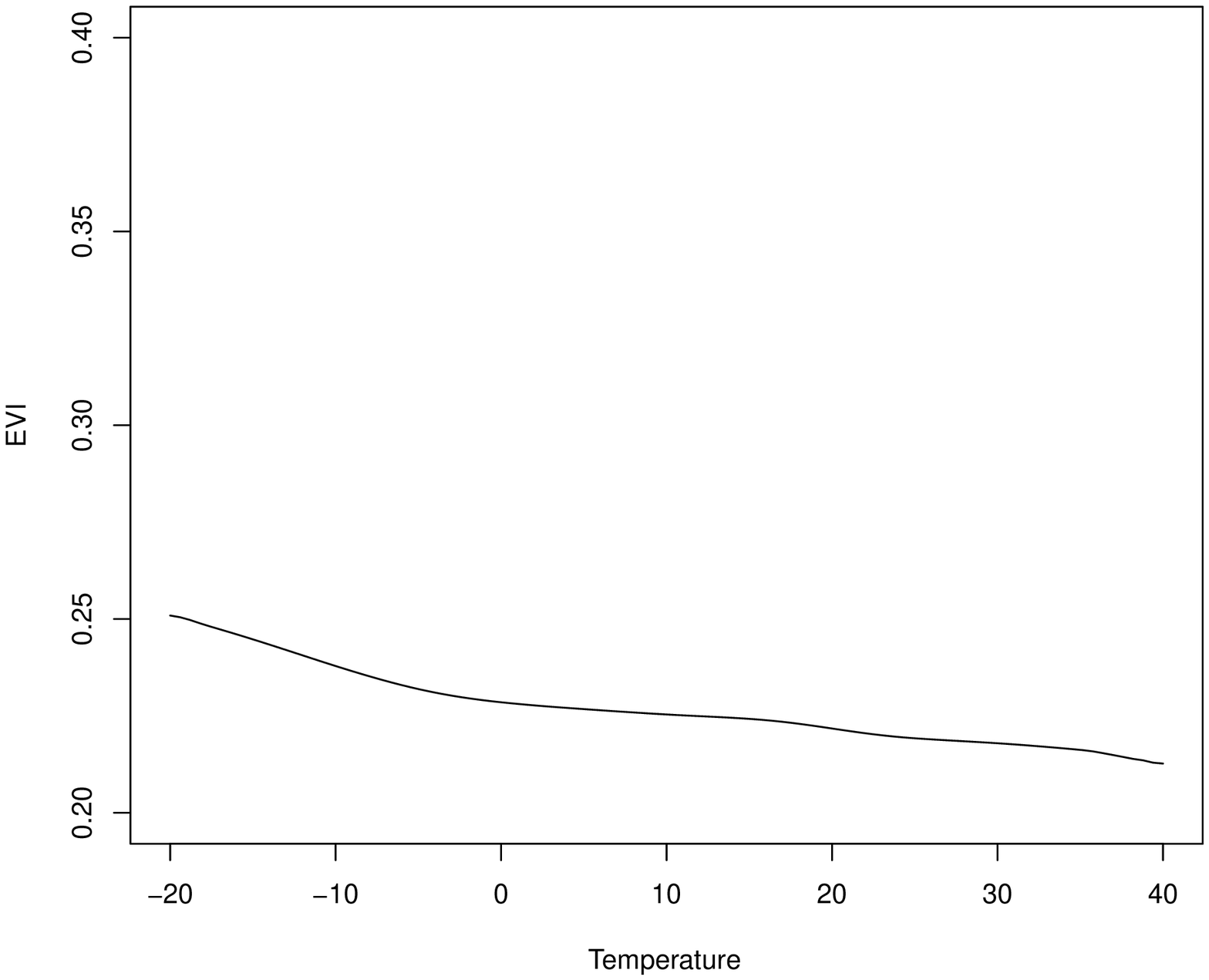}
}
\\
\subfloat[2013]{
\includegraphics[width=65mm,height=35mm]{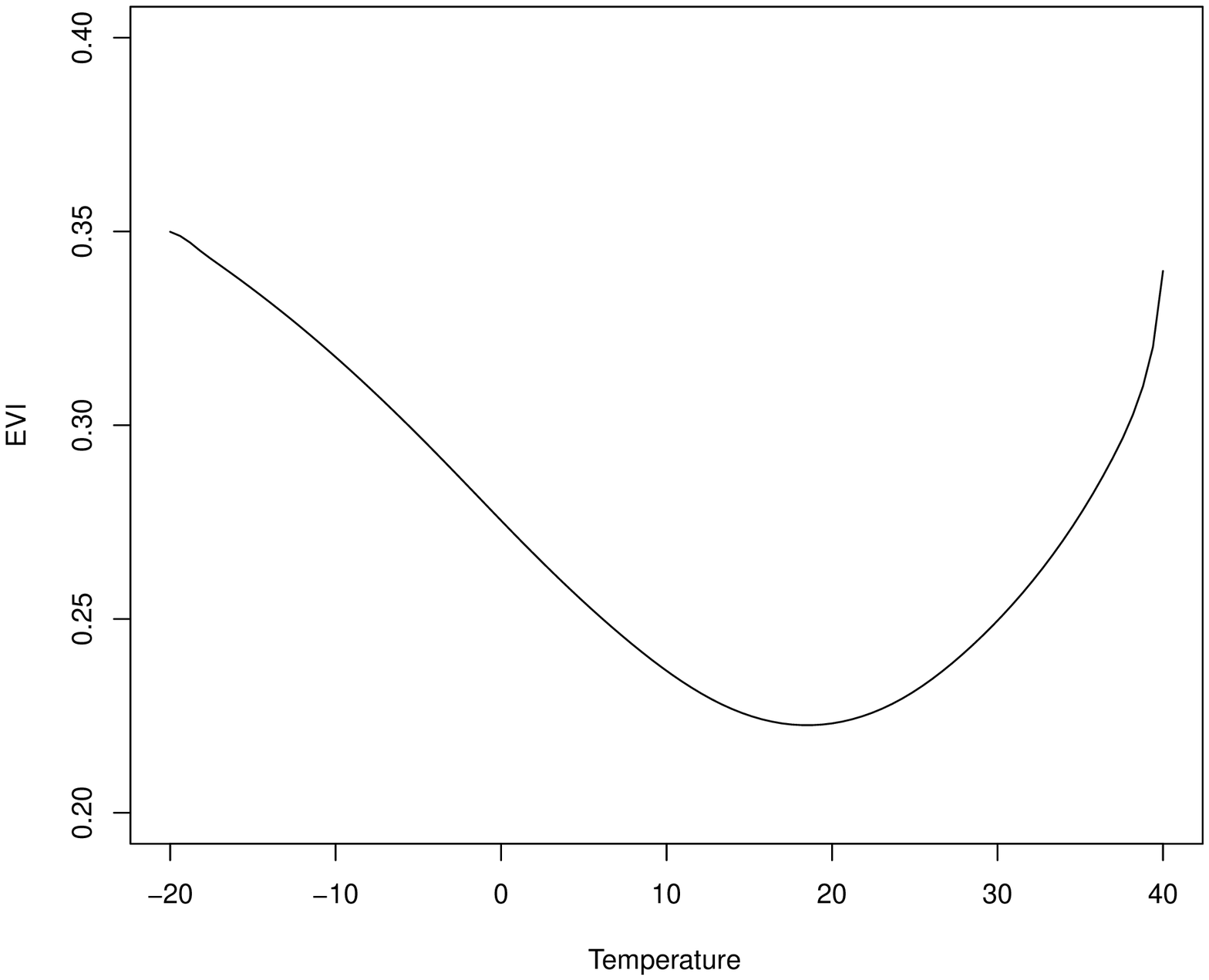}
}
\subfloat[2014]{
\includegraphics[width=65mm,height=35mm]{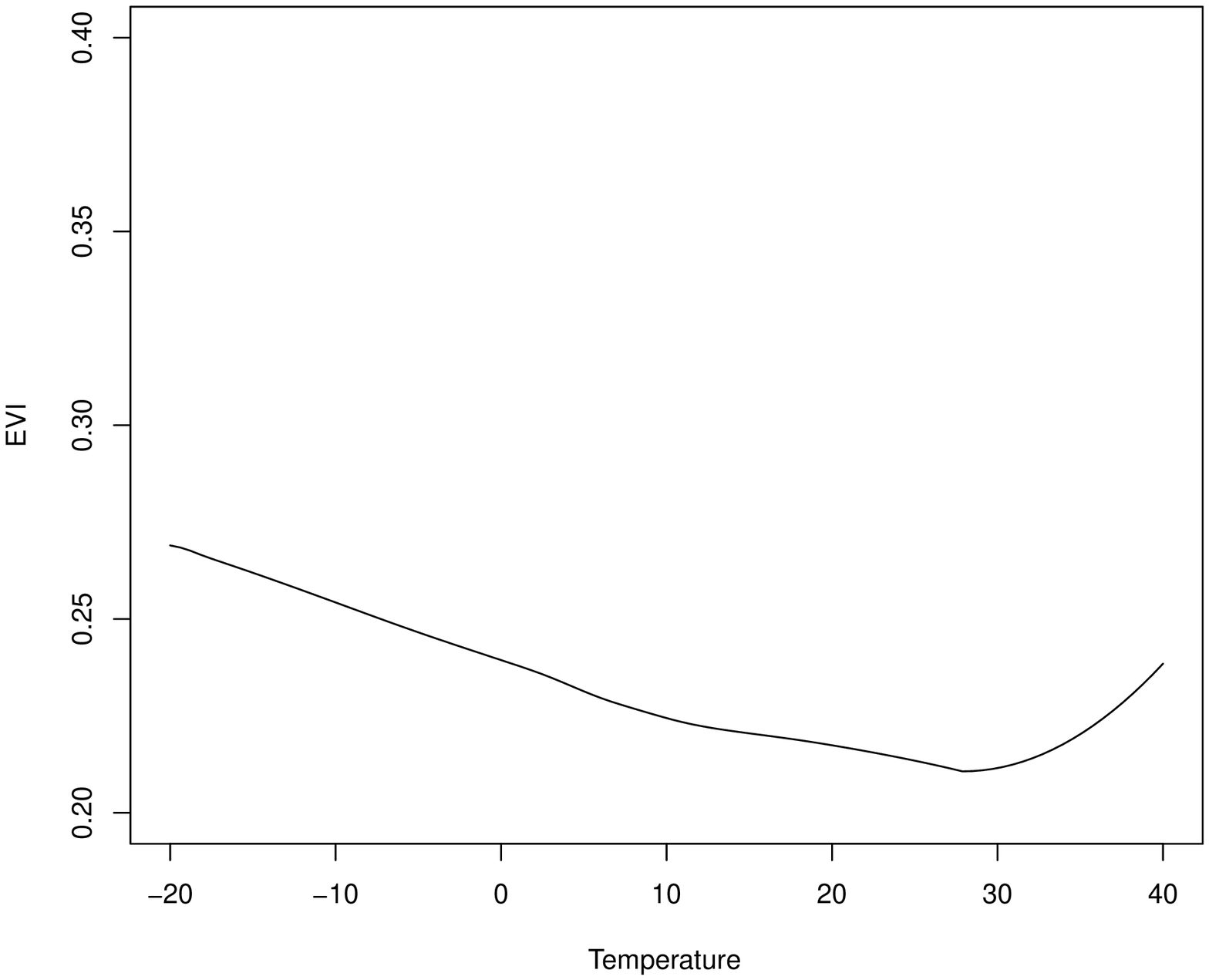}
}
\end{center}
\caption{The EVI estimates $\hat{\gamma}(x)$ with selected $k$ versus $x$ for each year.
\label{fig12}}
\end{figure}

\section{Conclusion}

We have developed the nonparametric extremal quantile regression methods for heavy-tailed data.
To show the mathematical property of the proposed estimator, we have used the hybrid techniques of asymptotic theory for the nonparametric regression and EVT for the tail behavior of the conditional distribution.
We then considered two quantile rates: (i) the intermediate order quantile that $(1-\tau)n\rightarrow \infty$ as $n\rightarrow\infty$; and (ii) the extreme order quantile that $(1-\tau)n\rightarrow \xi <\infty$. 
For the intermediate order quantile, the penalized spline estimator and its asymptotic normality have been developed. 
On the other hand, for the extremal order quantile, we have studied the Weissman-type extrapolated estimator using the intermediate order quantile estimator and its asymptotic normality. 
For the both intermediate and extreme order quantile, we show the asymptotic normality and the optimal rate of convergence of the proposed estimator. 
In particular, we found that the convergence speed of the estimator for extremal quantile is slower than that for center quantile. 
This result would be intuitively correct. 

We now discuss some future directions of study. 
First, for technical reasons, we assumed that the tail behavior of the conditional distribution of $Y$ given $X=x$ is equivalent across the predictor $x$ (see Conditions A1). 
Since the estimation of the tail behavior is difficult due to data sparsity, this assumption is helpful in data analysis. 
However, if this assumption is violated, additional research is needed to explicate the performance of the estimator. 

Second, in this paper, we focused on the spline smoothing with $\ell_2$ penalty. 
On the other hand, Koenker et al.\ (1994) and Koenker (2011) studied the smoothing spline with the $\ell_1$-type penalty. 
That is, the penalty is defined as $\int_a^b |s^{(m)}(x)|dx$ instead of $\int_a^b \{s^{(m)}(x)\}^2dx$. 
It is known that the estimator with $\ell_1$ penalty has local adaptiveness.
Therefore, for some cases, the performance of the estimator with $\ell_1$ penalty would be better than that with $\ell_2$ penalty. 
Recently, the $\ell_1$-type penalty has been rapidly developed in mean regression (Kim et al. 2009; Tibshirani 2014; Sadhanala and Tibshirani 2017). 
Although it is difficult to show the asymptotic distribution of an $\ell_1$ penalized estimator, the developments of the $\ell_1$ penalized smoothing to the extremal quantile regression is an interesting problem. 

Finally, we can consider extending the proposed method to the multidimensional case. 
In particular, it is important to use the additive models (Hastie and Tibshirani 1990) that for $x=(x_1,\ldots,x_d)\in\mathbb{R}^d$, the true function is can be decomposed as 
$$
f(x)=f_1(x_1)+\ldots+f_d(x_d),
$$
where each $f_j$ is the univariate function.
The additive model is known to enables us to avoid the problem of dimensionality. 
The nonparametric additive quantile regression (for center quantile) was studied by Lu and Yu (2004), Horowitz and Lee (2005), Cheng et al. (2011), Koenker (2011), Lee et al. (2012) and references therein.
However, the extremal inference of the additive quantile regression has not yet been studied until now.
It seems that the developments of the extremal quantile regression with the additive model is an important issue.

\subsection*{Appendix}

\subsubsection*{Appendix A: Computation of the intermediate order quantile estimator} 

We describe here the approximation algorithm to solve (\ref{nonparaquantile}). 
Nychka et al. (1995) and Reiss and Huang (2012) proposed the penalized iteratively reweighted least squares algorithm. 
We use here the modified version of Nychka et al. (1995).
Nychka et al. (1995) proposed the following optimization:
\begin{eqnarray}
\tilde{\vec{b}}=\underset{\vec{b}}{\argmin}\left\{\sum_{i=1}^n \rho_{\tau,\alpha}(y_i-\vec{B}(x_i)^T\vec{b})+\lambda \vec{b}^T D_m^TRD_m\vec{b}\right\}, \label{pirls}
\end{eqnarray}
where 
$$
\rho_{\tau,\alpha}(u)=
\left\{
\begin{array}{cc}
\rho_\tau(u), &|u|>\alpha,\\
\tau u^2/\alpha, & 0\leq u\leq \alpha,\\
(1-\tau)u^2/\alpha,& -\alpha\leq u\leq 0
\end{array}
\right.
$$
Obviously, the loss function $\rho_{\tau,\alpha}(u)$ tends to $\rho_\tau$ as $\alpha\rightarrow 0$.
However, for the tail quantile ($\tau\approx 0,1$), the above algorithm will not converge in our implementation.
Therefore, we suggest using the slightly modified version of (\ref{pirls}). 
The idea is similar to the proximal gradient method and it is very simple. 
The modified algorithm is defined as given the $t$th iteration estimate $\tilde{\vec{b}}^{(t)}$, 
\begin{eqnarray}
\tilde{\vec{b}}^{(t+1)}=\underset{\vec{b}}{\argmin}\left\{\sum_{i=1}^n \rho_{\tau,\alpha_t}(y_i-\vec{B}(x_i)^T\vec{b})+\lambda \vec{b}^T D_m^TRD_m\vec{b}+\eta_t ||\vec{b}-\tilde{\vec{b}}^{(t)}||^2\right\}, \label{pirls2}
\end{eqnarray}
where $a_t$ and $\eta_t$ are the step sizes. 
For $a_t$, it is sufficient to use some sequence $\{\alpha_t\}$ satisfying $\alpha_t\rightarrow 0$ as $t\rightarrow \infty$.
On the other hand, the choice of $\eta_t$ is more important than $a_t$.
If $\eta_t$ is large, $\vec{b}^{(t+1)}\approx \vec{b}^{(t)}$ and hence the speed of convergence is very fast. 
When $\eta_t$ is small, on the other hand, there is almost no difference between (\ref{pirls2}) and (\ref{pirls}), that is, the algorithm does not converge in many cases. 
In Sections 4 and 5, we used $\alpha_t=0.1\times 2^{-t}$ and $\eta_{t+1}=1.2 \eta_t  (\eta_0=1)$.

\subsubsection*{Appendix B: Proof of theorems}

Let $Z$ be the $n\times (K+p)$ design matrix having elements $B_j(x_i)$, $G_n=n^{-1}Z^T{\rm diag}[H(x_i)^{-\gamma}]Z$ and let $\Lambda_n=G_n+(\lambda q(\tau)/((1-\tau) n))D_m^TRD_m$. 
For any matrix $A=(a_{ij})_{ij}$, we denote $||A||_{\infty}=\max_{i,j}|a_{ij}|$.
We first state the technical lemmas to prove theorems in this paper.

\begin{lemma}\label{cla1}
Suppose that $K(m,\tau)\geq 1$. Under Conditions A--C, the following statements holds: $||G(1)||_{\infty}=O(K^{-1})$, $||G(H^{-\gamma})||_{\infty}=O(K^{-1})$, $||\Lambda_n^{-1}||_{\infty}=O(K(1+K(m,\tau)^{2m})^{-1})$,  $||\Lambda(H^{-\gamma})^{-1}||_{\infty}=O(K(1+K(m,\tau)^{2m})^{-1})$, $\max_{i,j}\{|\{G_n^{-1}-G(H^{-\gamma})^{-1}\}_{ij}|\}=o(K)$ and $\max_{i,j}\{|\{\Lambda_n^{-1}-\Lambda(H^{-\gamma})^{-1}\}_{ij}|\}=o(K(1+K(m,\tau)^{2m})^{-1})$.
\end{lemma}

\begin{lemma}\label{cla2}
Suppose that $K(m,\tau)\geq 1$. Under Assumptions 2--3, for any $K+p$ square matrix $A$ satisfying $||A||_\infty=O(n^a)$ for $a\in\mathbb{R}$, $||G(H^{-\gamma})A||_\infty=O(n^aK^{-1})$ and $||\Lambda(H^{-\gamma})^{-1}A||_\infty=O(n^a K(1+K(m,\tau)^{2m})^{-1})$.
\end{lemma}

Lemma \ref{cla1} is proved by Lemma 6.3 and 6.4 of Zhou et al. (1998) and Lemma A1 and A2 of Claeskens et al. (2009). 
Lemma \ref{cla2} says that the order of product of matrices is dependent only on the order of element of these matrices although the each element of $GA$ and $\Lambda(h)^{-1}A$ is infinite sum as $n\rightarrow \infty$. 
The proof of Lemma \ref{cla2} can also be shown by Lemma A1 of Claeskens et al. (2009). 
Therefore, we only describe the outline here.

\begin{proof}[Proof of Lemma \ref{cla2}]

For any continuous and bounded function $h$, the matrix $G(h)$ is the band matrix from property of $B$-splines and hence $||GA||_\infty=O(n^aK^{-1})$ is obvious. 
Next, $\Lambda(h)^{-1}$ is the inverse of the band matrix. 
From Demko (1977), there exists $c>0$ and $r\in(0,1)$ such that $|\{\Lambda(h)^{-1}\}_{ij}|<cr^{|i-j|}K(1+K(m,\tau)^{2m})^{-1}$. 
Thus, straightforward calculation yields that the infinite sum of each element of $\Lambda(h)^{-1}A$ is bounded by order of $\Lambda(h)^{-1}$ and the absolute of the maximum of element of $A$. 

\end{proof}

\begin{proof}[Proof of Theorem \ref{biasvariance1}]
Define $h(x)=H(x)^{-\gamma}$.
We write $\tau\equiv \tau_I$ and hence $\tau\rightarrow 1$ and $(1-\tau)n\rightarrow \infty$ as $n\rightarrow \infty$.
By the fundamental property of $B$-spline basis, $s_0^{(m)}(\tau|x)=d^m s_0(\tau|x)/dx^m$ can be written as $s_0^{(m)}(\tau|x)= \vec{B}^{[p-m]}(x)^T D_m \vec{b}_0(\tau)$, where $\vec{B}^{[p-m]}(x)$ is the vector having element $\{B_1^{[p-m]}(x),\ldots,B_{K+p-m}^{[p-m]}(x)\}$ and $B_k^{[p-m]}(x)$'s are $(p-m)$th degree $B$-spline bases.
Therefore, the shrinkage bias can be expressed as 
\begin{eqnarray*}
b_\lambda(\tau|x)=\frac{\lambda q(\tau)}{(1-\tau) n}\vec{B}(x)^T\Lambda(H^{-\gamma})^{-1}D_m^T\int_a^b \vec{B}^{[p-m]}(x)s_0^{(m)}(\tau|x)dx. 
\end{eqnarray*}
From Conditions A--B, we have $s_0(\tau|x)=q_Y(\tau|x)(1+o(1))=h(x)q(\tau)(1+o(1))$ and $s_0^{(m)}(\tau|x)=h^{(m)}(x)q(\tau)(1+o(1))$ as $\tau\rightarrow 1$. 
Since $h$ is bounded function, we get $\sup_{x\in[a,b]}|s_0^{(m)}(\tau|x)|=O(q(\tau))$. 
From the property of $B$-spline basis, we also obtain $\int_a^b B_k(x)ds=O(K^{-1})$. 
Thus, each element of $\int_a^b \vec{B}^{[p-m]}(x)s_0^{(m)}(\tau|x)dx$ has the order $O(K^{-1}q(\tau))$ 
Furthermore, the result of Cardot (2000) provides $||D_m||_\infty=O(K^{m})$. 
Therefore, Lemmas \ref{cla1}, \ref{cla2} and the fact that $K(m,\tau)^m/(1+K(m,\tau)^{2m})<1/2$ yield that 
\begin{eqnarray*}
b_\lambda(\tau|x)
&=&
O\left(q(\tau)\frac{\lambda q(\tau)K^{m}}{(1-\tau) n}(1+K(m,\tau)^{2m})^{-1}\right)\\
&=&
O\left(q(\tau)\left(\frac{\lambda q(\tau)}{(1-\tau) n}\right)^{1/2}\frac{K(m,\tau)^{m}}{1+K(m,\tau)^{2m}}\right)\\
&\leq &
O\left(q(\tau)\left(\frac{\lambda q(\tau)}{(1-\tau) n}\right)^{1/2}\right).
\end{eqnarray*}
Next, we show the asymptotic order of $v(x|\tau)$. 
Since $||G||_\infty=O(K^{-1})$, from Lemmas \ref{cla1}--\ref{cla2}, we have
\begin{eqnarray*}
v(\tau|x)&=&O\left(\frac{K q(\tau)^2}{(1-\tau)n}\{1+K(m,\tau)^{2m}\}^{-2} \right)\\
&=&O\left(\frac{q(\tau)}{(1-\tau)n}\left(\frac{\lambda q(\tau)}{(1-\tau) n}\right)^{-1/2m}K(m,\tau)\{1+K(m,\tau)^{2m}\} \right)\\
&\leq &O\left(\frac{q(\tau)}{(1-\tau)n}\left(\frac{\lambda q(\tau)}{(1-\tau) n}\right)^{-1/2m}\right),
\end{eqnarray*}
which completes the proof.
\end{proof}

\begin{proof}[Proof of Theorem \ref{as.dist}] 
We write $\tau\equiv\tau_I$ and hence $\tau\rightarrow 1$ and $(1-\tau)n\rightarrow \infty$ as $n\rightarrow \infty$.
Let $U_i=Y_i-\vec{B}(x)^T\vec{b}_0(\tau)$, $a_n=\sqrt{(1-\tau) n}/q(\tau)$ and  
\begin{eqnarray*}
Q_n(\vec{\delta}|\tau)&=&\frac{a_n}{\sqrt{n(1-\tau)}} \sum_{i=1}^n \left\{\rho_\tau(U_i-\vec{B}(x_i)^T\vec{\delta}/a_n)-\rho_\tau(U_i)\right\}\\
&&+\frac{a_n\lambda}{2\sqrt{n(1-\tau)}} (\vec{b}_0(\tau)+\vec{\delta}/a_n)^TD_m^TRD_m(\vec{b}_0(\tau)+\vec{\delta}/a_n).
\end{eqnarray*}
Then the minimizer of $Q_n$ is obtained as 
$$
\tilde{\vec{\delta}}=a_n(\tilde{\vec{b}}(\tau)-\vec{b}_0(\tau)).
$$
Using Knight's identity (Knight, 1998), we have 
\begin{eqnarray*}
\rho_\tau(u-v)-\rho_\tau(u)=-v(\tau-I(u<0))+\int_0^v \{I(u\leq s)-I(u\leq 0)\}ds
\end{eqnarray*}
and writing
\begin{eqnarray*}
Q_n(\vec{\delta}|\tau)=W_n(\tau)^T\vec{\delta}+G_n(\vec{\delta}|\tau)+\frac{a_n\lambda}{2\sqrt{(1-\tau) n}}(\vec{b}_0(\tau)+\vec{\delta}/a_n)^TD_m^TRD_m(\vec{b}_0(\tau)+\vec{\delta}/a_n),
\end{eqnarray*}
where 
\begin{eqnarray*}
W_n(\tau)&\equiv&\frac{-1}{\sqrt{(1-\tau) n}}\sum_{i=1}^n (\tau-I(Y_i<\vec{B}(x_i)^T\vec{b}_0))\vec{B}(x_i),\\
G_n(\vec{\delta}|\tau)&\equiv&\sum_{i=1}^n \int_0^{\vec{B}(x)^T\vec{\delta}/a_n}I(U_i\leq  s)-I(U_i\leq 0)ds.
\end{eqnarray*}
Since $\tau=P(Y<q_Y(\tau|x)|X=x)=P(Y<\vec{B}(x)^T\vec{b}_0(\tau)|X=x)+o(K^{-m})$ and $E[I(Y<\vec{B}(x)^T\vec{b}_0(\tau))]=P(Y<q_Y(\tau|x)|X=x)+o(K^{-m}(1-\tau)^{-\gamma})$, we obtain $E[W_n(\tau)]=o(1)$. 
The variance of $W_n(\tau)$ can be evaluated as 
\begin{eqnarray*}
V[W_n(\tau)^T\vec{\delta}]
=
\frac{\tau(1-\tau)}{1-\tau}\vec{\delta}^T\left(\frac{1}{n}Z^TZ\right)\vec{\delta}
\stackrel{P}{\longrightarrow }
\vec{\delta}^TG\vec{\delta}
\end{eqnarray*}
as $n\rightarrow \infty$ and $\tau\rightarrow 1$.
Lyapnov's condition for the central limit theorem and Cram${\rm \grave{e}}$r-Wold device yield that 
$W_n(\tau)$ is asymptotically distributed as $\vec{W}$, which is the normal with mean $\vec{0}$ and variance $G$. 

Next, we show that as $n\rightarrow \infty$ and $\tau\rightarrow 1$, 
\begin{eqnarray}
G_n(\vec{\delta}|\tau)\stackrel{P}{\longrightarrow }\frac{1}{2}\gamma^{-1}\vec{\delta}^T G(H^{-\gamma}) \vec{\delta}. \label{varproof1}
\end{eqnarray}
Before that, we provide some differential results.
Let $f_Y(y)$ and $f_Y(y|x)$ be the marginal density of $Y$ and conditional density of $Y$ given $X=x$, respectively.
From A3, $f_Y(y|x)=f_V(y|x)(1+o(1))=H(x)f_Y(y)(1+o(1))$ as $y\rightarrow \infty$. 
In addition, since $1=\partial F_V(q_V(\tau|x)|x)/\partial \tau=f_V(q_V(\tau|x)|x) \partial q_V(\tau|x)/\partial \tau$, we have $f_V(q_V(\tau|x)|x)=\{\partial q_V(\tau|x)/\partial \tau\}^{-1}$. 
Meanwhile, A4 and $q((1-\tau)/H(x))=\{(1-\tau)/H(x)\}^{-\gamma}L(H(x)/(1-\tau))(1+o(1))$ yield that $\partial q((1-\tau)/H(x))/\partial \tau =\gamma (1-\tau)^{-\gamma-1}H(x)^{\gamma}L(H(x)/(1-\tau))(1+o(1))$.
Consequently, as $\tau\rightarrow 1$, 
$$
f_Y(q_Y(\tau|x)|x) =\gamma^{-1}H(x)^{-\gamma} (1-\tau)^{\gamma+1}L(H(x)/(1-\tau))^{-1}(1+o(1)).
$$
Furthermore, by $L\in RV(0)$, $q(\tau)f_Y(q_Y(\tau|x)|x)/\{1-\tau\}=\gamma^{-1}H(x)^{-\gamma}(1+o(1)). $

We return to show (\ref{varproof1}).
Since 
\begin{eqnarray*}
G_n(\vec{\delta}|\tau)=\frac{1}{\sqrt{(1-\tau) n}}\sum_{i=1}^n \left(\int_0^{\vec{\delta}^T\vec{B}(x_i)}I(U_i\leq  s/a_n)-I(U_i\leq 0)ds\right),
\end{eqnarray*}
we obtain 
\begin{eqnarray*}
E[G_n(\vec{\delta}|\tau)]
&=&
\frac{1}{\sqrt{(1-\tau) n}}\sum_{i=1}^nE\left[\int_0^{\vec{\delta}^T\vec{B}(x_i)}I(U_i\leq  s/a_n)-I(U_i\leq 0)ds\right]\\
&=&
\frac{1}{\sqrt{(1-\tau) n}}\sum_{i=1}^n\left[\int_0^{\vec{\delta}^T\vec{B}(x_i)}F_Y(q_Y(\tau|x_i)+s/a_n|x_i)-F_Y(q_Y(\tau|x_i)|x_i)ds\right]\\
&=&
\sum_{i=1}^n\int_0^{\vec{\delta}^T\vec{B}(x_i)}\frac{f_Y(q_Y(\tau|x_i))}{a_n\sqrt{(1-\tau) n}}sds(1+o(1))\\
&=&
\frac{1}{n}\sum_{i=1}^n\int_0^{\vec{\delta}^T\vec{B}(x_i)}\frac{q(\tau)f_Y(q_Y(\tau|x_i))}{(1-\tau)}sds(1+o(1))\\
&=&2^{-1}\gamma^{-1}\vec{\delta}^T \left(\frac{1}{n}\sum_{i=1}^n H(x_i)^{-\gamma}\vec{B}(x_i)\vec{B}(x_i)^T\right)\vec{\delta}.
\end{eqnarray*}
From the simple but tedious calculation,
$V[G_n(\vec{\delta}|\tau)]=o(1)$ can be evaluated. 
These results yield that
\begin{eqnarray*}
E[G_n(\vec{\delta}|\tau)]\stackrel{P}{\longrightarrow }
\frac{1}{2}\gamma^{-1}\vec{\delta}^T G(H^{-\gamma})\vec{\delta}.
\end{eqnarray*}

Thus, $Q_n(\vec{\delta}|\tau)$ is asymptotically equivalent to 
\begin{eqnarray*}
Q_0(\vec{\delta}|\tau)&=&\left(W^T+\frac{\lambda}{a_n\sqrt{(1-\tau) n}}\vec{b}_0(\tau)^TD_m^TRD_m\right)\vec{\delta}\\
&&+\frac{1}{2}\vec{\delta}^T\left(\gamma^{-1} G(H^{-\gamma})+\frac{\lambda}{2a_n\sqrt{(1-\tau) n}}D_m^TRD_m\right)\vec{\delta}.
\end{eqnarray*}
By the convexity lemma (see, Pollard, 1991 and Knight, 1998), the minimizer of $Q_n$ and $Q_0$ are asymptotically equivalent and hence we have
\begin{eqnarray*}
\tilde{\vec{\delta}}=\left(\gamma^{-1} G(H^{-\gamma})+\frac{\lambda}{2a_n\sqrt{(1-\tau) n}}D_m^TRD_m\right)^{-1}\left(W-\frac{\lambda}{\sqrt{a_n(1-\tau) n}}D_m^TRD_m\vec{b}_0(\tau)\right)+o_P(1).
\end{eqnarray*}
Since $\tilde{q}_Y(\tau|x)-s_0(\tau|x)=\vec{B}(x)^T(\tilde{\vec{b}}(\tau)-\vec{b}_0(\tau))$, we obtain from $a_n=\sqrt{(1-\tau) n}/q(\tau)$ that
\begin{eqnarray}
&&\frac{\sqrt{(1-\tau) n}}{q(\tau)} (\tilde{q}_Y(\tau|x)-s_0(\tau|x))\nonumber\\
&&=\vec{B}(x)^T\left(\gamma^{-1} G(H^{-\gamma})+\frac{\lambda q(\tau)}{2(1-\tau) n}D_m^TRD_m\right)^{-1}W\nonumber\\
&&\quad -\frac{\lambda}{\sqrt{(1-\tau) n}}\vec{B}(x)^T\left(\gamma^{-1} G(H^{-\gamma})+\frac{\lambda q(\tau)}{2(1-\tau) n}D_m^TRD_m\right)^{-1}D_m^TRD_m\vec{b}_0(\tau)\nonumber\\
&&\quad\quad  +o_P(1). \label{as.sp1}
\end{eqnarray}
The second term of right hand side of (\ref{as.sp1}) is the shrinkage bias. 
Consequently, as $n\rightarrow \infty$, 
$$
\frac{\sqrt{(1-\tau) n}}{q(\tau)} \frac{\tilde{q}_Y(\tau|x)-q_Y(\tau|x)-b_\lambda(\tau|x)}{\sqrt{\vec{B}(x)^T\Lambda(H^{-\gamma})^{-1}G \Lambda(H^{-\gamma})^{-1}\vec{B}(x)}} \stackrel{D}{\longrightarrow } N(0,1).
$$

Finally, we obtain 
\begin{eqnarray}
E[\{\hat{q}_Y(\tau|x)-q_Y(\tau|x)\}^2]&=&b_\lambda(\tau|x)^2+v(\tau|x)\nonumber\\
&=&
O\left(q(\tau)^2\frac{\lambda q(\tau)}{(1-\tau) n}\right)+O\left(\frac{q(\tau)^2}{(1-\tau)n}\left(\frac{\lambda q(\tau)}{(1-\tau) n}\right)^{-1/2m}\right). \label{MISEinter}
\end{eqnarray}
We now derive the optimal rate of convergence of MISE of $\tilde{q}_Y(\tau|x)$. 
For the constant $C_1>0, C_2>0$, the solution of 
\begin{eqnarray*}
C_1q(\tau)\lambda-C_2 \left(\frac{ q(\tau)}{(1-\tau) n}\right)^{-1/2m}\lambda^{-1/2m}=0
\end{eqnarray*}
is $\lambda=C q(\tau_I)^{-1}\{n(1-\tau_I)\}^{1/(2m+1)}$ for $C>0$. 
By applying this $\lambda$ in (\ref{MISEinter}), we obtain
$$
E\left[\left\{\frac{\hat{q}_Y(\tau|x)}{q_Y(\tau|x)}-1\right\}^2\right]=O(\{(1-\tau)n\}^{-2m/(2m+1)}),
$$
which completes the proof.
\end{proof}

To improve the outlook, we now describe about the asymptotic bias and variance of $\hat{\gamma}(x)$ before prove Theorem 3. 
Define
\begin{eqnarray}
b(k|x)=\frac{1}{k-1}\sum_{j=1}^{k-1} \frac{b_\lambda(\tau_j|x)}{q_Y(\tau_j|x)}-\frac{b_\lambda(\tau_k|x)}{q_Y(\tau_k|x)},\ \ \vec{v}(k|x)=\frac{1}{k-1}\sum_{j=1}^{k-1} \frac{\vec{\nu}(\tau_j|x)}{q_Y(\tau_j|x)}-\frac{\vec{\nu}(\tau_k|x)}{q_Y(\tau_k|x)}, \label{asbiasvarianceevi}
\end{eqnarray}
where 
$$
\vec{\nu}(\tau|x)=\frac{q(\tau)}{\sqrt{(1-\tau)n}}G^{1/2}\Lambda(H^{-\gamma})^{-1}\vec{B}(x).
$$
As the result, $b(k|x)$ and $\sqrt{\vec{v}(k|x)^T\vec{v}(k|x)}$ are the asymptotic bias and standard deviation of $\hat{\gamma}(x)$. 
We here get the asymptotic order of $b(k|x)$ and $v(k|x)$ from easy calculation.

Since $b_\lambda(\tau_j|x)/q_Y(\tau_j|x)=O(\{(1-\tau_j)n\}^{-m/(2m+1)})$ and each element of $\vec{\nu}(\tau_j|x)$ has $O(\{(1-\tau_j)n\}^{-m/(2m+1)})$, we have
\begin{eqnarray*}
b(k|x)=O\left(\frac{1}{k-1}\sum_{j=1}^{k-1} \{(1-\tau_j)n\}^{-\frac{m}{2m+1}}-\{(1-\tau_k)n\}^{-\frac{m}{2m+1}}\right)
\end{eqnarray*}
and 
\begin{eqnarray*}
\vec{v}(k|x)=O\left(\frac{1}{k-1}\sum_{j=1}^{k-1} \{(1-\tau_j)n\}^{-\frac{m}{2m+1}}-\{(1-\tau_k)n\}^{-\frac{m}{2m+1}}\right).
\end{eqnarray*}
We then have from $[n^\eta]/k\rightarrow 0 (n,k\rightarrow \infty)$ that 
\begin{eqnarray*}
\frac{1}{k}\sum_{j=1}^k \{(1-\tau_j)n\}^{-\frac{m}{2m+1}}
&=&
\frac{1}{k}\sum_{j=1}^k \left\{\frac{[n^\eta]+j}{n+1}n\right\}^{-\frac{m}{2m+1}}\\
&=&
k^{-\frac{m}{2m+1}} \frac{1}{k}\sum_{j=1}^k \left\{\frac{[n^\eta]+j}{k+1}\right\}^{-\frac{m}{2m+1}}\\
&=&k^{-\frac{m}{2m+1}} \int_0^1 u^{-\frac{m}{2m+1}}du(1+o(1))\\
&=&\frac{2m+1}{m}k^{-\frac{m}{2m+1}}(1+o(1))
\end{eqnarray*}
and 
\begin{eqnarray*}
\{(1-\tau_k)n\}^{-\frac{m}{2m+1}}
=\left\{\frac{[n^\eta]+k}{n+1}n\right\}^{-\frac{m}{2m+1}}
=k^{-\frac{m}{2m+1}} (1+o(1)).
\end{eqnarray*}
This indicates that 
$b(k|x)=O(k^{-\frac{m}{2m+1}})$
and $\vec{v}(k|x)=O(k^{-\frac{m}{2m+1}})$.

\begin{proof}[Proof of Theorem 3] 
Theorem 1 indicates that $\tilde{q}_Y(\tau_k|x)-q_Y(\tau_k|x)=o_P(1)$ as $n\rightarrow \infty$.
Therefore, the proposed estimator can be calculated as for $x\in\mathbb{R}$,
\begin{eqnarray*}
\hat{\gamma}(x)
&=&
\frac{1}{k-1}\sum_{j=1}^{k-1} \log \frac{\hat{q}_Y(\tau_j|x)}{\hat{q}_Y(\tau_k|x)}\\
&=&
\frac{1}{k-1}\sum_{j=1}^{k-1} \log \frac{q_Y(\tau_j|x)\left\{1+\frac{\hat{q}_Y(\tau_j|x)-q_Y(\tau_j|x)}{q_Y(\tau_j|x)}\right\}}{q_Y(\tau_k|x)\left\{1+\frac{\hat{q}_Y(\tau_k|x)-q_Y(\tau_k|x)}{q_Y(\tau_k|x)}\right\}}\\
&=&
\frac{1}{k-1}\sum_{j=1}^{k-1} \log \frac{q_Y(\tau_j|x)}{q_Y(\tau_k|x)}+\frac{1}{k-1}\sum_{j=1}^{k-1} \log \frac{\left\{1+\frac{\hat{q}_Y(\tau_j|x)-q_Y(\tau_j|x)}{q_Y(\tau_j|x)}\right\}}{\left\{1+\frac{\hat{q}_Y(\tau_k|x)-q_Y(\tau_k|x)}{q_Y(\tau_k|x)}\right\}}\\
&=&
\frac{1}{k-1}\sum_{j=1}^{k-1} \log \frac{q_Y(\tau_j|x)}{q_Y(\tau_k|x)}+\frac{1}{k-1}\sum_{j=1}^{k-1} \log \left\{1+\frac{\hat{q}_Y(\tau_j|x)-q_Y(\tau_j|x)}{q_Y(\tau_j|x)}\right\}\\
&&-\frac{1}{k-1}\sum_{j=1}^{k-1} \log 
\left\{1+\frac{\hat{q}_Y(\tau_k|x)-q_Y(\tau_k|x)}{q_Y(\tau_k|x)}\right\}\\
&=&
\frac{1}{k-1}\sum_{j=1}^{k-1} \log \frac{q_Y(\tau_j|x)}{q_Y(\tau_k|x)}\\
&&+\frac{1}{k-1}\sum_{j=1}^{k-1} \left\{\frac{\hat{q}_Y(\tau_j|x)-q_Y(\tau_j|x)}{q_Y(\tau_j|x)}- \frac{\hat{q}_Y(\tau_k|x)-q_Y(\tau_k|x)}{q_Y(\tau_k|x)}\right\}(1+o_P(1))\\
&\equiv& D_{1n}+D_{2n}. 
\end{eqnarray*}
We then note that $D_{1n}$ is not random variable. 
Similar to the proof of Theorem 2.3 of Wang et al. (2012), we have as $k\rightarrow \infty$,
$$
D_{1n}=\gamma+O(k^{-1/2})=\gamma+o(k^{-m/(2m+1)})
$$
for $m\geq 1$.
Next, we consider $D_{2n}$. 
Under the conditions for Theorem 3, using the result of Theorems 1--2 and the property of $B$-spline basis, we have 
\begin{eqnarray*}
\frac{\tilde{q}_Y(\tau_j|x)-q_Y(\tau_j|x)}{q_Y(\tau_j|x)}
&=&
\frac{b_\lambda(\tau_j|x)+\vec{\nu}(\tau_j|x)^T \vec{W}}{q_Y(\tau_j|x)}(1+o_P(1)),
\end{eqnarray*}
where $\vec{W}\sim N_{K+p}(\vec{0},I)$. 
Therefore, $D_{2n}$ can be evaluated as 
\begin{eqnarray*}
D_{2n}&=& \{b(k|x)+ \vec{v}(k|x)^T \vec{W}\}(1+o_P(1))
\end{eqnarray*}
That is, 
$$
\hat{\gamma}(x)=\gamma+ b(k|x)+ \vec{v}(k|x)^T \vec{W} + o(k^{-m/(2m+1)}),
$$
where $b(k|x)=O(k^{-m/(2m+1)})$ and $\vec{v}(k|x)=O(k^{-m/(2m+1)})$.
Consequently, we get
$$
\frac{\hat{\gamma}(x)-\gamma-b(k|x)}{\sqrt{\vec{v}(k|x)^T\vec{v}(k|x)}}\stackrel{D}{\longrightarrow } N(0,1)
$$
and $E[\{\hat{\gamma}(x)-\gamma\}^2]=O(k^{-2m/(2m+1)})$. 
For the common index version $\hat{\gamma}^{C}$, similar to above, the straightforward calculation yields that 
$$
\hat{\gamma}^{C}=\frac{1}{n}\sum_{i=1}^n \hat{\gamma}(x_i)=\gamma+ E[b(k|X)]+ E[\vec{v}(k|X)]^T \vec{W} + o(k^{-m/(2m+1)}).
$$ 
This completes the proof. 
\end{proof}

\begin{proof}[Proof of Theorem 4]
First, the second order condition for $U_Y(1/(1-\tau)|x)=q_Y(\tau|x)$ yields that 
$$
\frac{q_Y(\tau_I|x)}{q_Y(\tau_E|x)}=\left(\frac{U_Y(1/(1-\tau_E)|x)}{U_Y(1/(1-\tau_I)|x)}\right)^{-1}=\left(\frac{1-\tau_I}{1-\tau_E}\right)^{-\gamma}\{1+o(k^{-m/(2m+1)})\}.
$$
Furthermore, the result of Theorem 3 indicates that 
\begin{eqnarray*}
\left(\frac{1-\tau_I}{1-\tau_E}\right)^{\hat{\gamma}(x)-\gamma}
&=&
\exp\left[(\hat{\gamma}(x)-\gamma)\log\left(\frac{1-\tau_I}{1-\tau_E}\right)\right]\\
&=&1 +(\hat{\gamma}(x)-\gamma)\log\left(\frac{1-\tau_I}{1-\tau_E}\right)(1+o_P(1))\\
&=&1 +\{b(k|x)+\vec{v}(k|x)^T\vec{W}\}\log\left(\frac{1-\tau_I}{1-\tau_E}\right)(1+o_P(1)).
\end{eqnarray*}
Meanwhile, we obtain
$$
\frac{\hat{q}_Y(\tau_I|x)}{q_Y(\tau_I|x)}=1+\frac{b_\lambda(\tau_I|x)}{q_Y(\tau_I|x)}+\frac{\vec{\nu}(\tau_I|x)^T}{q_Y(\tau_I|x)}\vec{W}+o_P(k^{-m/(2m+1)}),
$$
where $\vec{W}$ is that given in the proof of Theorem 3.
Using above, we have 
\begin{eqnarray*}
\frac{\hat{q}_Y(\tau_E|x)}{q_Y(\tau_E|x)}
&=&
\left(\frac{1-\tau_I}{1-\tau_E}\right)^{\hat{\gamma}(x)}\frac{\hat{q}_Y(\tau_I|x)}{q_Y(\tau_E|x)}\\
&=&
\left(\frac{1-\tau_I}{1-\tau_E}\right)^{\hat{\gamma}(x)}\frac{\hat{q}_Y(\tau_I|x)}{q_Y(\tau_I|x)}\frac{q_Y(\tau_I|x)}{q_Y(\tau_E|x)}\\
&=&
\left(\frac{1-\tau_I}{1-\tau_E}\right)^{\hat{\gamma}(x)}\frac{\hat{q}_Y(\tau_I|x)}{q_Y(\tau_I|x)}(1+o(k^{-m/(2m+1)}))\\
&=&
\left\{1 +(b(k|x)+\vec{v}(k|x)^T)\vec{W})\log\left(\frac{1-\tau_I}{1-\tau_E}\right)(1+o_P(1))\right\}\\
&&\times\left[1+\left\{\frac{b_\lambda(\tau_I|x)}{q_Y(\tau_I|x)}+\frac{\vec{\nu}(\tau_I|x)^T}{q_Y(\tau_I|x)}\vec{W}\right\}(1+o_P(1))\right](1+o(k^{-m/(2m+1)}))\\
&=&
1+
\left\{\log\left(\frac{1-\tau_I}{1-\tau_E}\right)\vec{v}(k|x)+\frac{\vec{\nu}(\tau_I|x)}{q_Y(\tau_I|x)}\right\}^T\vec{W}\\
&&
+b(k|x)\log\left(\frac{1-\tau_I}{1-\tau_E}\right)+\frac{b_\lambda(\tau_I|x)}{q_Y(\tau_I|x)}\\
&&+o\left(k^{-m/(2m+1)}\log\left(\frac{1-\tau_I}{1-\tau_E}\right)\right)+o_P(\{n(1-\tau_I)\}^{-m/(2m+1)}).
\end{eqnarray*}
Consequently, we obtain 
$$
\frac{\frac{\hat{q}_Y(\tau_E|x)}{q_Y(\tau_E|x)}-1-bias(\tau_E|x)}{s(\tau_E|x)}\stackrel{D}{\longrightarrow } N(0,1),
$$
where 
\begin{eqnarray}
bias(\tau_E|x)=b(k|x)\log\left(\frac{1-\tau_I}{1-\tau_E}\right)+\frac{b_\lambda(\tau_I|x)}{q_Y(\tau_I|x)} \label{biasExtrapolation}
\end{eqnarray}
and 
\begin{eqnarray}
s(\tau_E|x)= \left|\left|\log\left(\frac{1-\tau_I}{1-\tau_E}\right)\vec{v}(k|x)+q_Y(\tau_I|x)^{-1}\vec{\nu}(\tau_I|x)\right|\right|. \label{sdExtrapolation}
\end{eqnarray}
Here, for a vector $\vec{a}$, $||\vec{a}||$ means the $\ell_2$-norm of $\vec{a}$.
Furthermore, we get
$$
E\left[\left\{\frac{\hat{q}_Y(\tau_E|x)}{q_Y(\tau_E|x)}-1\right\}^2\right]=O\left(\max\left\{k^{-\frac{2m}{2m+1}}\log^2\left(\frac{1-\tau_I}{1-\tau_E}\right),\{(1-\tau_I)n\}^{-\frac{2m}{2m+1}}\right\}\right).
$$
Similarly, for the common index estimator $\hat{q}_Y^{C}(\tau_E|x)$, we have 
\begin{eqnarray*}
\frac{\hat{q}^{C}_Y(\tau_E|x)}{q_Y(\tau_E|x)}&=&1 +E[b(k|X)]\log\left(\frac{1-\tau_I}{1-\tau_E}\right)+\frac{b_\lambda(\tau_I|x)}{q_Y(\tau_I|x)}\\
&&+\left\{E[\vec{v}(k|X)]\log\left(\frac{1-\tau_I}{1-\tau_E}\right)+\frac{\vec{\nu}(\tau_I|x)}{q_Y(\tau_I|x)}\right\}^{T}\vec{W}\\
&&\quad +o_P\left(k^{-m/(2m+1)}\log\left(\frac{1-\tau_I}{1-\tau_E}\right)\right)+o_P(\{(1-\tau_I)n\}^{-m/(2m+1)}).
\end{eqnarray*} 
Accordingly, 
$$
\frac{\frac{\hat{q}^{C}_Y(\tau_E|x)}{q_Y(\tau_E|x)}-1-bias^C(x)}{s^C(\tau_E|x)}\stackrel{D}{\longrightarrow } N(0,1),
$$
where 
\begin{eqnarray}
bias^C(\tau_E|x)=E[b(k|X)]\log\left(\frac{1-\tau_I}{1-\tau_E}\right)+\frac{b_\lambda(\tau_I|x)}{q_Y(\tau_I|x)}
\label{biasExtrapolationC}
\end{eqnarray}
and 
\begin{eqnarray}
s^C(\tau_E|x)=\left|\left|E[\vec{v}(k|X)]\log\left(\frac{1-\tau_I}{1-\tau_E}\right)+\frac{\vec{\nu}(\tau_I|x)}{q_Y(\tau_I|x)}\right|\right|.
\label{sdExtrapolationC}
\end{eqnarray}
Finally, we obtain the optimal rate of convergence of MISE of the common index estimator as 
$$
E\left[\left\{\frac{\hat{q}^{C}_Y(\tau_E|x)}{q_Y(\tau_E|x)}-1\right\}^2\right]=O\left(\max\left\{k^{-\frac{2m}{2m+1}}\log^2\left(\frac{1-\tau_I}{1-\tau_E}\right),\{(1-\tau_I)n\}^{-\frac{2m}{2m+1}}\right\}\right).
$$
\end{proof}

\noindent{\bf Acknowledgements}

The authors are grateful to Associate Editor and anonymous referees for their valuable comments and suggestions, which led to improvements of the paper. 
The research of the author was partially supported by KAKENHI 18K18011.

\def\bibname{Reference}

\end{document}